\documentclass[11pt]{amsart}
\usepackage{fullpage}
\usepackage{amsthm, amsmath, amssymb}
\newtheorem{thm}{Theorem}
\newtheorem{lem}[thm]{Lemma}

\newtheorem{ex}[thm]{Example}

\newtheorem{rem}[thm]{Remark}

\newtheorem{conj}[thm]{Conjecture}
\newtheorem{definition}[thm]{Definition}
\usepackage{graphicx}
\usepackage{tikz}
\usepackage[foot]{amsaddr}
\usetikzlibrary{decorations.markings}

\def\Singlearrow{{\arrow[scale=1.5,xshift={0.5pt+2.25\pgflinewidth}]{>}}}
\def\Doublearrow{{\arrow[scale=1.5,xshift=1.35pt +2.47\pgflinewidth]{>>}}}

\usepackage{pgf} 
\usepackage{color}
\usepackage{xcolor}

\usepackage[colorlinks=true,linkcolor=blue,citecolor=blue,urlcolor=cyan]{hyperref}
\usepackage{bbm}
\usepackage{cleveref}
\usepackage{relsize}
\usepackage{float}

\usepackage[backend=biber,date=year,giveninits=true,sorting=nyt,natbib=true,maxcitenames=2,maxbibnames=10,url=false,doi=true,backref=false]{biblatex}
\renewbibmacro{in:}{\ifentrytype{article}{}{\printtext{\bibstring{in}\intitlepunct}}}

\let\originalleft\left
\let\originalright\right
\renewcommand{\left}{\mathopen{}\mathclose\bgroup\originalleft}
\renewcommand{\right}{\aftergroup\egroup\originalright}
\hypersetup{
	colorlinks,
	citecolor=blue,
	filecolor=black,
	linkcolor=blue,
	urlcolor=blue
}
\usepackage[font=small,labelfont=bf]{caption}
\renewcommand*\thetable{\Roman{table}}
\numberwithin{equation}{section}
\fontfamily{cmr}\selectfont
	\fontsize{11}{10}\selectfont

\begin{document}

\title{A generalized expansion method for computing Laplace--Beltrami eigenfunctions on manifolds}

\author{Jackson C. Turner}
\address{Department of Applied Physics and Applied Mathematics, Columbia University, New York City 10027}
\email{jackson.turner@columbia.edu}
\author{Elena Cherkaev}
\address{Department of Mathematics, University of Utah, Salt Lake City 84112}
\email{elena@math.utah.edu}
\author{Dong Wang}
\address{School of Science and Engineering, The Chinese University of Hong Kong, Shenzhen, Guangdong, 518172, China \\
Shenzhen International Center for Industrial and Applied Mathematics, Shenzhen Research Institute of Big Data, Guangdong, 518172, China }
\email{wangdong@cuhk.edu.cn}


\begin{abstract}

 Eigendecomposition of the Laplace--Beltrami operator is instrumental for a variety of applications from physics to data science. We develop a numerical method of computation of the eigenvalues and eigenfunctions of the Laplace--Beltrami operator on a smooth bounded domain based on the relaxation to the Schr\"odinger operator with finite potential on a Riemannian manifold and projection in a special basis. We prove spectral exactness of the method and  provide examples of calculated results and applications, particularly, in quantum billiards on manifolds.


\end{abstract}
\maketitle

\section{Introduction}\label{sec:intro}
The Laplace--Beltrami operator plays an important role in the differential equations that describe many physical systems. These include, for example, vibrating membranes, fluid flow, heat flow, and solutions to the Schr\"odinger equation. Another example is that of spectral partitions---collections of $k$ pairwise disjoint open subsets such that the sum of their first Laplace–Beltrami eigenvalues is minimal \cite{bogosel2017efficient,bourdin2010optimal,wang_2022,Wang_2019}. This has a wide class of applications including data classification \cite{osting2013minimal}, interacting agents \cite{conti2003optimal,chang2004segregated,cybulski2008}, and so on. In all the above applications, the fundamental question is how to efficiently compute the eigenvalues of the Laplace--Beltrami operator in an arbitrary domain with a proper boundary condition. Also, the Laplace--Beltrami operator is crucial to understanding systems described by nonlinear Schr\"odinger equations, such as the propagation of Langmuir waves in an ionized plasma \cite{plasma1, plasma2}, the single-particle ground-state wavefunction in a Bose–Einstein condensate \cite{plasma2}, the slowly-varying envelope of light waves in Kerr media \cite{fibich2015nonlinear}, and water surface wave packets \cite{zakharov1968stability}.

The Laplace--Beltrami operator of a scalar function $f$ on a Riemannian manifold $(\mathcal{M},g)$ is defined as the surface divergence of the vector field gradient of $f$,
\begin{equation}
    \Delta_g f=\nabla \cdot \nabla f.
\end{equation}
The divergence of a vector field $X$ with metric $g$ is (in Einstein notation)
\begin{equation}
    \nabla \cdot X=\frac{1}{\sqrt{|g|}} \partial_{i}\left(\sqrt{|g|} X^{i}\right),
\end{equation}
and the gradient of a scalar function $f$ is
\begin{equation}
    (\operatorname{grad} f)^{i}=\partial^{i} f=g^{i j} \partial_{j} f.
\end{equation}
From above, we obtain the Laplace--Beltrami operator acting on functions over $(\mathcal{M},g)$,
\begin{equation} \label{eq:LapBel}
    \Delta_g f=\frac{1}{\sqrt{|g|}} \partial_{i}\left(\sqrt{|g|} g^{i j} \partial_{j} f\right).
\end{equation}
In general, the Helmholtz equation (\textit{i.e.}~Laplace--Beltrami eigenvalue problem) with Dirichlet boundary conditions on $\Omega \subset (\mathcal{M},g)$ is
\begin{equation}\label{eq:h1}
   \begin{cases} -\Delta_g u(x) = \lambda u(x), & x \in \Omega\\
   u(x) = 0, & x\in \partial \Omega. \end{cases}
\end{equation}

In this paper, we develop a numerical method to find the eigenvalues and eigenfunctions of the Laplace--Beltrami operator with Dirichlet and periodic boundary conditions for arbitrary domains on various surfaces. The idea is highly motivated by the Schr\"odinger operator and is based off the method given in \cite{expmethod}. By using the Schr\"odinger operator relaxation, we relax the eigenvalue problem on an arbitrary domain into the eigenvalue problem for a Schr\"odinger operator on a regular domain which is convenient for the numerical discretization.


In \cite{num2}, comparable methods on manifolds using linear and cubic FEM operators and discrete geometric Laplacians are explored, and \cite{hyp1} provides a method for hyperbolic domains. There is extensive literature on the Laplacian for planar regions \cite{num1,num3, num4, num5, num6}. In methods for solving nonlinear Schr\"odinger equations, finite difference discretizations of the Laplace operator are often used \cite{bao2004ground,bao2004computing, chang2004segregated}. It is likely many of these methods above can be extended to the Laplace--Beltrami operator on manifolds. The method we present in this paper has some immediate advantages over the finite difference method---since the boundary of domains are characterized by a potential function (see Theorem~\ref{thm:V}), no creation of a complicated mesh is needed, allowing for more generic domains and producing smooth solutions. Also, the method has promise to be quite robust in discretizing the operator on domains with corners (as in \autoref{tab:Lshape}), especially in applications when computation of many eigenvalues is required, whereas the finite difference method is notoriously inefficient on such domains. 

The rest of the paper is organized as follows. In Section~\ref{sec:GEM}, we recall the Schr\"odinger operator and introduce the generalized expansion method. We discuss and prove the convergence and accuracy of the relaxation and approximation in Section~\ref{sec:analysis} and show extensive numerical experiments in Section~\ref{sec:numerics}. We investigate applications to spherical domains, periodic domains, and billiard problems in Section~\ref{sec:applications} and draw some conclusion in Section~\ref{sec:conclusion}.

\section{Generalized Expansion Method}\label{sec:GEM}

The time-independent Schr\"odinger equation on a Riemannian manifold $\mathcal{M}$ with metric $g$, potential ${V}(x)$, and energy levels $E_n$ is
\begin{equation} \label{eq:timeind}
\hat{H}\psi_{n}(x) = \left[-
\Delta_g + {V}(x)\right]\psi_{n}(x) = E_n \psi_{n}(x), \\ 
\end{equation}
where $\Delta_g$ is the Laplace--Beltrami operator on $(\mathcal{M},g)$ as in \eqref{eq:LapBel}. The Schr\"odinger equation is an eigenvalue problem for the Schr\"odinger operator $\hat H = -\Delta_g + V(x)$. When
\begin{equation} \label{eq:9}
{V}(x) =  \left\{
\begin{array}{ll}
      0 & x \in \Omega \\
      \infty & x \not\in \Omega,
\end{array}
\right.
\end{equation}
the eigenvalue problem for $\hat{H}$ is equivalent to \eqref{eq:h1}. Eigenfunctions are normalized by setting
\begin{equation} \label{eq:11}
\int_{\Omega}|\psi_n(x)|^2d x = 1,
\end{equation}
where $|\psi_n(x)|^2dx$ is a probability density.

In \cite{expmethod}, a method is given to solve \eqref{eq:h1} with $g = I_2$ on any bounded smooth $\Omega\subset \mathbb R^2$ by embedding it in a rectangle, as in \autoref{fig:domain}. In order to evaluate the Laplace--Beltrami eigenvalues for $\Omega$ on a 2-D surface, we generalize this method when considering $\Omega$ as a smooth subset of a bounded manifold $S = (\mathcal{M},g)$ using a complete set of orthonormal eigenfunctions $\mathcal F^S_\infty = \{\phi_n\}_{n=1}^\infty$ on $S$ with corresponding eigenvalues $\{\lambda_n(S)\}_{n=1}^\infty$ (with $\lambda_1 < \lambda_2 \leq \lambda_3 \leq \dots$). Here, we assume all $\phi_n \in \mathcal F^S _ \infty$ have Dirichlet boundary conditions on $\partial S$, but for cases when $|\partial \Omega \cap \partial S| > 0$, one may choose to use other boundary conditions to obtain solutions of \eqref{eq:h1} with $\left. u\right\vert_{x \in \partial \Omega \cap \partial S} \not \equiv 0$, as in the periodic examples in Section~\ref{sec:per}.
\begin{figure}[h!]\centering
\begin{tikzpicture}    
\pgfmathsetseed{3}
     \draw[fill={rgb:black,1;white,4}] (-3,-3) rectangle (3,3);
     \node at (-3,2) [below right] {$S$};
\draw[fill={rgb:black,2;white,3}] plot [smooth cycle, samples=7,domain={1:7}]
     (\x*360/7+5*rnd:1cm+2cm*rnd) node at (0,0) {$\Omega$};
\end{tikzpicture}
  \caption{With the expansion method, a bounded domain $\Omega$ is embedded onto a rectangle $S$ with euclidean geometry and $\hat H$ is projected onto $\mathcal F^S_N$.}
  \label{fig:domain}
\end{figure}
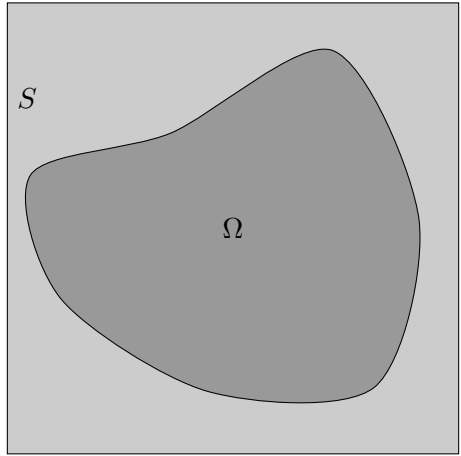

In this method, we use $\mathcal F^S _N = \{\phi_n\}_{n=1}^N$ as a basis on which we expand the operator $\hat H$ and seek solutions of \eqref{eq:h1}, or the equivalent equation involving the eigenvalue problem of the Schr\"odinger operator $\hat H$,
\begin{equation}
    \hat{H}\psi(x) = \left[-\Delta_g + \Tilde V(x)\right]\psi(x) = \lambda\psi(x),
\end{equation}
with $\Tilde V(x)$ defined as
\begin{equation}\label{eq:veq2}
\Tilde V(x) =  \left\{
\begin{array}{ll}
  0 & x \in \Omega \\
  \infty & x \in S\setminus\Omega.
\end{array} \right.
\end{equation}
We approximate $\Tilde V(x)$ as
\begin{equation}\label{eq:veqnew}
V(x) = V_0 \chi_{S\setminus \Omega}(x) = \left\{
\begin{array}{ll}
      0 & x \in \Omega \\
      V_0 \gg 1 & x \in S\setminus\Omega.
\end{array} \right.
\end{equation}
This allows us to discretize the operator $\hat H \approx H_N$,
\begin{equation} \label{eq:inprod}
    H_{N_{nm}} = \langle \phi_n, \hat{H}\phi_m\rangle = \lambda_n^S\delta_{nm} + \int_S V(x)\phi^*_n\phi_m dx,
\end{equation}
where we truncate $n,m \leq N$. The eigenvalues and eigenvectors of $H_N$ approximate those of $\hat{H}$.

\section{Convergence analysis}\label{sec:analysis}
In this section, we provide a rigorous proof on the convergence of the proposed method in the sense of $V_0\rightarrow \infty$ and $N \rightarrow \infty$. To keep self-consistency of the paper, we first recall some definitions and preliminary results from \cite{extremum,teschl2009mathematical,gilbarg1977elliptic}.





\begin{definition}
Suppose $A_{n}$ and $A$ are self-adjoint operators. We say that $A_{n}$ converges to $A$ in the strong resolvent sense, if
\begin{equation}
\| (R_{A_{n}}(z) - R_{A}(z)) \phi \| \to 0, \quad \forall \phi \in \mathfrak D(A_n)
\end{equation}
for some $z \in \Gamma=\mathbb{C} \setminus \Sigma, \Sigma=\sigma(A) \cup \bigcup_{n} \sigma\left(A_{n}\right)$ where the function $R_A(z)$ is the resolvent of $A$.
\end{definition}
\begin{definition}
A subset $\mathfrak{D}_0 \subseteq \mathfrak{D}(A)$ is a core of $A$ when $\{(x, A x): x \in \mathfrak D_0\}$ is dense in $\{(x, A x): x \in \mathfrak D(A)\}$.
\end{definition}

\begin{lem}\label{lem:book1}
(6.36 of \cite{teschl2009mathematical}): Let $A_{n}, A$ be self-adjoint operators. Then $A_{n}$ converges to $A$ in the strong resolvent sense if there is a core $\mathfrak{D}_{0}$ of $A$ such that for any $\psi \in \mathfrak{D}_{0}$ we have $P_{n} \psi \in \mathfrak{D}\left(A_{n}\right)$ for $n$ sufficiently large and $A_{n} P_n \psi \rightarrow A \psi$.
\end{lem}

\begin{thm}\label{thm:book1}
(6.38 of \cite{teschl2009mathematical}). Let $A_{n}$ and $A$ be self-adjoint operators. If $A_{n}$ converges to $A$ in the strong resolvent sense, we have $\sigma(A) \subseteq \lim _{n \rightarrow \infty} \sigma\left(A_{n}\right)$.
\end{thm}

\begin{thm}\label{thm:book2} (2.2.3 of  \cite{extremum}). Let $L_{n}$ be a sequence of uniformly elliptic operators defined on an open set $D$ by
\begin{equation}\label{eq:elliptic}
L_{n} u:=-\sum_{i, j=1}^{N} \frac{\partial}{\partial x_{i}}\left(a_{i j}^{n}(x) \frac{\partial u}{\partial x_{j}}\right)+a_{0}^{n}(x) u.
\end{equation}
We assume that, for fixed $i, j$, the sequence $a_{i, j}^{n}$ is bounded in $L^{\infty}$ and converge almost everywhere to a function $a_{i, j}$; we also assume that the sequence $a_{0}^{n}$ is bounded in $L^{\infty}$ and converges weakly-* in $L^{\infty}$ to a function $a_{0}$. Let $L$ be the (elliptic) operator defined on $D$ as in \eqref{eq:elliptic} by the functions $a_{i, j}$ and $a_{0}$. Then each eigenvalue of $L_{n}$ converges to the corresponding eigenvalue of $L$.
\end{thm}
\begin{thm}\label{thm:A1}
(9.29 of \cite{gilbarg1977elliptic}) Let $u \in W^{2, p}(S) \cap C^{0}(\bar{S})$ satisfy $L u=f$ in $S, u=\varphi$ on $\partial S$ where $f \in L^{p}(S), \varphi \in C^{\beta}(\bar{S})$ for some $\beta>0$, and suppose that $\partial S$ satisfies a uniform exterior cone condition. Then $u \in C^{\alpha}(\bar{S})$ for some $\alpha >0$.\end{thm}

Now, we prove spectral exactness of the expansion method in $V_0$ and $N$ in Theorems~\ref{thm:V} and \ref{thm:M}. We give an example of calculating solutions and the rate of convergence in $V_0$ of eigenvalues for the relaxed problem on an interval in Example~\ref{ex:L}. We provide intuition for efficient implementation of the expansion method in Remark~\ref{rem:A}.
\begin{thm} \label{thm:V}
The eigenvalues of the Schr\"odinger operator
\begin{equation}
    \hat{H}(V_0) = -\Delta_g + V(x), \qquad V(x) =\left\{\begin{array}{ll}
         0, & x\in \Omega \\
         V_0, & x \not\in \Omega
    \end{array} \right.
\end{equation}
acting on a bounded Riemannan manifold $(\mathcal{M},g)$, with $
\Omega$ smooth in $(\mathcal{M},g)$ converge monotonically to the eigenvalues of $-\Delta_g$ with Dirichlet boundary conditions on $\Omega$ as $V_0\rightarrow \infty$.
\end{thm}
\begin{proof}
By considering the volume form on the manifold, we have the inner product:
\begin{equation}\langle f_1, f_2\rangle_g = \int_\mathcal{M} \overline{f_1(x)}f_2(x) (\operatorname{det} g)^{1/2} d x_1 \cdots d x_n\end{equation}
Now, from the Rayleigh quotient of an elliptic linear operator $\mathcal{L}$ on a Riemannian manifold $(\mathcal{M},g)$,
\begin{equation}
    R(\mathcal{L},u) = \frac{\langle u,\mathcal{L}u\rangle_g}{\langle u,u \rangle_g},
\end{equation}
we have
\begin{equation}
    R(\hat{H}(V_0),u) = \frac{\langle u,-\Delta_g u\rangle_g}{\langle u,u \rangle_g} + \frac{\langle u,Vu \rangle_g}{\langle u,u \rangle_g}.
\end{equation}
By the Courant-Fischer formula,
\begin{equation}
    \lambda_k(\hat{H}(V_0)) = \inf_{\mathcal{Q}\in \mathcal{Q}_k}\sup_{u\in \mathcal{Q}}\left(R(\hat{H},u)\right)
\end{equation}
with $\mathcal{Q}_{k}$ being the family of subspaces of $H_{0}^{1}(\mathcal{M})$ of dimension $k$, we obtain the following inequality,
\begin{equation}
    \lambda_k(\hat{H}(V_0^*)) \geq \lambda_k(\hat{H}(V_0)), \quad \text{for } V_0^* > V_0,
\end{equation}
giving us monotonicity. We also have for $u\in H^1_0(\mathcal{M})$ and as $V_0\rightarrow \infty$, $R(\hat{H}(V_0),u) < \infty$ if and only if $\operatorname{supp}(u) \subseteq \Omega$ almost everywhere. Hence, we have
\begin{equation}
    \lim_{V_0\rightarrow\infty}\lambda_k(\hat{H}(V_0)) = \inf_{\mathcal{Q}^*\in \mathcal{Q}^*_k}\sup_{u\in \mathcal{Q}^*}\left(R(\hat{H},u)\right)
\end{equation}
with $\mathcal{Q}^*_{k}$ being the family of subspaces of $H_{0}^{1}(\Omega)$ of dimension $k$. This is precisely the Courant-Fischer definition of the eigenvalues of the Laplace--Beltrami operator with Dirichlet boundary conditions on $\Omega$.
\end{proof}
\begin{ex}\label{ex:L}
Consider the regions $S = (0,2)$ and $\Omega = (0,1)$. The eigenvalues of the Helmholtz equation on $\Omega$,
\begin{equation}\label{eq:ex1}
    -v_k'' = \mu_k v_k, \qquad v_k(0) = v_k(1) = 0,
\end{equation}
can be approximated by the eigenvalues of the Schr\"odinger operator $\hat H_{V_0} = -\partial_x^2 + V(x;V_0)$ on $S$ with $V(x;V_0) = V_0 \chi_{S\setminus \Omega} (x)$ and large $V_0 \gg 1$, with
\begin{equation}\label{eq:ex2}
\hat H_{V_0} u_k = \lambda_k(V_0) u_k,\qquad u(0) = u(2) = 0,
\end{equation}
and rate of convergence
\begin{equation}
    |\lambda_k(V_0) - \mu_k| \sim \frac{1}{\sqrt{V_0}}
\end{equation}
up to a constant, as $V_0 \uparrow \infty.$
\end{ex}

\begin{proof}
We have \eqref{eq:ex1} has eigenvalues $\mu_k = \pi^2 k^2$ for $k \in \mathbb N \setminus \{0\}$. Solutions to \eqref{eq:ex2} are of the form:
\begin{equation}\label{eq:ex3}
    u_k = \begin{cases}
    \sin \sqrt{\lambda_k} x, & x \in(0,1] \\ \frac{\sin \sqrt{\lambda_k}}{\sinh \sqrt{V_0 - \lambda_k}} \sinh( \sqrt{V_0 - \lambda_k} (2-x) ), & x \in(1,2).
    \end{cases}
\end{equation}
By setting $u_k \in C^1(0,2)$, we arrive at
\begin{equation}
\begin{aligned}
    \sqrt{\lambda_k} \cot \sqrt{\lambda_k} = -\sqrt{V_0 - \lambda_k} \coth \sqrt{V_0 - \lambda_k} \implies
    \frac{\sqrt{\lambda_k}}{\sin \sqrt \lambda_k} \sim \sqrt{V_0},
    \end{aligned}
\end{equation}
as $V_0 \uparrow \infty$. We then have by Taylor expansion of $\sin(\cdot)$ about $\sqrt{\mu_k}$,
\begin{equation}
\begin{aligned}
    \sin \sqrt{\lambda_k} = \sin(\sqrt{\lambda_k} - \sqrt{\mu_k}+ \sqrt{\mu_k}) \approx \pm(\sqrt{\lambda_k} - \sqrt{\mu_k}) \sim \frac{\sqrt{\lambda_k}}{\sqrt{V_0}} \sim \frac{1}{\sqrt{V_0}}
    \end{aligned}
\end{equation}
since $\sin\sqrt{\mu_k}=0$. The last relation above arises from $\sqrt{\lambda_k} \sim \sqrt{\mu_k}$, a constant. Hence altogether,
\begin{equation}
    |\lambda_k- \mu_k| = |\sqrt{\lambda_k} -\sqrt{\mu_k}| \cdot |\sqrt{\lambda_k} + \sqrt{\mu_k}| \sim \frac{1}{\sqrt{V_0}}.
\end{equation}
\end{proof}
\begin{thm} \label{thm:M}
Given a complete orthonormal basis $ \mathcal F^S_\infty \subset H^1_0(S)$ of Laplace--Beltrami eigenfunctions on a bounded smooth domain $S =(\mathcal M, g)$, the Dirichlet eigenvalues of the $N$-dimensional operator $H_N$ where $H_{N_{ij}} = \langle \phi_i, \hat{H} \phi_j\rangle$ for $i,j\leq N$ converge to those of $\hat{H}$ as $N\rightarrow\infty$ where 
\begin{equation}
    \hat{H} = -\Delta_g + V(x), \qquad
        V(x) \in  L^\infty(S,\mathbb R).
\end{equation}
The results in this theorem hold for Neumann and periodic boundary conditions as well, using the appropriate basis and Sobolev space.
\end{thm}

\begin{proof}
We first note
\begin{equation}H_N: \mathcal F^S_N \to  \mathcal F^S_N \qquad \hat H :H^1_0(S) \to L^2(S),\end{equation}
where $\mathcal F^S_N = \operatorname{span}\{\phi_j\}_{j=1}^N$. We also have, by the definition of $H_N$, with $u = \sum_{j=1}^\infty c_j \phi_j$:
\begin{equation}
\begin{aligned}
H_N u &:= \sum_{i=1}^N \left(\sum_{k=1}^N\langle\phi_i,\hat{H} \phi_k\rangle c_k \right) \phi_i = \sum_{j=1}^N\langle \phi_j, \hat{H}u\rangle \phi_j = P_N \hat{H}u\\
P_N v &:= \sum_{j=1}^N \langle \phi_j, v\rangle \phi_j
\end{aligned}
\end{equation}
 Furthermore, by convention we may extend the domain of $H_N$ to $H^1_0(S)$ by the following extension, which we will now use in the proof:
 \begin{equation}H_N u := P_N \hat H P_N u.\end{equation}
Without loss of generality, because $V(x)$ is bounded we can assume further that $V(x) \in [0,\infty)$ since the resulting spectra are merely shifted by a constant $\alpha$ when adding $\alpha$ to $V(x)$, hence we have $\operatorname{ker}(\hat H^*) = \emptyset$ by injectivity and self-adjointness, giving us  $\overline{\operatorname{ran}(\hat H)} = L^2(S)$. We consider the set $\mathfrak{D}_0 =  \mathcal F^S_\infty \cap H^1_0(S)$, and we have $\overline{\mathfrak D_0} = L^2(S) = \overline{H^1_0(S)}$, and $\overline{H \mathfrak D_0} = \overline{\operatorname{span}\{H\phi_j\}_{j=1}^\infty} \cap \overline{L^2(S)} = L^2(S)$, by $V \in L^\infty(S)$. Hence, the graph $\{( x , \hat H x ) : x \in \mathfrak D_0\}$ is dense in $\{(x,\hat H x) : x \in \mathfrak D (\hat H)\}$, therefore $\mathfrak D_0$ is a \textit{core} of $\hat H$. Furthermore, we have for all $\psi \in \mathfrak D_0$ that $P_N \psi \in \operatorname{span}\{\phi_j\}_{j=1}^N$ and $H_N \psi = P_N \hat H P_N \psi \to \hat H P_N \psi$, so by Lemma \ref{lem:book1} we have strong convergence in the resolvent sense, and hence the conditions for Theorem \ref{thm:book1} are met, and we have 
\begin{equation}\lim_{N\to \infty} \sigma(H_N) \supseteq \sigma(\hat H).\end{equation}
Now, it is well-known these operators have purely point spectra. We consider $V_N =P_N V$ and consider $\lambda_N \to \lambda_*$, some converging sequence of eigenvalues of $H_N$ with corresponding eigenvectors $u_N$. We have $P_N u_N = u_N$ and $P_N \Delta_g u_N = \Delta_g u_N$, giving us
\begin{align*}
H_N u_N = (-\Delta_g + V_N) u_N = \lambda_N u_N &\implies \lambda_N \in \sigma(-\Delta_g + V_N), \\
\sigma(-\Delta_g + V_N) \longrightarrow \sigma(\hat H) &\implies \lambda_* \in \sigma(\hat H)
\end{align*}
by Theorem~\ref{thm:book2}, giving us
\begin{equation}\lim_{N\to \infty} \sigma(H_N) \subseteq \sigma(\hat H).\end{equation}
Altogether, we have the desired result,
\begin{equation}\lim_{N\to \infty} \sigma(H_N) = \sigma(\hat H).\end{equation}
Similar proofs can be made for the Neumann and periodic boundary cases.
\end{proof}
\begin{rem}\label{rem:A} For fixed $\Omega,\ V_0,$ and $N$, an efficient implementation of this method is to seek an integrable domain $S \supset \Omega$ to minimize the following value involving the $L^2$-induced norm of the difference of the operators acting on the finite-dimensional space $\mathcal F^S_N$:
\begin{equation}\tau_{\Omega}(S) := V_0^{-1}\cdot \left\|\left.[\hat H - H_{N}(S)]\right\vert_{\mathcal F^S_N}\right\|_{\mathcal B(L^2(\Omega),L^2(\Omega))} =   V_0^{-1} \cdot \left\| V - \sum_{j=1}^N \langle \phi_j, V\rangle \phi_j \right\|_{L^\infty(\Omega)}\end{equation}
This procedure is equivalent to fitting a domain $\Omega$ properly into a solvable set $S$ so that the potential $V$ that characterizes $\Omega$ is well-approximated by $P_N V$.
\end{rem}
\section{Numerical Accuracy}\label{sec:numerics}

Note that contrary to Theorem \ref{thm:V}, in practice, the method loses accuracy if $V_0$ is chosen to be too large, due to floating point round-off errors. We provide a numerical example here. One can identify the shape of a triangle given the spectrum of the solution to the Helmholtz equation \cite{triangle}, and a formula can be derived for the Laplacian eigenvalues $\lambda_n$ of an equilateral triangle with Dirichlet boundary conditions \cite{diab}, for positive integers $p$ and $q$,
\begin{equation}\label{eq:trian}
    \lambda_n \equiv \lambda_{pq} = \Big(\frac{4\pi}{3}\Big)^2(p^2+q^2-pq),\ 1\leq q\leq p/2,
\end{equation}
where $\lambda_n$ is a multiple eigenvalue with multiplicity 2 if $p \neq 2q$. $\lambda_n$ is in units ${1}/{a^2}$ and $a$ is the side length of the triangle. In \autoref{fig:ac2}, we compare these known eigenvalues with those computed using the expansion method with varying $V_0$. Although we showed monotonic convergence as $V_0 \to \infty$ in \ref{sec:analysis}, a properly chosen value would be at about $V_0 \approx 2.6 \times 10^6$, depending on how many eigenvalues one wishes to compute and the chosen domain.

\begin{figure}[h!]\centering
\includegraphics[width=.5\linewidth]{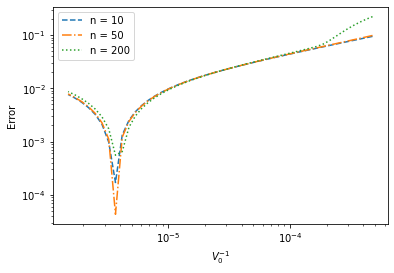}
\caption{Average relative error for the first $n$ eigenvalues for varying values of $V_0^{-1}$.}\label{fig:ac2}
\end{figure}

We now solve for the L-shaped domain modes numerically using the expansion method. Here, we use $\mathcal F^S_N = \{\phi_n\}_{n=1}^N$ on $S = (0,2)^2$ and $\Omega = S\setminus [1,2)^2$ with 
\begin{equation}\phi_n = \phi_{(n_1,n_2)} = 
\sqrt{\frac{2}{a_{1}}} \sin \left(\frac{\pi}{a_{1}} n_{1} x_{1}\right) \sqrt{\frac{2}{a_{2}}} \sin \left(\frac{\pi}{a_{2}} n_{2} x_{2}\right) =  \sin \left(\frac{\pi}{2} n_{1} x_{1}\right) \sin \left(\frac{\pi}{2} n_{2} x_{2}\right),
\end{equation}
just as in \cite{expmethod}. We provide examples of the computed eigenmodes in \autoref{fig:Lshape}. This L-shaped domain is common in the literature, as it is a simple construction of a domain
with no closed form solution \cite{trefethen2006computed, yuan2009bounds,fox1967approximations}. In \autoref{tab:Lshape}, we provide the computed eigenvalues corresponding to the provided eigenmodes, along with those computed using a second-order finite difference operator (see \cite{leveque1992numerical}), with uniform grid spacings $h$ for both the $x$ and $y$ axes:
\begin{align*}H^{\operatorname{FD}} &= -( I \otimes D^2_x + D_y^2 \otimes I) + V_{\operatorname{d}} \\
D_x^2 &= D_y^2 =\frac{1}{h^{2}}\left(\begin{array}{cccc}
-2 & 1 & & \\
1 & \ddots & \ddots & \\
& \ddots & \ddots & 1 \\
& & 1 & -2
\end{array}\right)\\
V_{\operatorname{d}} &= \left(\begin{array}{cccccccc}
v_{11} &  &  &  &  &  &  &  \\
 & v_{21} &  &  &  &  &  &  \\
 &  & \ddots &  &  &  &  &  \\
 &  &  & v_{N 1} &  &  &  &  \\
 &  &  &  & v_{22} &  &  &  \\
 &  &  &  &  & \ddots &  &  \\
& & &  & & & & v_{N N}
\end{array}\right) \\ 
v_{nm} &= V(x_n,y_m).
\end{align*}
The results in Table~\ref{tab:Lshape} of the computation of Laplacian eigenvalues for the L-shaped domain using the expansion method and FD are compared with the known values from \cite{trefethen2006computed}, which were remarkably computed with up to 8 digits of accuracy. This comparison provides numerical validation of the expansion method.

\begin{figure}[h!]\centering
 
\includegraphics[width=.2\linewidth]{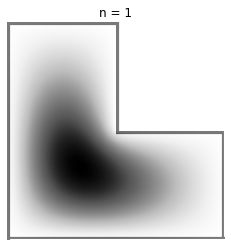}
\includegraphics[width=.2\linewidth]{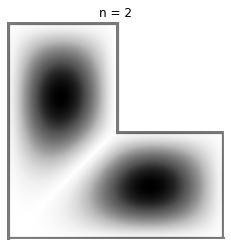}
\includegraphics[width=.2\linewidth]{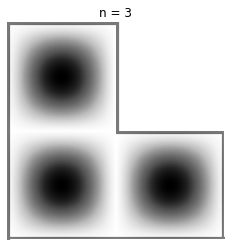}\\
\includegraphics[width=.2\linewidth]{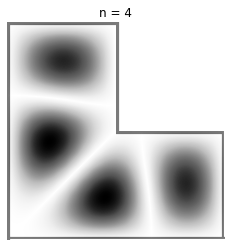}
\includegraphics[width=.2\linewidth]{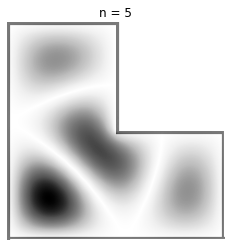}
\includegraphics[width=.2\linewidth]{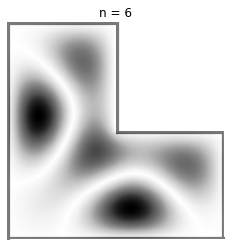}\\
\includegraphics[width=.2\linewidth]{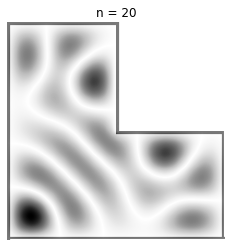}
\includegraphics[width=.2\linewidth]{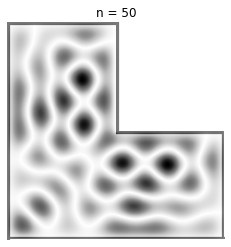}
\includegraphics[width=.2\linewidth]{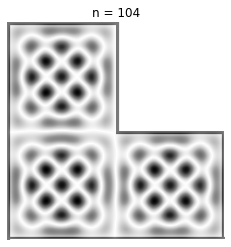}
\caption{Various states $|\psi|$ for the L-shaped region on a plane, computed using the expansion method.}
\label{fig:Lshape}
 
\end{figure}

\begin{table}[h!]
 
    \begin{minipage}{.5\linewidth}
      \centering
      \scalebox{0.75}{
        \begin{tabular}{|c|c|c|c|}
        \hline
           $n$ &  FD Method & Exp. Method &   Known Values \\
        \hline
        \hline
           1 &   9.33328 &    9.63359 &     9.63972 \\
           2 &   14.8927 &   15.1964  &    15.1973  \\
           3 &   19.4634 &    19.7385  &    19.7392  \\
           4 &   29.2480 &   29.5209  &    29.5215  \\
           5 &   31.2302 &   31.8982  &    31.9126  \\
           6 &   40.3540 &   41.4629  &    41.4745  \\
          20 &   98.9878 &  101.585   &   101.605   \\
          50 &   239.487 &  250.777   &   250.785   \\
         104 &   462.102 &  493.543   &   493.480    \\
        \hline
        \end{tabular}}
    \end{minipage}%
    \begin{minipage}{.5\linewidth}
      \centering
        \scalebox{0.75}{
        \begin{tabular}{|c|c|c|c|}
        \hline
           $n$ &  FD Method & Exp. Method &   Known Values \\
        \hline
        \hline
           1 &     9.69329 &            10.1213 &     9.63972 \\
           2 &    14.9296  &            15.7156 &    15.1973  \\
           3 &    19.4097  &            20.1868 &    19.7392  \\
           4 &    28.5903  &            29.8785 &    29.5215  \\
           5 &    31.171   &            32.8286 &    31.9126  \\
           6 &    39.4412  &            42.9975 &    41.4745  \\
          20 &    91.6147  &           105.747  &   101.605   \\
          50 &   179.955   &           255.717  &   250.785   \\
         104 &   306.101   &           514.121  &   493.480    \\
        \hline
        \end{tabular}   }
    \end{minipage} 
        \caption{Computed Laplacian eigenvalues of the L-shaped region. Both schemes used a discretized 2-D Schr\"odinger operator in $\mathbb{R}^{N \times N}$ and $V_0 = 2.1\times 10^5$. $N=2500$ (left) and $N = 225$ (right). The known values were taken from \cite{trefethen2006computed}.}
\label{tab:Lshape}
 
\end{table}

\section{Applications}\label{sec:applications}
\subsection{Spherical domains.}

The expansion method can be used on a variety of manifolds, and for example, on a spherical surface. The eigenfunctions of the Laplace--Beltrami operator on a sphere are the spherical harmonics $Y$, which are solutions to
\begin{equation}\label{eq:sphhelm}
    \frac{1}{\sqrt{|g|}} \partial_{i}\left(\sqrt{|g|} g^{i j} \partial_{j} Y(\theta,\phi)\right) = \lambda Y(\theta,\phi)
\end{equation}
where
\begin{equation}
g=\left[\begin{array}{cc}
1 & 0 \\
0 & \sin ^{2} \theta
\end{array}\right] .
\end{equation}
Spherical harmonics $Y_{m \ell}$ provide a set of orthonormal functions and thus can be used as a basis. These functions are defined over the indices $m$ (integers) and $\ell$ (non-negative integers), where $Y_{m \ell}$ is defined for $|m|\leq \ell$. These functions are known explicitly ($P^m_\ell$ denoting associated Legendre polynomials),
\begin{equation}\label{eq:sphericalharmonics}
Y_{m \ell}=\left\{\begin{array}{ll}
(-1)^{m} \sqrt{2} \sqrt{\frac{2 \ell+1}{4 \pi} \frac{(\ell-|m|) !}{(\ell+|m|) !}} P_{\ell}^{|m|}(\cos \theta) \sin (|m| \varphi) & \text { if } m<0 \\
\sqrt{\frac{2 \ell+1}{4 \pi}} P_{\ell}^{m}(\cos \theta) & \text { if } m=0 \\
(-1)^{m} \sqrt{2} \sqrt{\frac{2 \ell+1}{4 \pi} \frac{(\ell-m) !}{(\ell+m) !}} P_{\ell}^{m}(\cos \theta) \cos (m \varphi) & \text { if } m>0 .
\end{array}\right.
\end{equation}
The eigenfunctions of the Laplace--Beltrami operator with Dirichlet boundary conditions for some smooth region on a sphere can be expressed in $L^2(S)$ as linear combinations of spherical harmonics,
\begin{equation}\label{eq:lincombsph}
   \psi_{j}(\phi,\theta) = \sum_{\substack{m,\ \ell}} c^{(j)}_{m\ell} Y_{\ell}^m(\phi,\theta).
\end{equation} 
As in \eqref{eq:inprod}, the matrix representation of the Schr\"odinger operator $\hat{H}$ in the space composed of the basis functions is given by
\begin{equation}\label{eq:dischamsph}
    H_{ij} = \int_{S}Y_i(\phi,\theta)\hat{H} Y_j(\phi,\theta) ds,
\end{equation}
where $i$ and $j$ each represent an index pair $(m,\ell)$. By substituting $\left[-\Delta_{g} + V(\phi,\theta)\right]$ for $\hat{H}$ and making a change of variables, we obtain $ds = r\sin (\theta)d\phi d\theta$ on the unit sphere and the discretized Hamiltonian,
\begin{equation} \label{eq:inprod2}
\begin{aligned}
    H_{N_{ij}} &= \langle Y_i, \hat{H}Y_j\rangle\\
    &= \ell(\ell+1)\delta_{ij} + V_0\int_S Y_i(\phi,\theta) Y_j(\phi,\theta) \sin{(\theta)}d\phi d\theta,
\end{aligned}
\end{equation}
for a large value $V_0 \gg 1$. We can then calculate the matrix $H_N$ and its eigenpairs numerically. We then expand the eigenvectors back into the spherical harmonic basis.

Using the presented method, we have calculated and plotted the first twelve states for the half-sphere, octant, and spherical square alongside their planar analogs (in grayscale) to illustrate the utility of the expansion method. We have plotted the absolute value to distinguish the nodal lines.

\begin{figure}[h!]\centering
\includegraphics[ width=.12\linewidth,trim={6cm 2.5cm 5.3cm .5cm},clip]{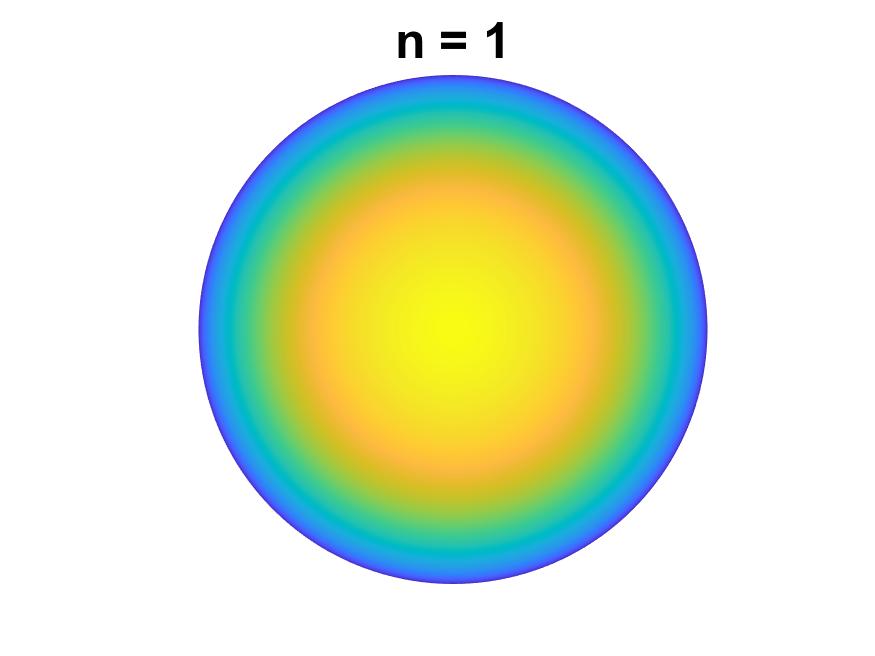}
\includegraphics[ width=.12\linewidth,trim={6cm 2.5cm 5.3cm .5cm},clip]{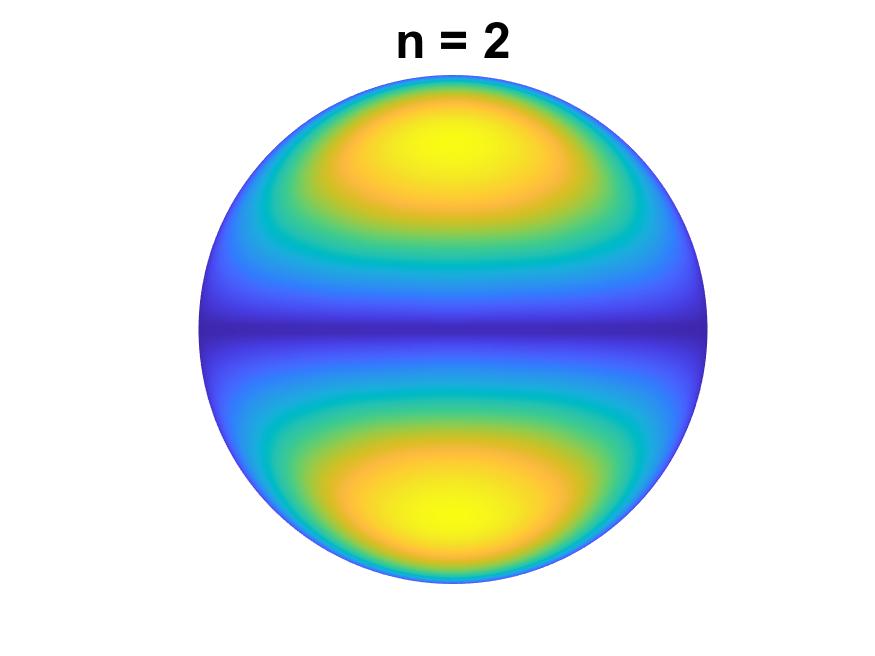}
\includegraphics[ width=.12\linewidth,trim={6cm 2.5cm 5.3cm .5cm},clip]{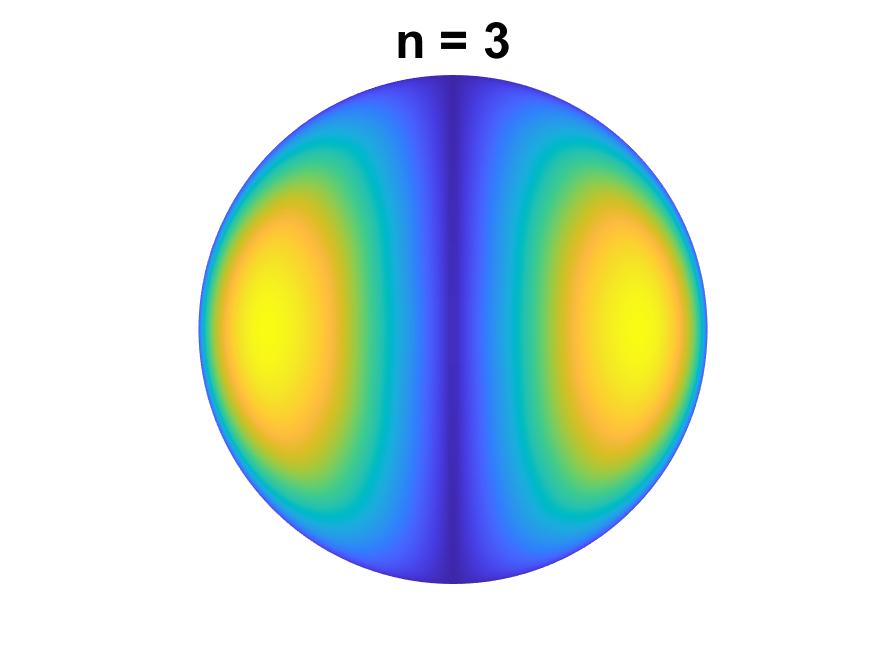}
\includegraphics[ width=.12\linewidth,trim={6cm 2.5cm 5.3cm .5cm},clip]{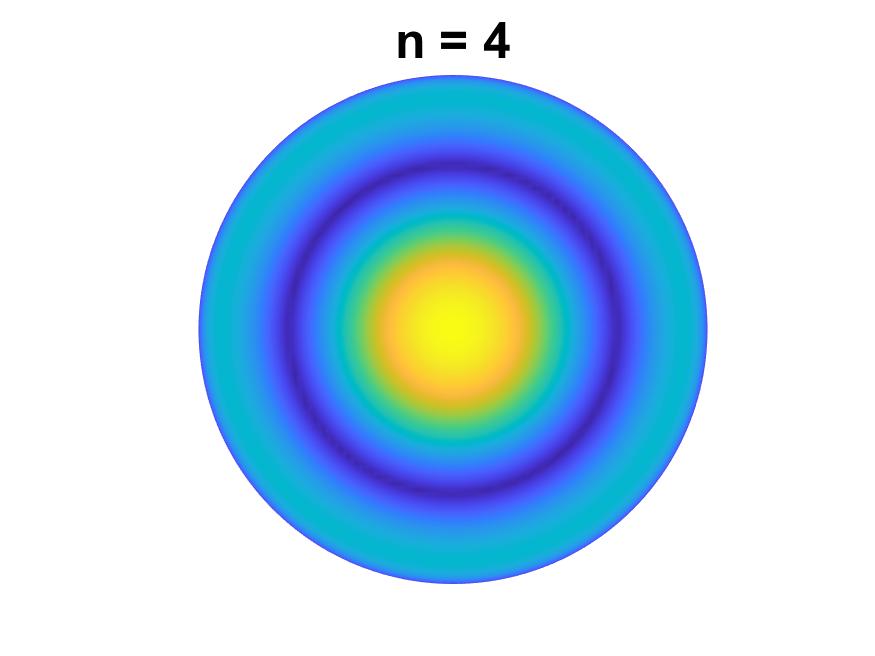}
\includegraphics[ width=.12\linewidth,trim={6cm 2.5cm 5.3cm .5cm},clip]{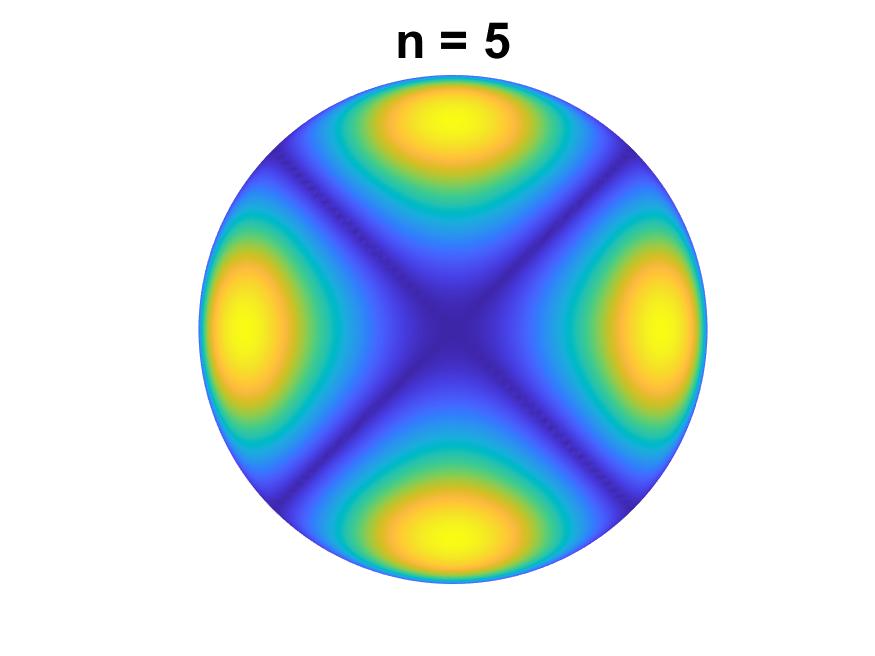}
\includegraphics[ width=.12\linewidth,trim={6cm 2.5cm 5.3cm .5cm},clip]{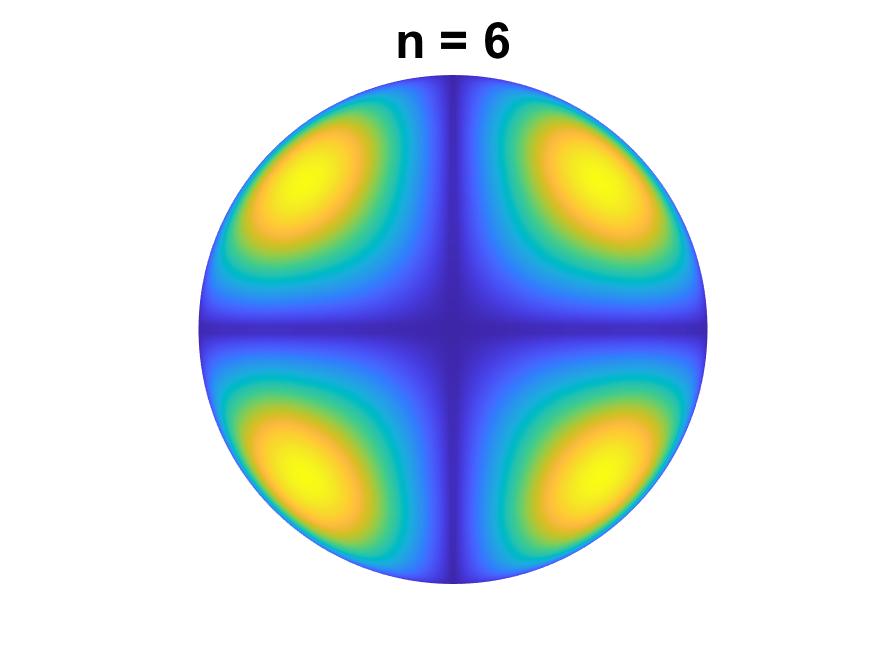} \\
\includegraphics[ width=.12\linewidth,trim={6cm 2.5cm 5.3cm .5cm},clip]{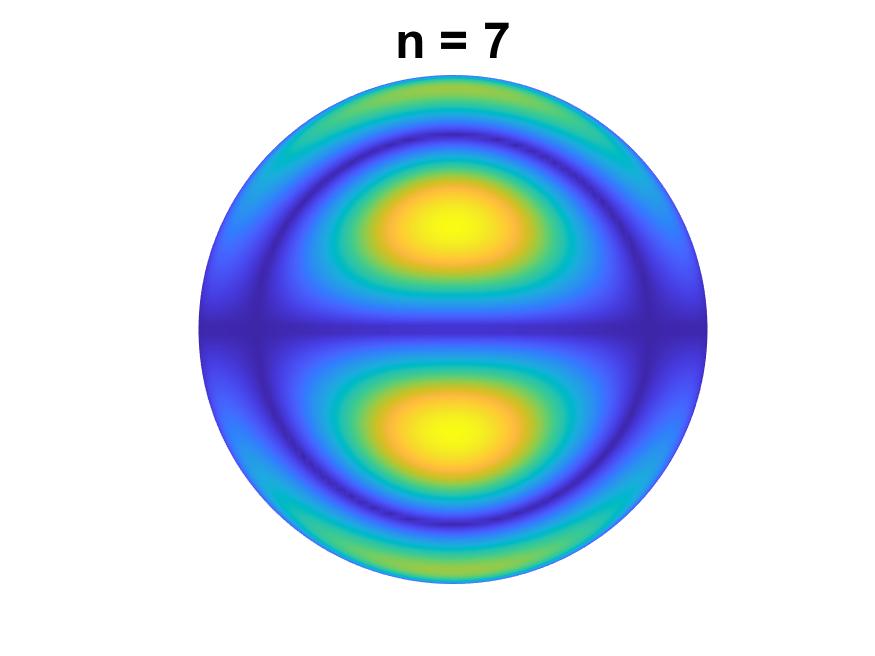}
\includegraphics[ width=.12\linewidth,trim={6cm 2.5cm 5.3cm .5cm},clip]{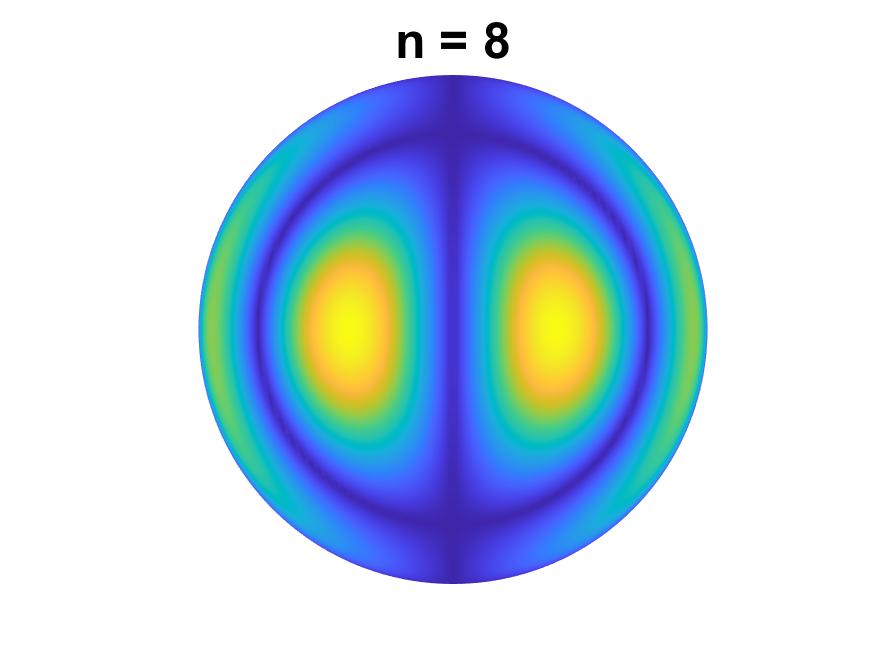}
\includegraphics[ width=.12\linewidth,trim={6cm 2.5cm 5.3cm .5cm},clip]{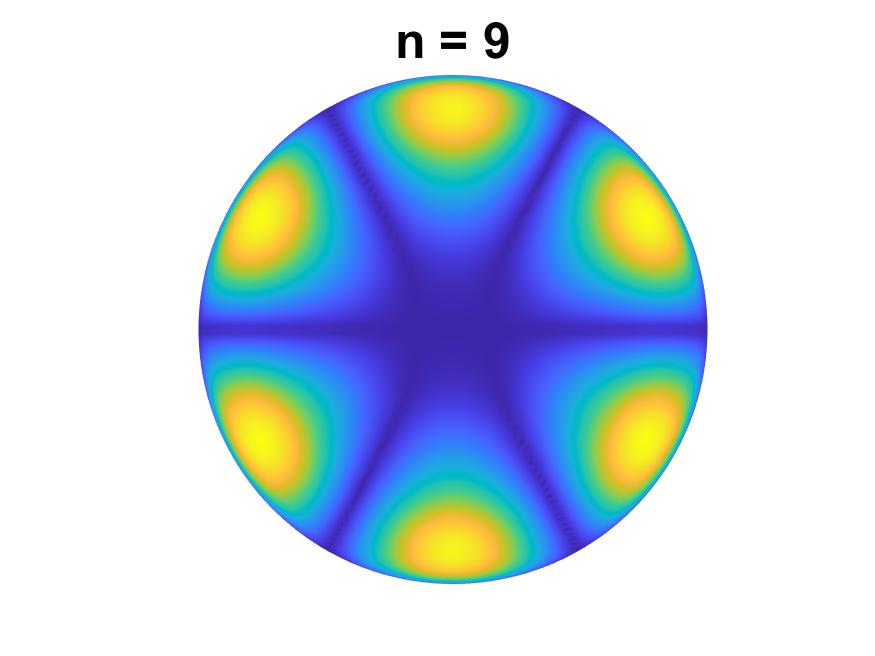}
\includegraphics[ width=.12\linewidth,trim={6cm 2.5cm 5.3cm .5cm},clip]{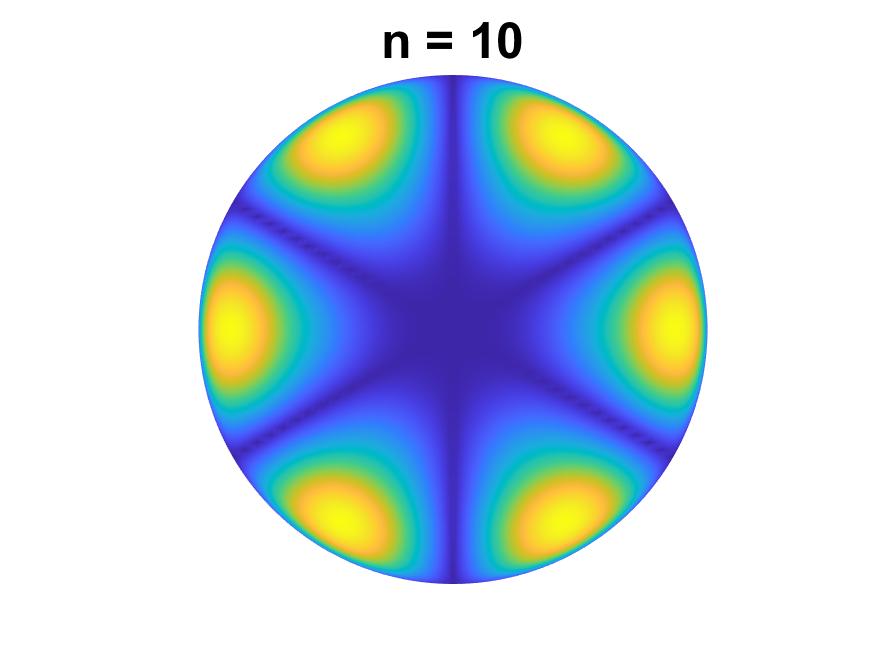}
\includegraphics[ width=.12\linewidth,trim={6cm 2.5cm 5.3cm .5cm},clip]{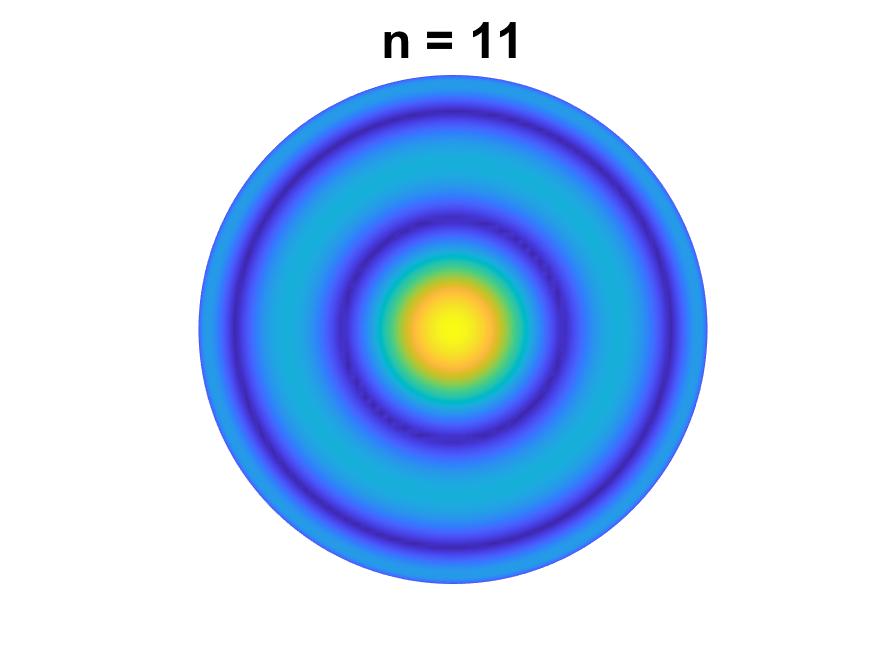}
\includegraphics[ width=.12\linewidth,trim={6cm 2.5cm 5.3cm .5cm},clip]{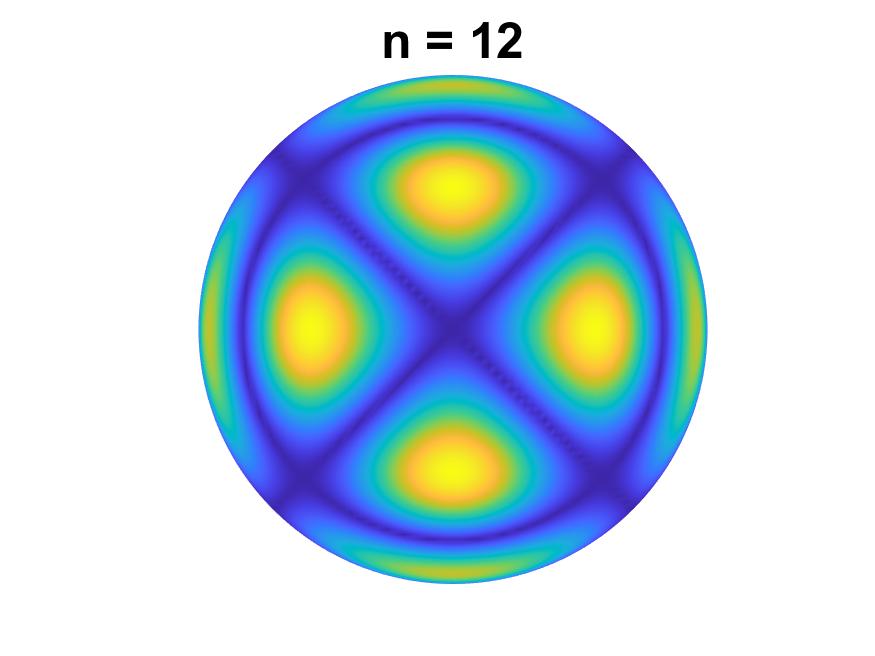} \\
\includegraphics[ width=.12\linewidth,trim={6cm 2.5cm 5.3cm .5cm},clip]{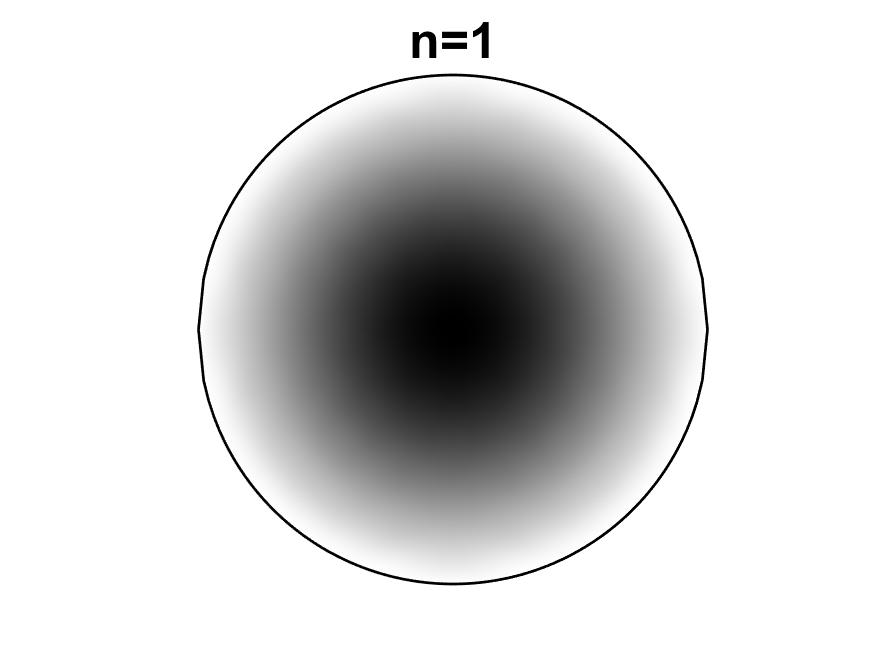}
\includegraphics[ width=.12\linewidth,trim={6cm 2.5cm 5.3cm .5cm},clip]{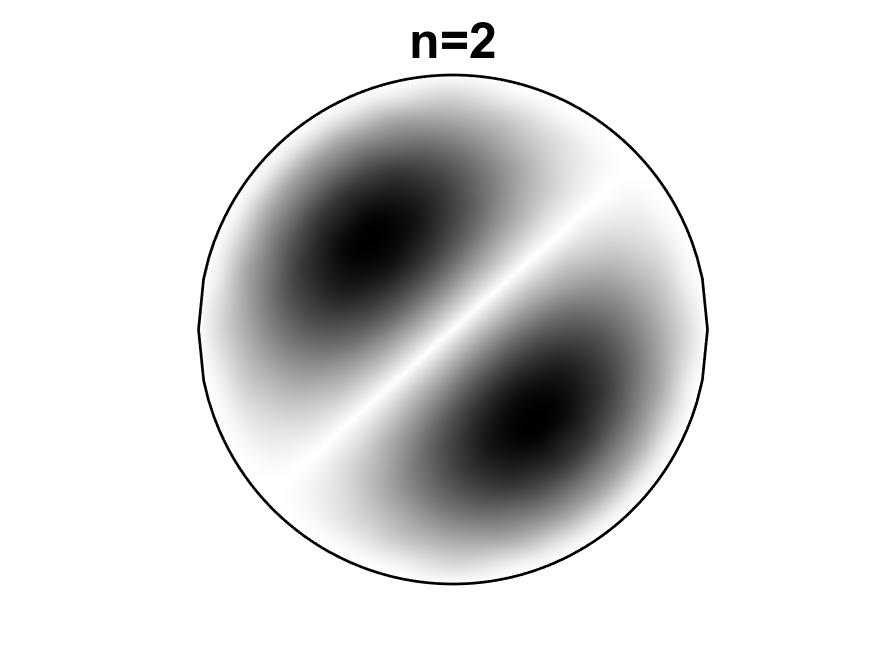}
\includegraphics[ width=.12\linewidth,trim={6cm 2.5cm 5.3cm .5cm},clip]{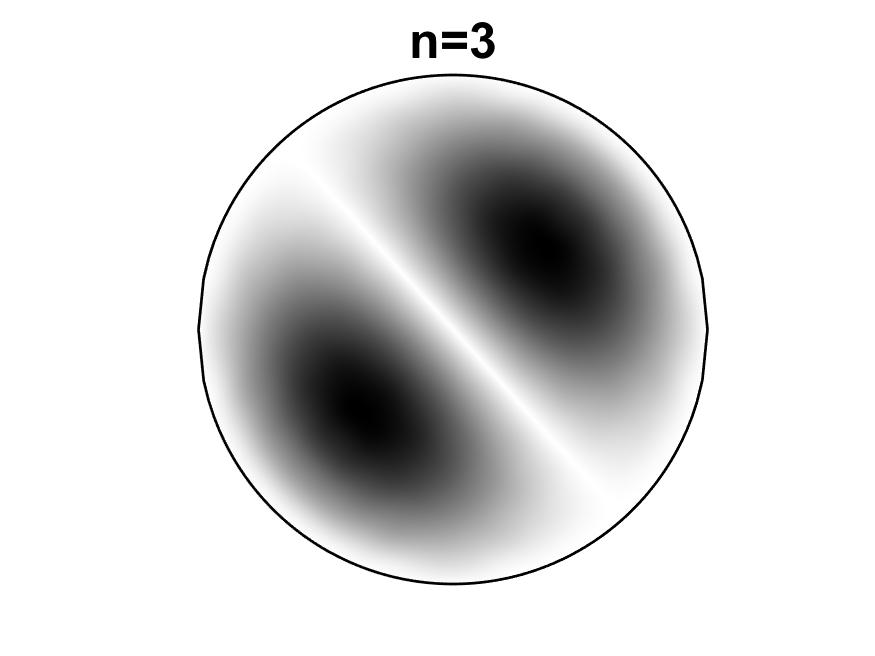}
\includegraphics[ width=.12\linewidth,trim={6cm 2.5cm 5.3cm .5cm},clip]{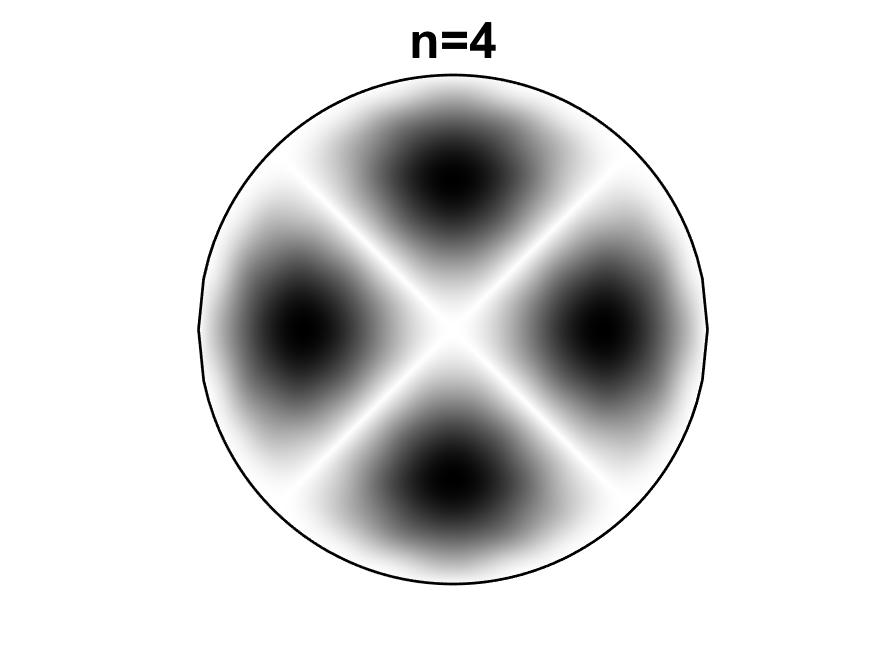}
\includegraphics[ width=.12\linewidth,trim={6cm 2.5cm 5.3cm .5cm},clip]{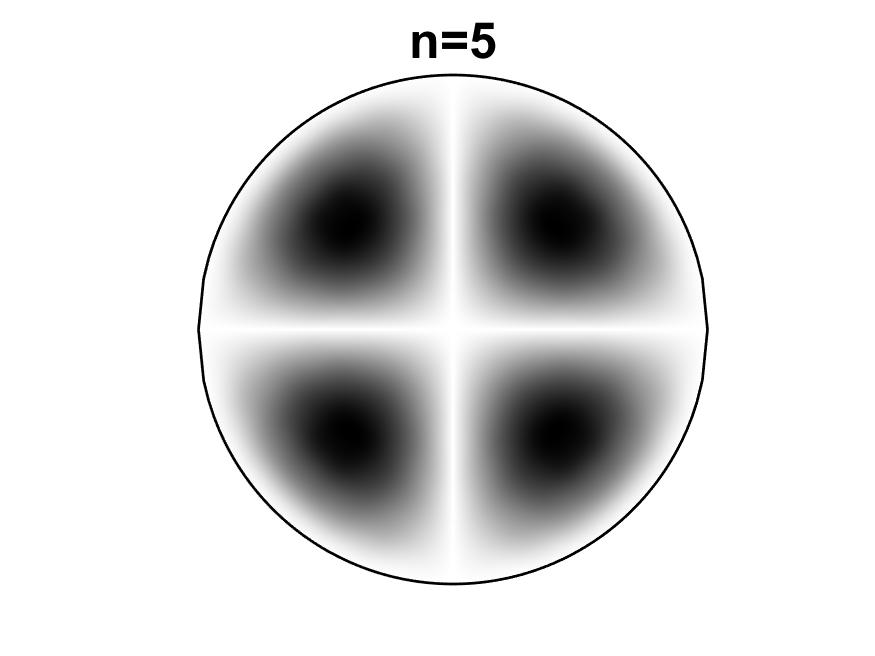}
\includegraphics[ width=.12\linewidth,trim={6cm 2.5cm 5.3cm .5cm},clip]{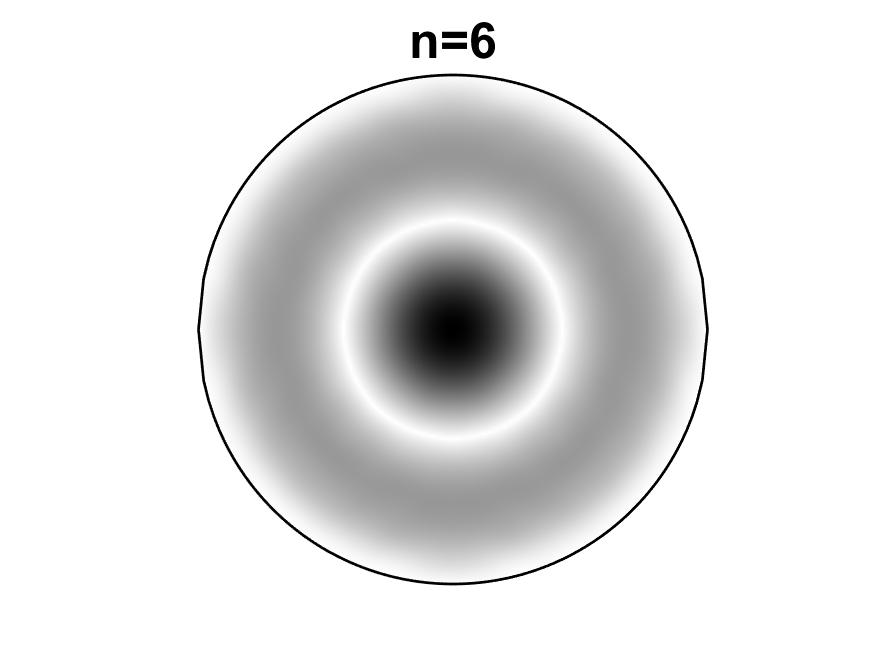} \\
\includegraphics[ width=.12\linewidth,trim={6cm 2.5cm 5.3cm .5cm},clip]{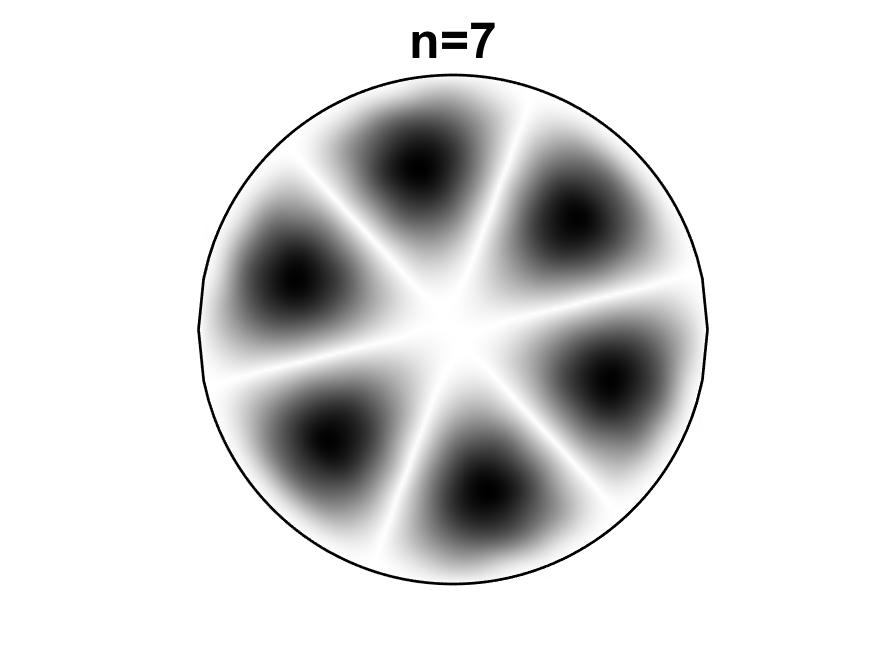}
\includegraphics[ width=.12\linewidth,trim={6cm 2.5cm 5.3cm .5cm},clip]{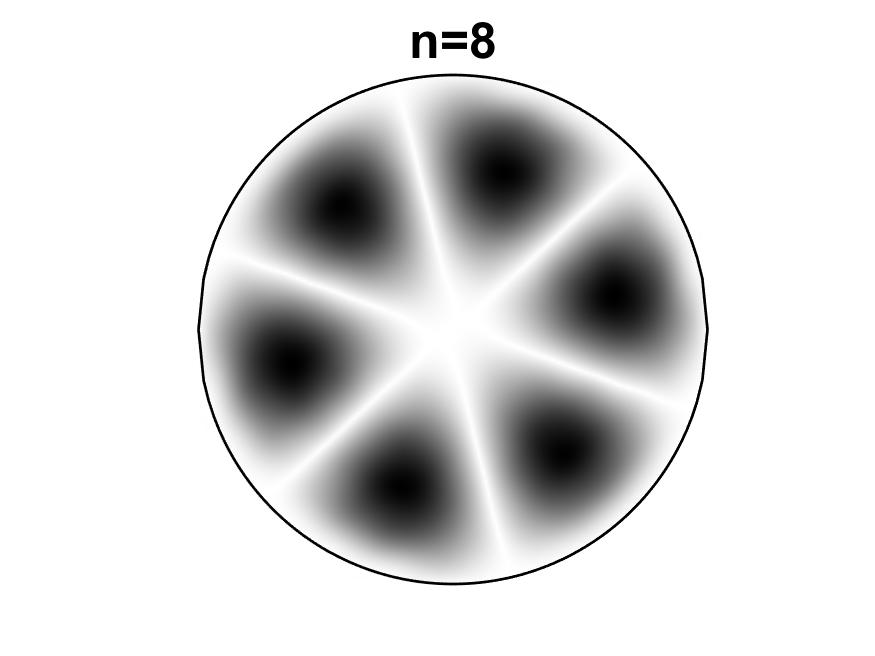}
\includegraphics[ width=.12\linewidth,trim={6cm 2.5cm 5.3cm .5cm},clip]{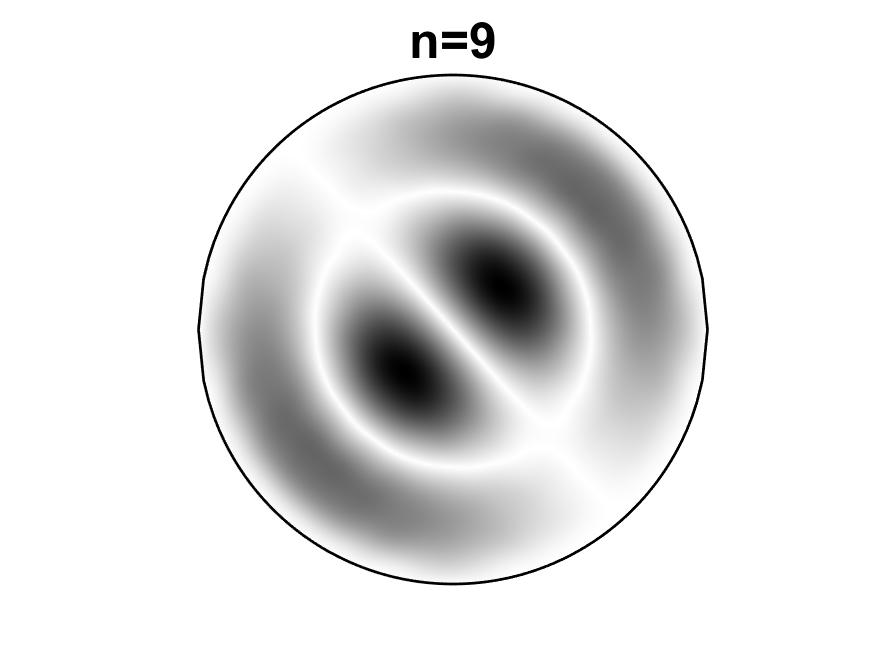}
\includegraphics[ width=.12\linewidth,trim={6cm 2.5cm 5.3cm .5cm},clip]{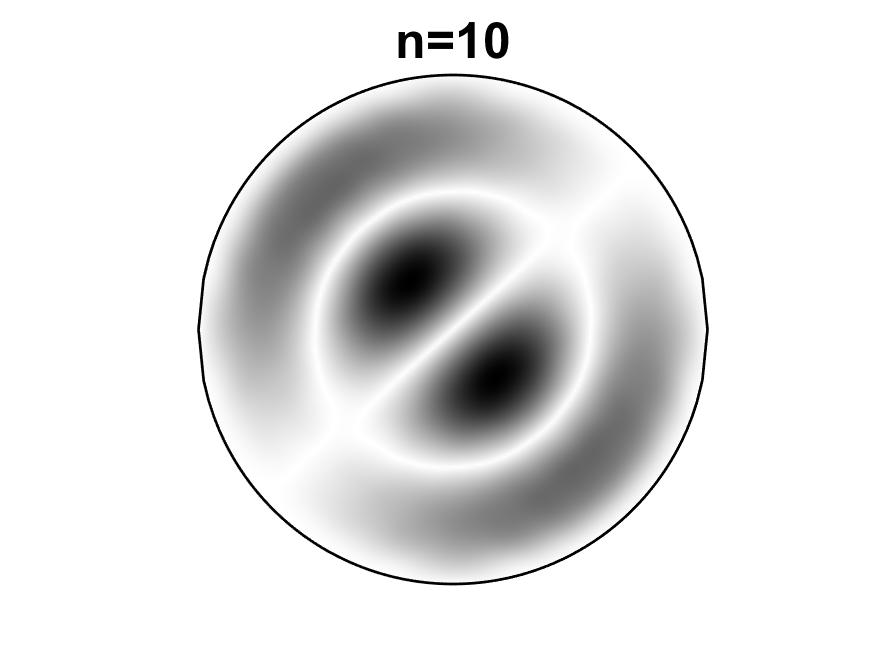}
\includegraphics[ width=.12\linewidth,trim={6cm 2.5cm 5.3cm .5cm},clip]{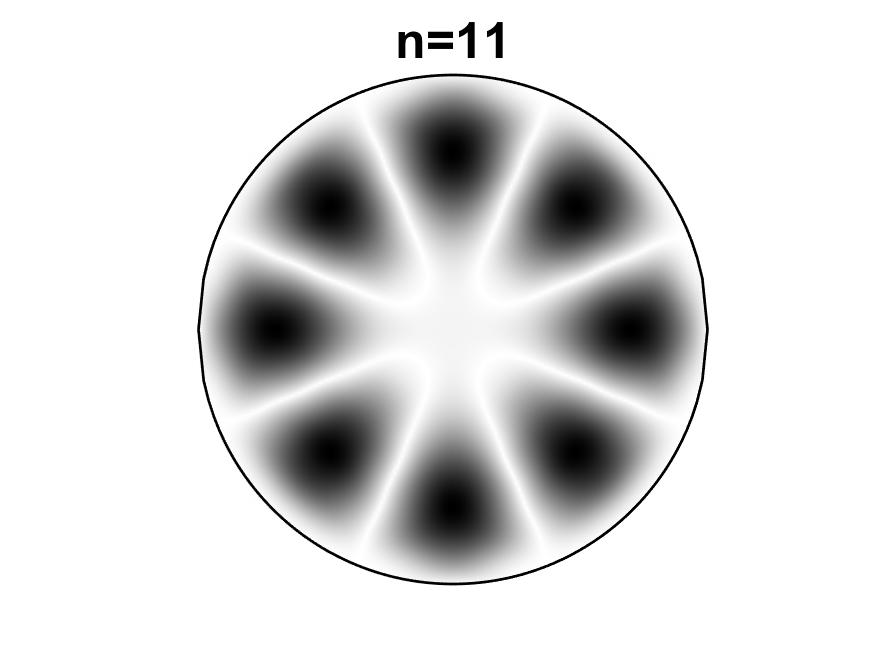}
\includegraphics[ width=.12\linewidth,trim={6cm 2.5cm 5.3cm .5cm},clip]{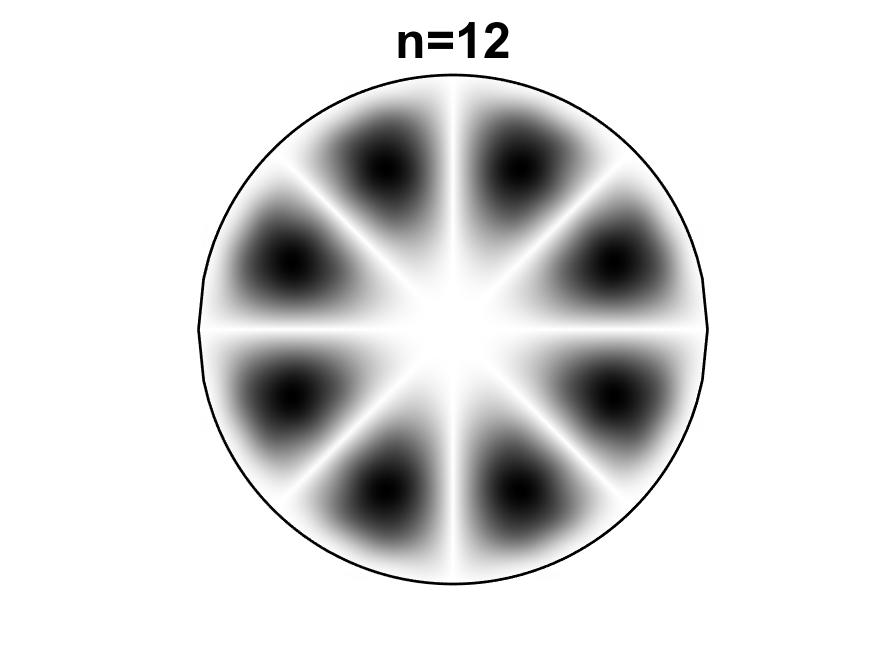}
\caption{The first 12 modes of a half-sphere (above) compared to those of a planar disk (below). Notice the ordering of the degenerate modes does not match.}
\label{fig:sphcircles}
\end{figure}
~~~

\begin{figure}[h!]\centering
\includegraphics[ width=.12\linewidth,trim={6cm 2.5cm 5.3cm .5cm},clip]{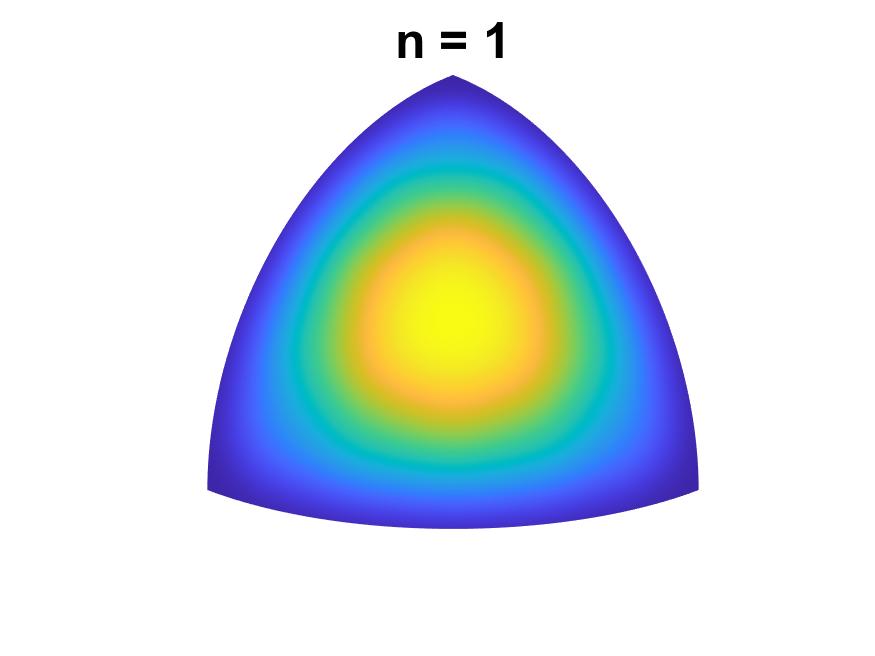}
\includegraphics[ width=.12\linewidth,trim={6cm 2.5cm 5.3cm .5cm},clip]{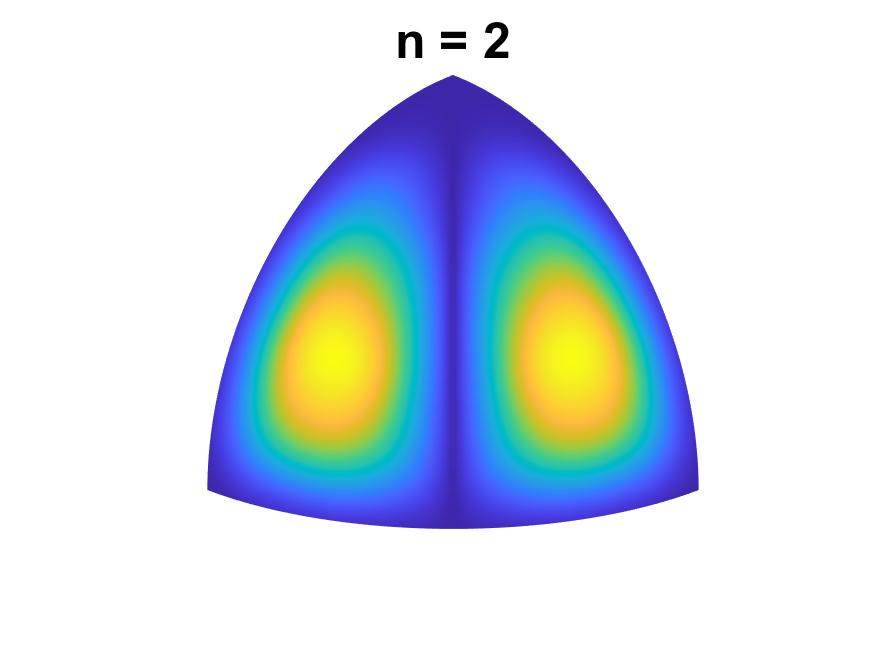}
\includegraphics[ width=.12\linewidth,trim={6cm 2.5cm 5.3cm .5cm},clip]{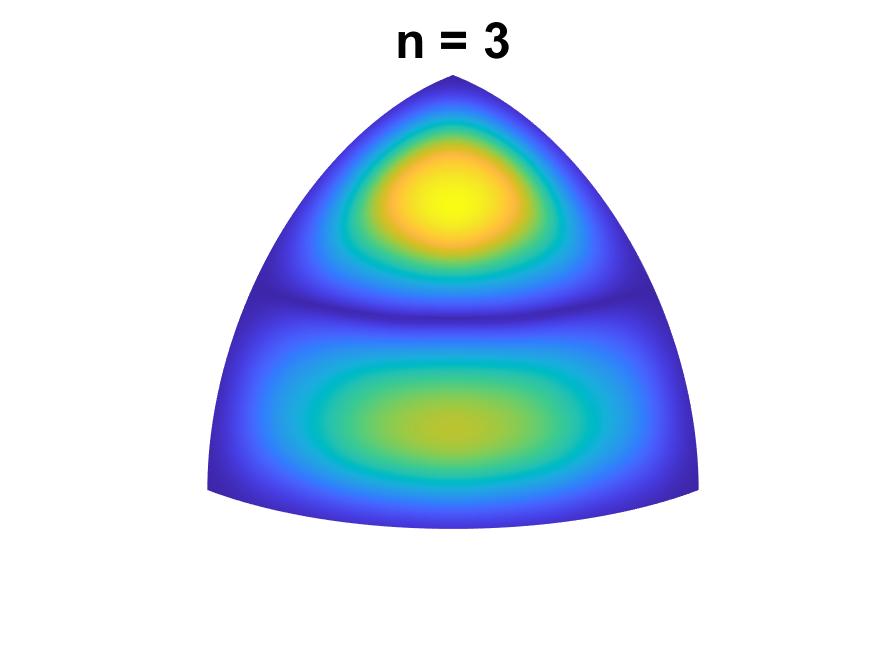}
\includegraphics[ width=.12\linewidth,trim={6cm 2.5cm 5.3cm .5cm},clip]{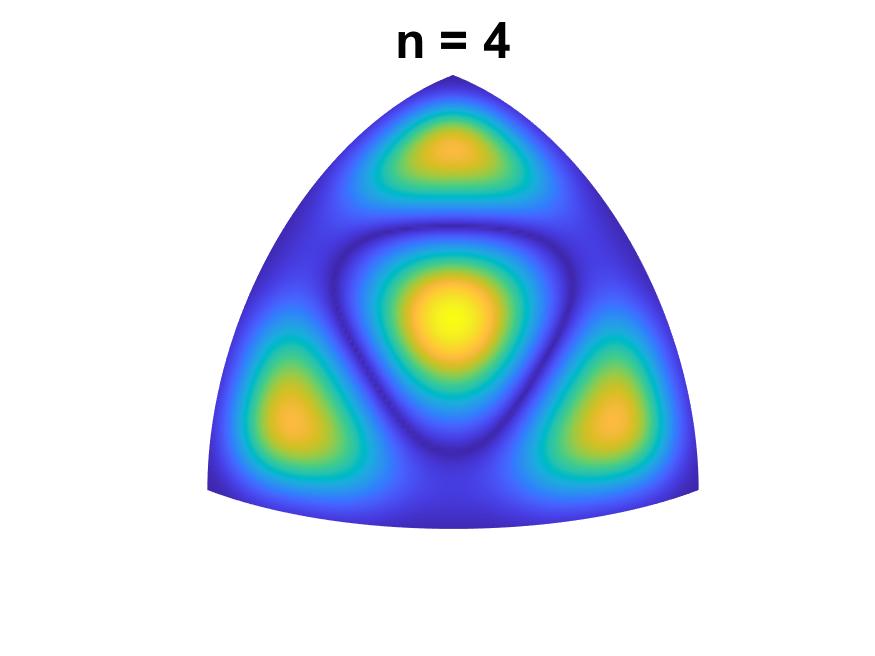}
\includegraphics[ width=.12\linewidth,trim={6cm 2.5cm 5.3cm .5cm},clip]{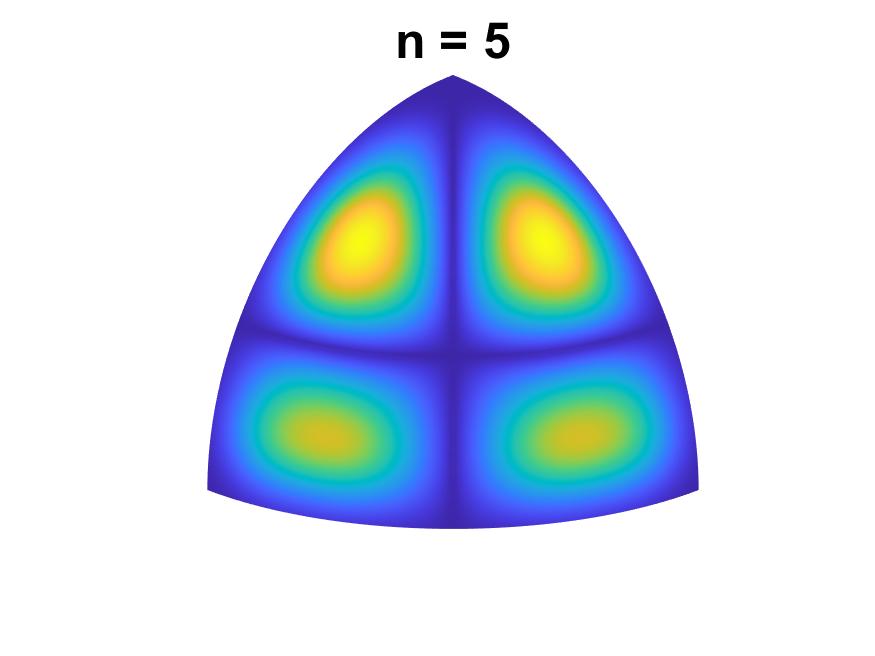}
\includegraphics[ width=.12\linewidth,trim={6cm 2.5cm 5.3cm .5cm},clip]{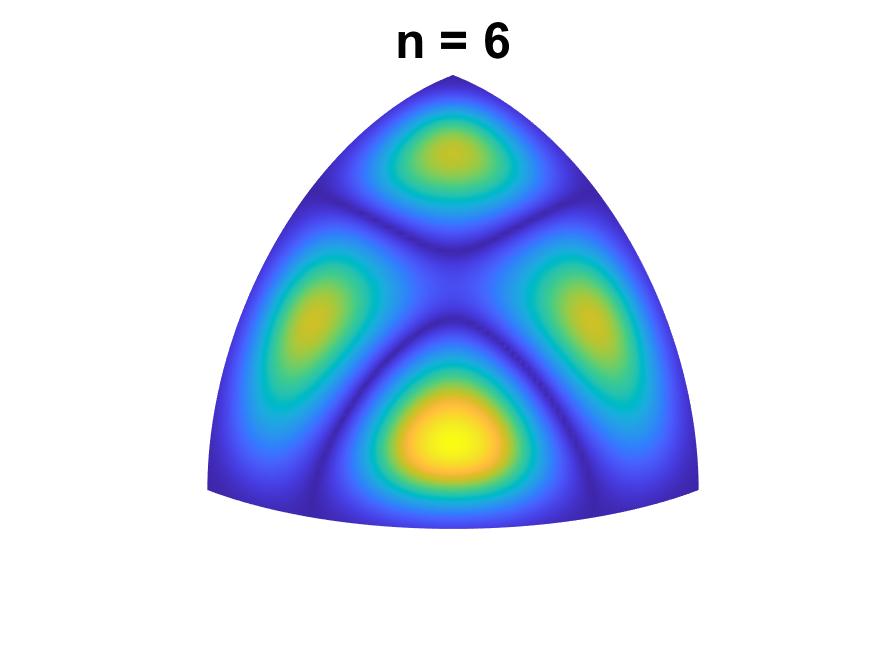} \\
\includegraphics[ width=.12\linewidth,trim={6cm 2.5cm 5.3cm .5cm},clip]{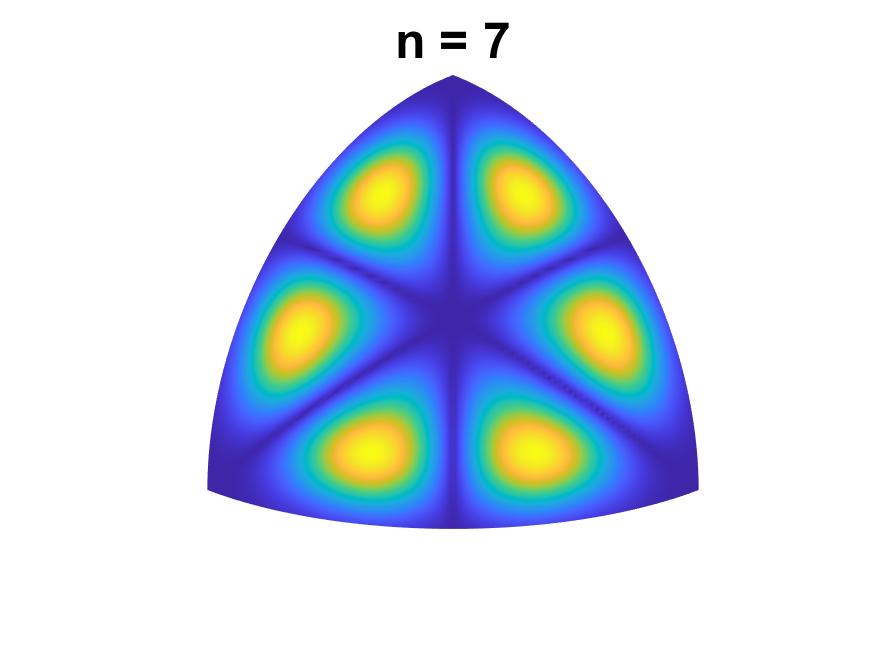}
\includegraphics[ width=.12\linewidth,trim={6cm 2.5cm 5.3cm .5cm},clip]{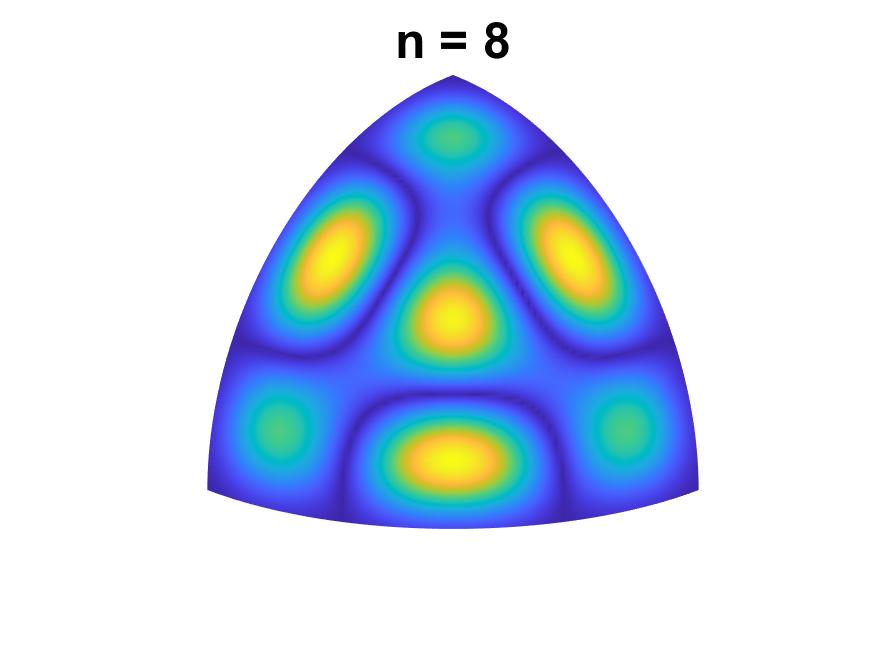}
\includegraphics[ width=.12\linewidth,trim={6cm 2.5cm 5.3cm .5cm},clip]{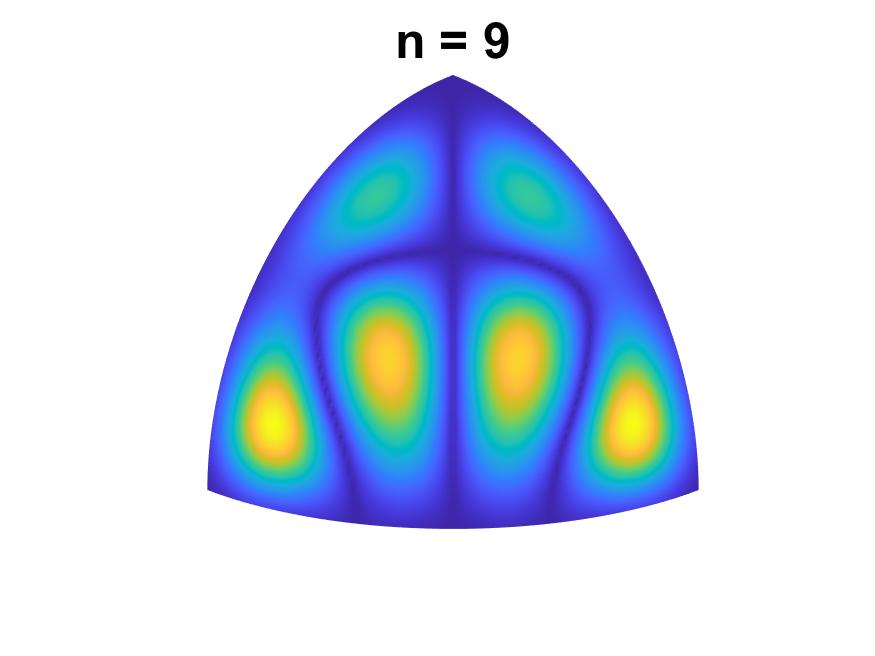}
\includegraphics[ width=.12\linewidth,trim={6cm 2.5cm 5.3cm .5cm},clip]{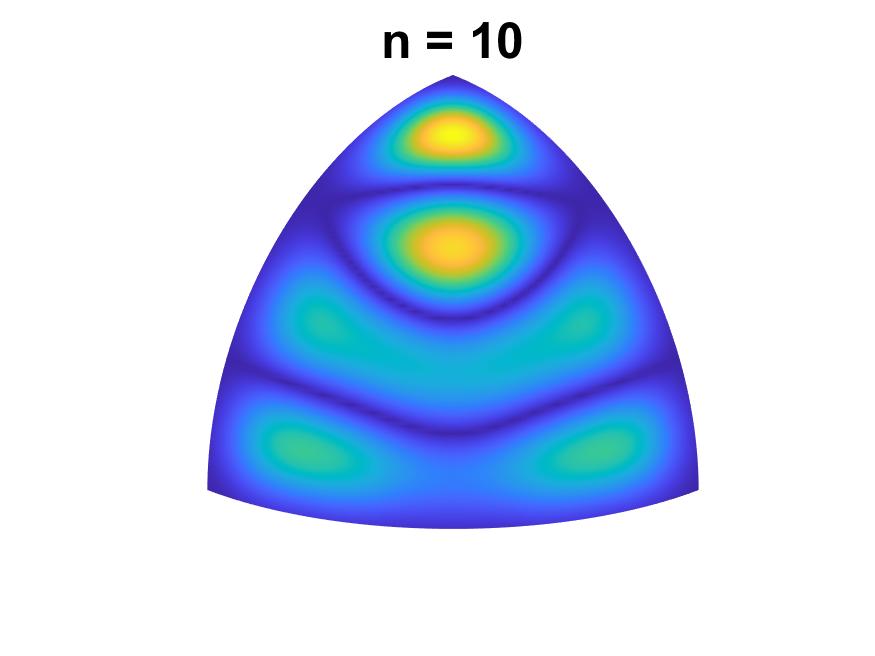}
\includegraphics[ width=.12\linewidth,trim={6cm 2.5cm 5.3cm .5cm},clip]{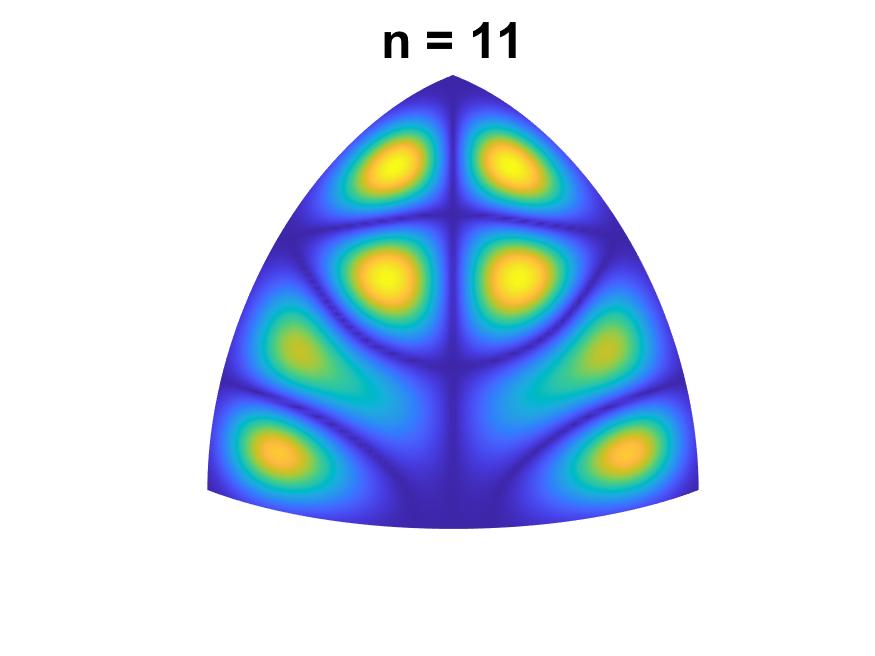}
\includegraphics[ width=.12\linewidth,trim={6cm 2.5cm 5.3cm .5cm},clip]{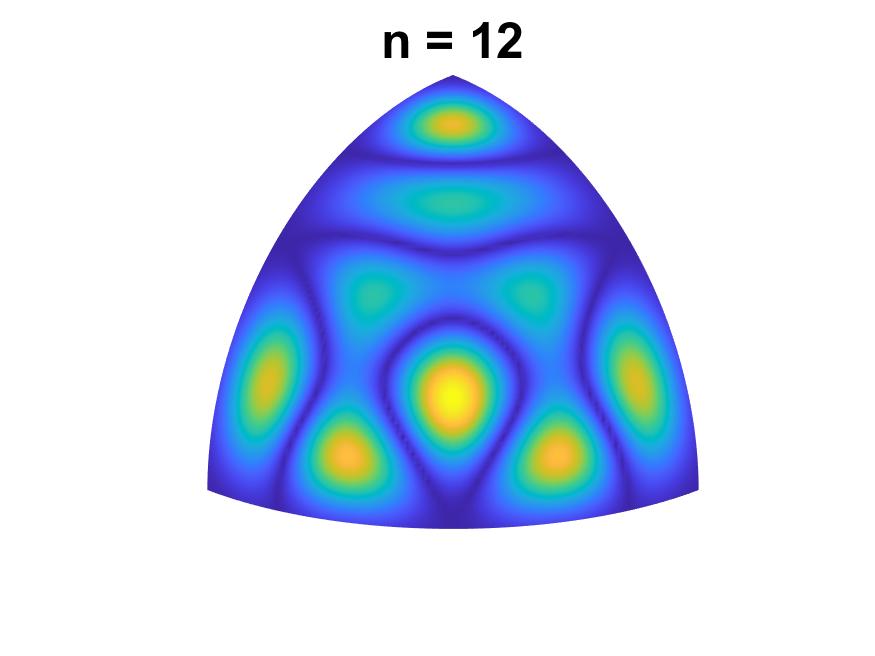} \\
\includegraphics[ width=.12\linewidth,trim={4cm 0cm 3.2cm 0cm},clip]{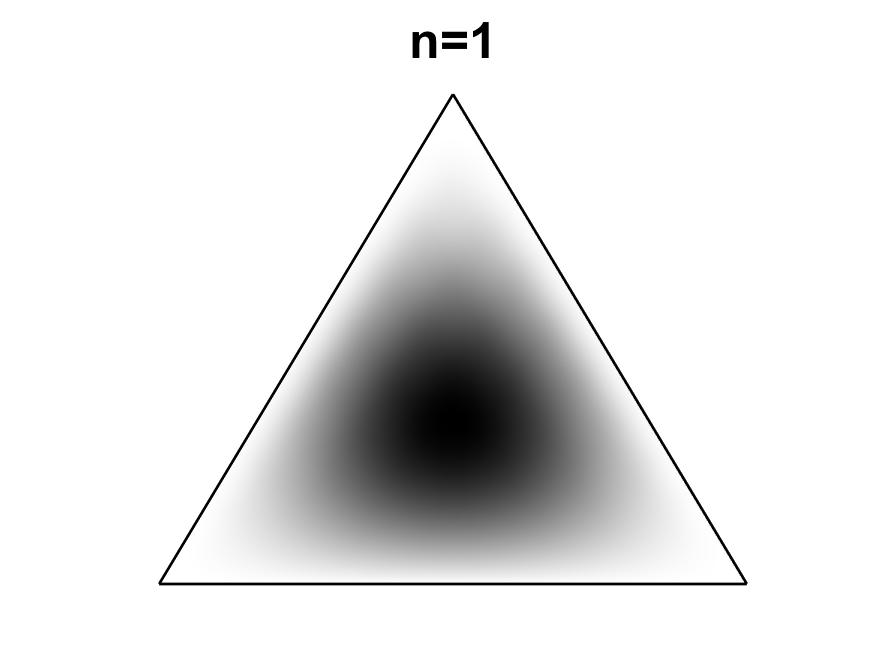}
\includegraphics[ width=.12\linewidth,trim={4cm 0cm 3.2cm 0cm},clip]{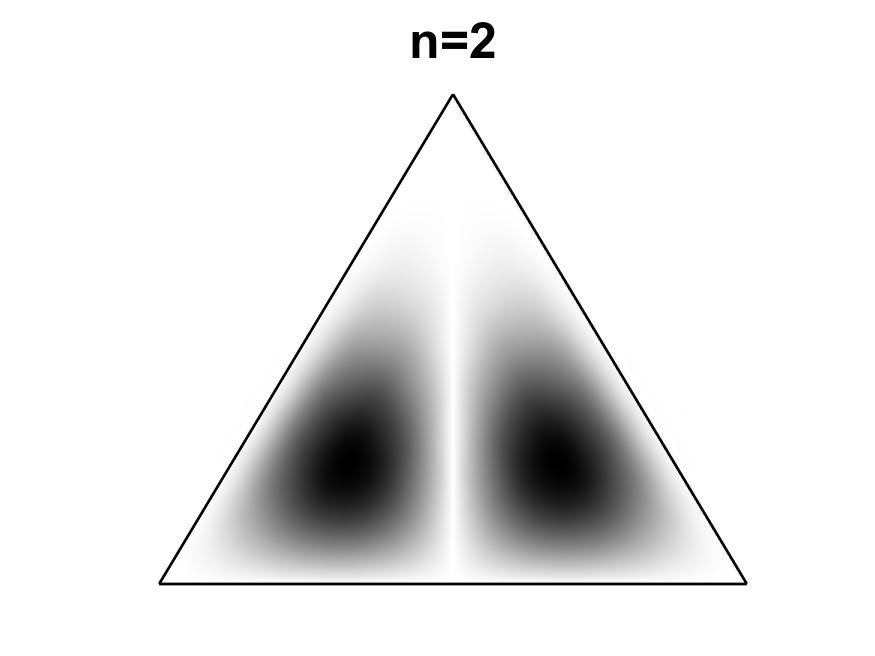}
\includegraphics[ width=.12\linewidth,trim={4cm 0cm 3.2cm 0cm},clip]{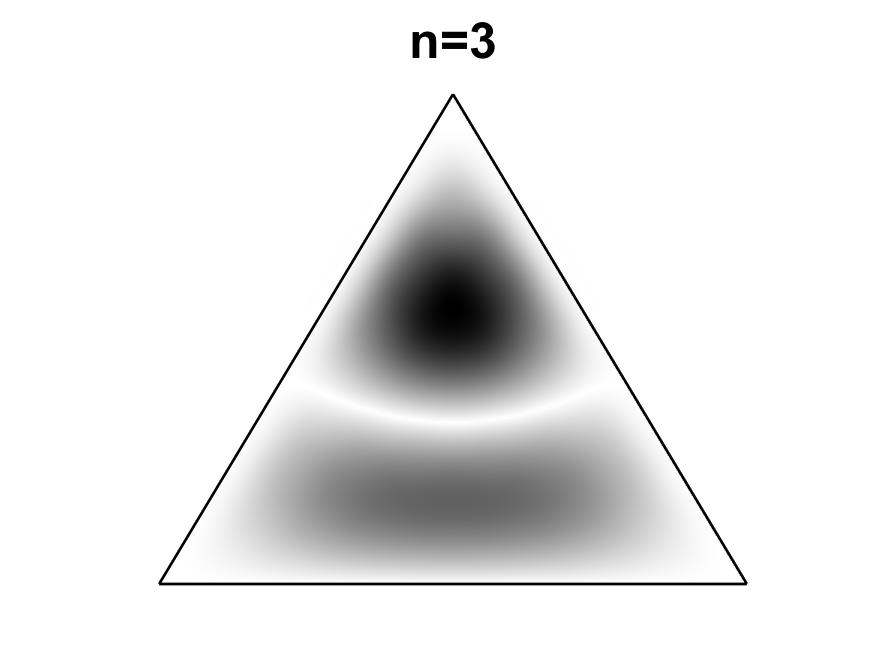}
\includegraphics[ width=.12\linewidth,trim={4cm 0cm 3.2cm 0cm},clip]{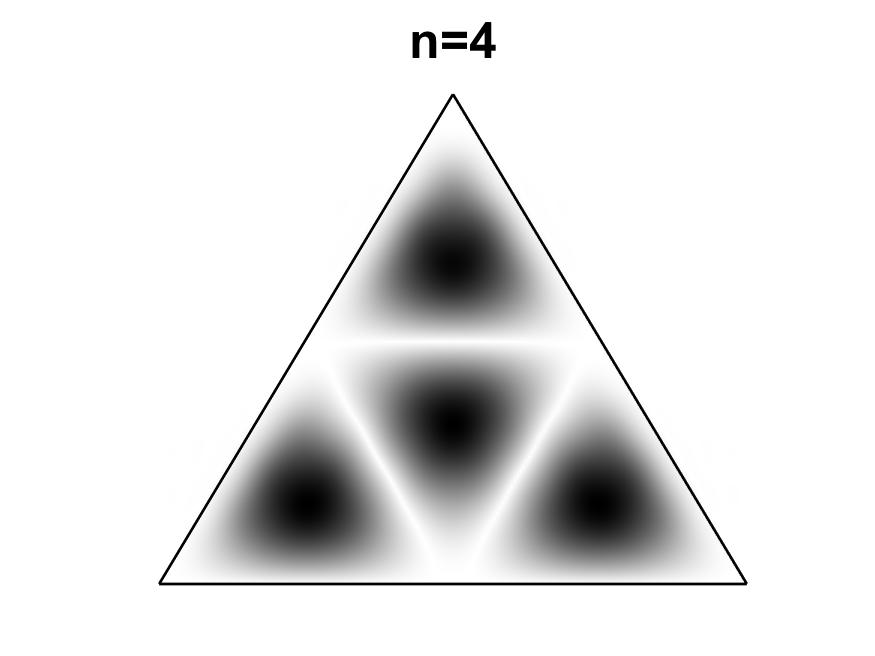}
\includegraphics[ width=.12\linewidth,trim={4cm 0cm 3.2cm 0cm},clip]{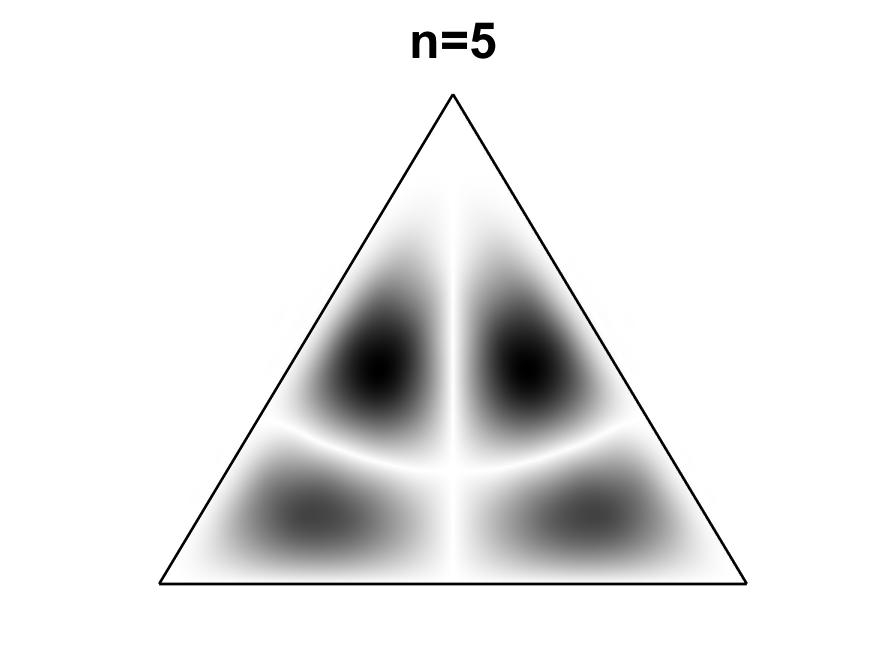}
\includegraphics[ width=.12\linewidth,trim={4cm 0cm 3.2cm 0cm},clip]{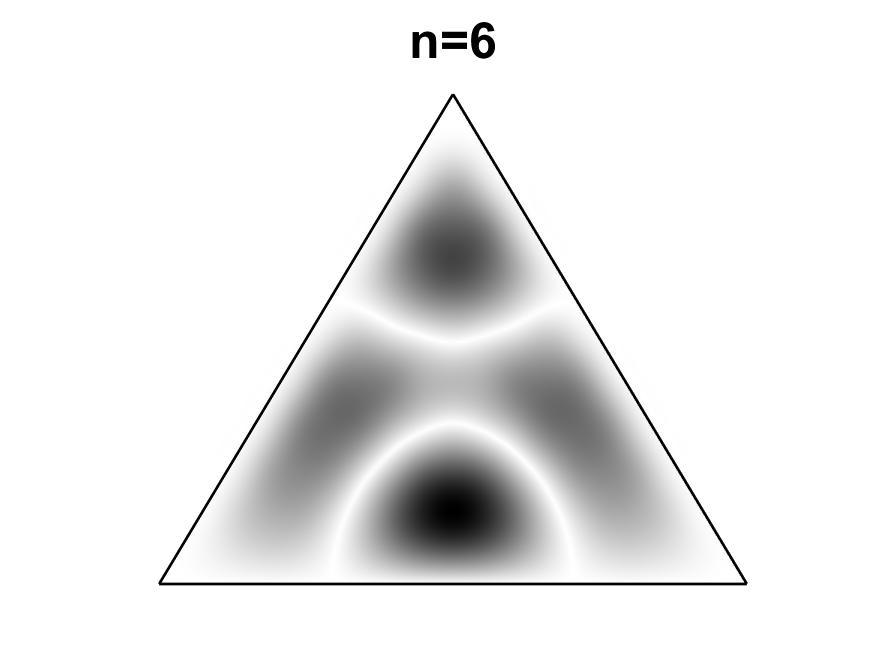} \\
\includegraphics[ width=.12\linewidth,trim={4cm 0cm 3.2cm 0cm},clip]{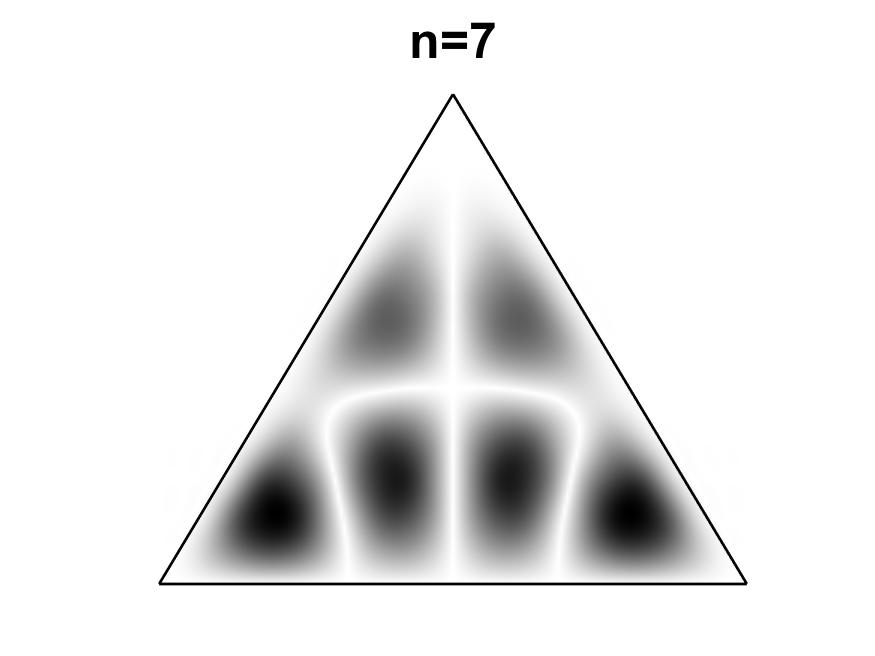}
\includegraphics[ width=.12\linewidth,trim={4cm 0cm 3.2cm 0cm},clip]{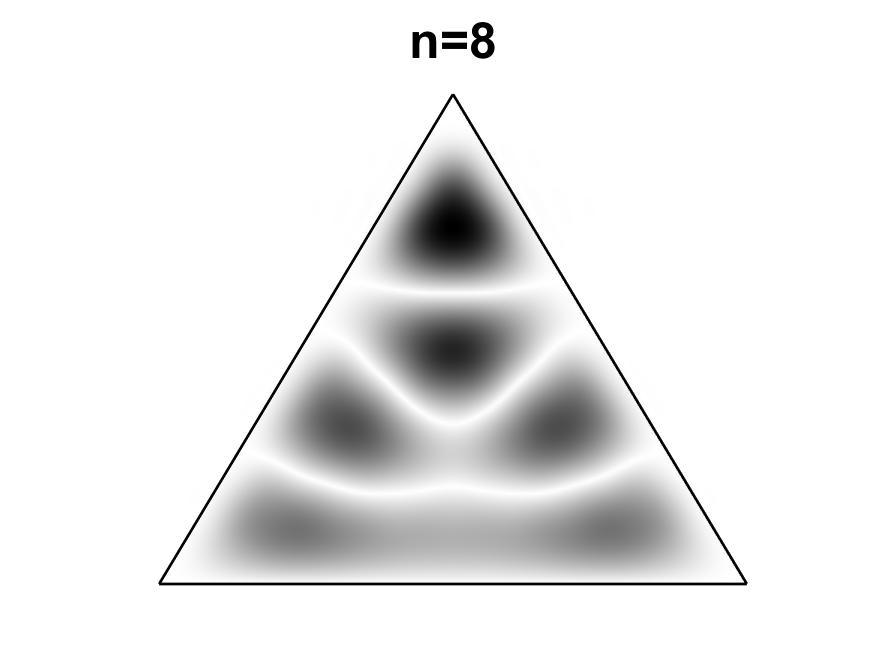}
\includegraphics[ width=.12\linewidth,trim={4cm 0cm 3.2cm 0cm},clip]{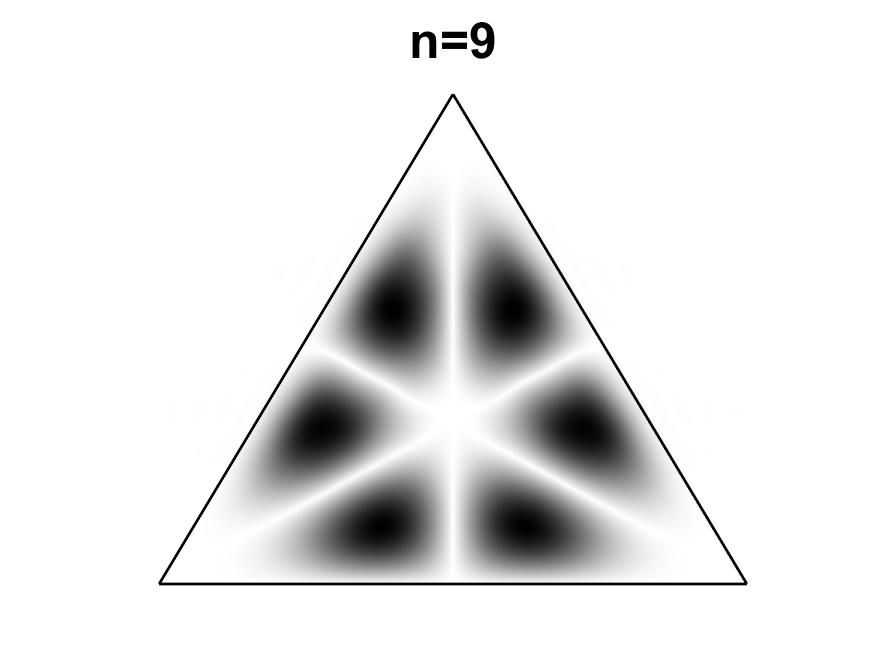}
\includegraphics[ width=.12\linewidth,trim={4cm 0cm 3.2cm 0cm},clip]{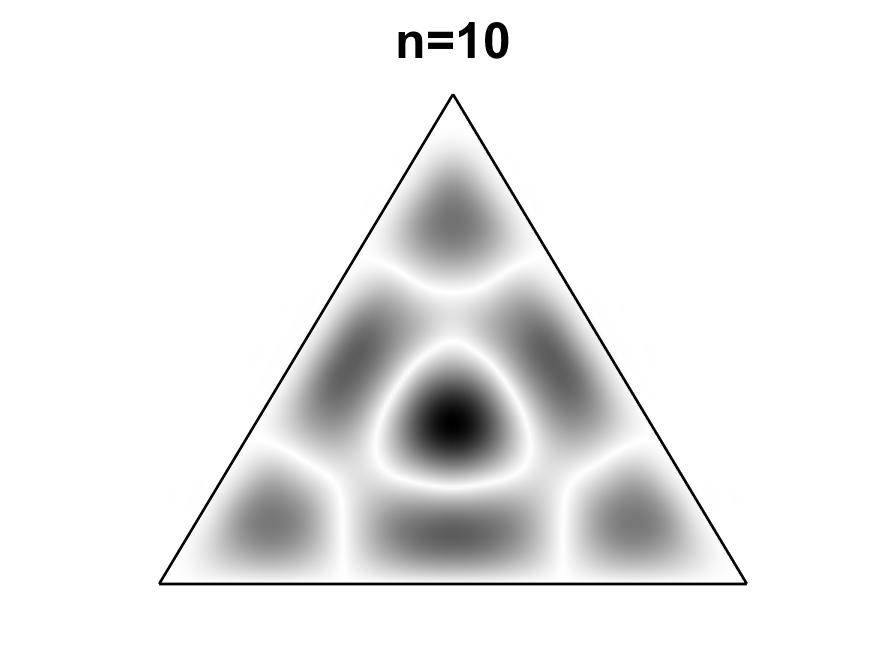}
\includegraphics[ width=.12\linewidth,trim={4cm 0cm 3.2cm 0cm},clip]{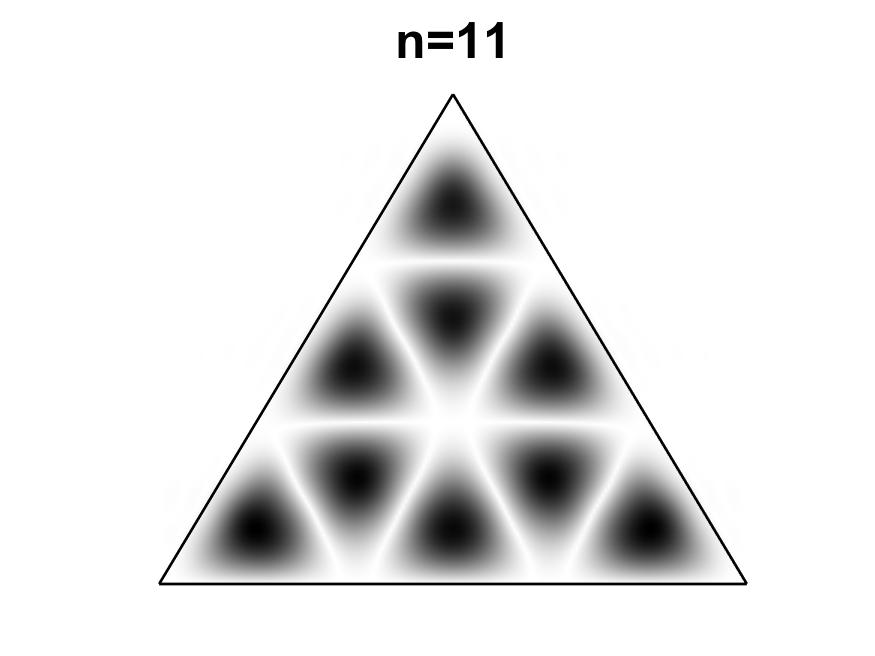}
\includegraphics[ width=.12\linewidth,trim={4cm 0cm 3.2cm 0cm},clip]{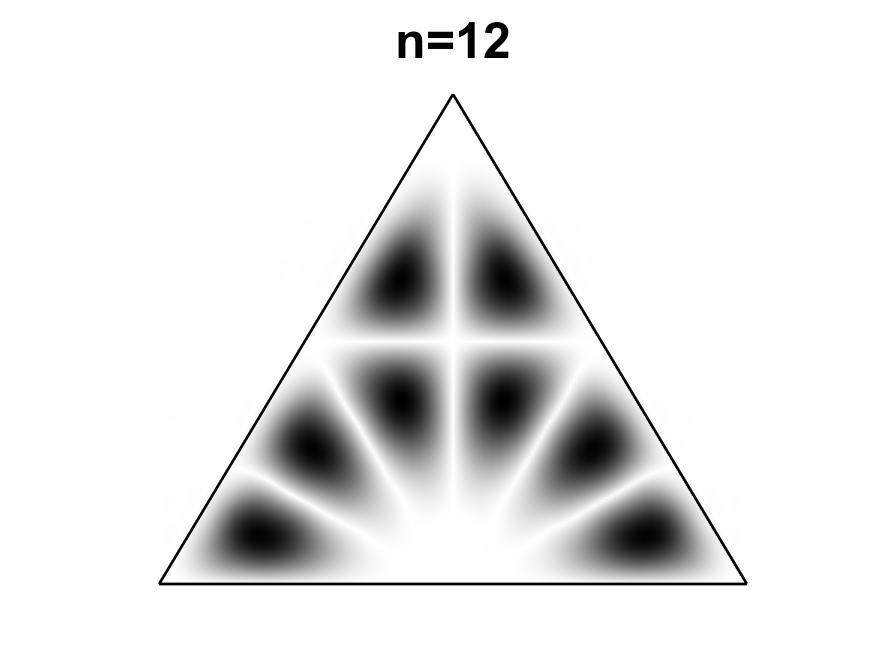}
\caption{The first 12 modes of the spherical octant (above) compared with those of an equilateral triangle (below). Notice the ordering of the degenerate modes does not match.} \label{fig:sphtri}
\end{figure}

\begin{figure}[h!]\centering
\includegraphics[ width=.12\linewidth,trim={6cm 2.5cm 5.3cm .5cm},clip]{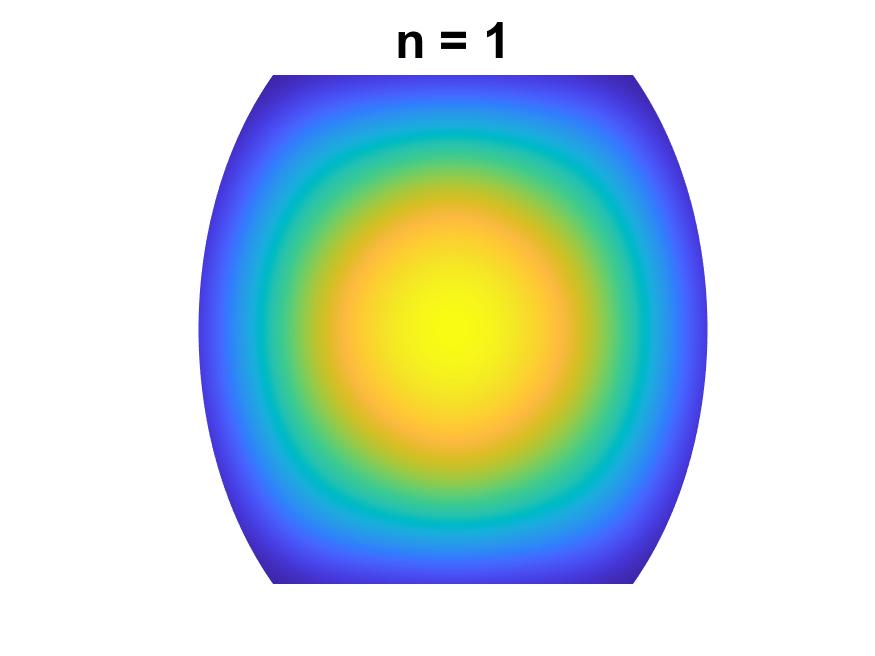}
\includegraphics[ width=.12\linewidth,trim={6cm 2.5cm 5.3cm .5cm},clip]{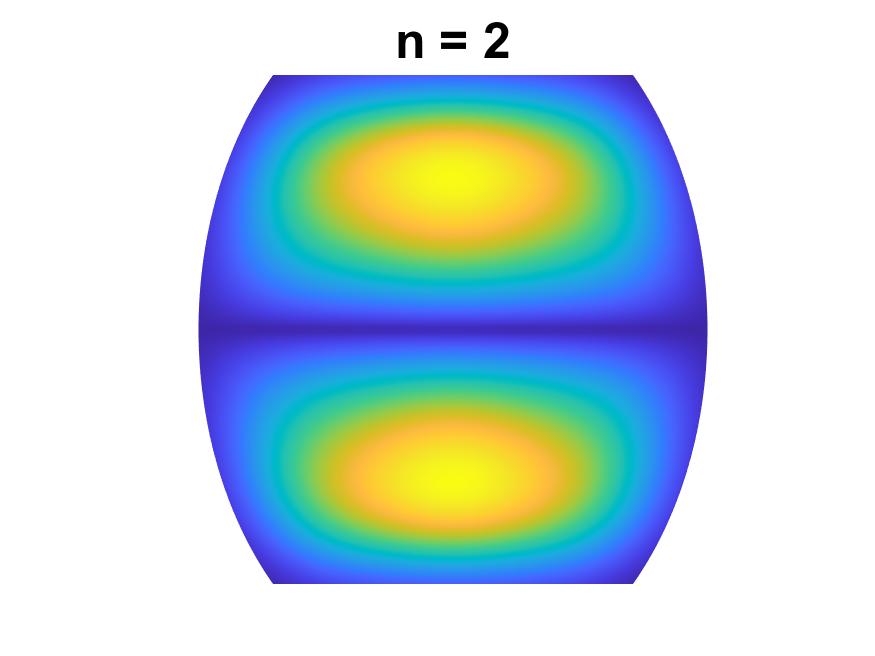}
\includegraphics[ width=.12\linewidth,trim={6cm 2.5cm 5.3cm .5cm},clip]{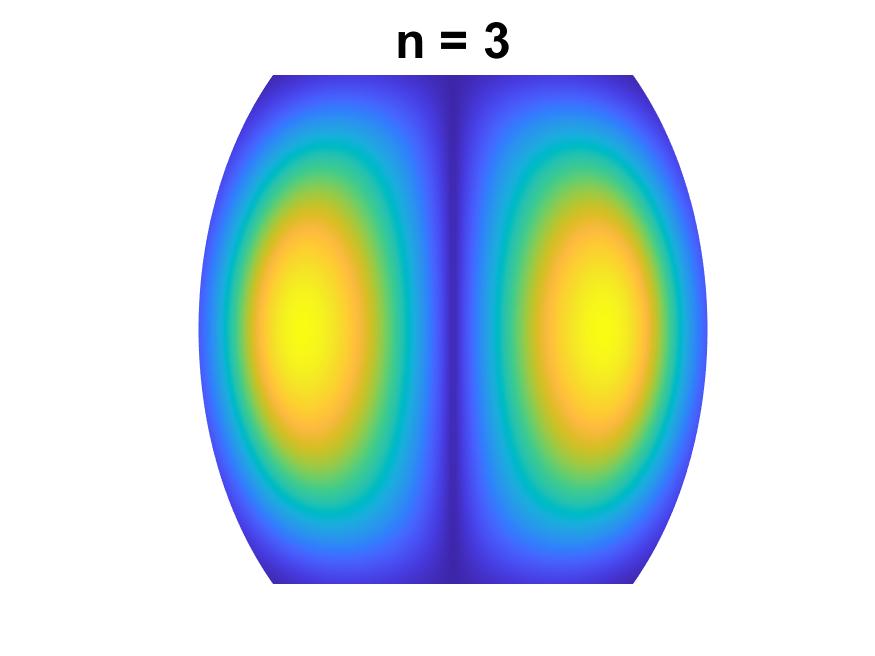}
\includegraphics[ width=.12\linewidth,trim={6cm 2.5cm 5.3cm .5cm},clip]{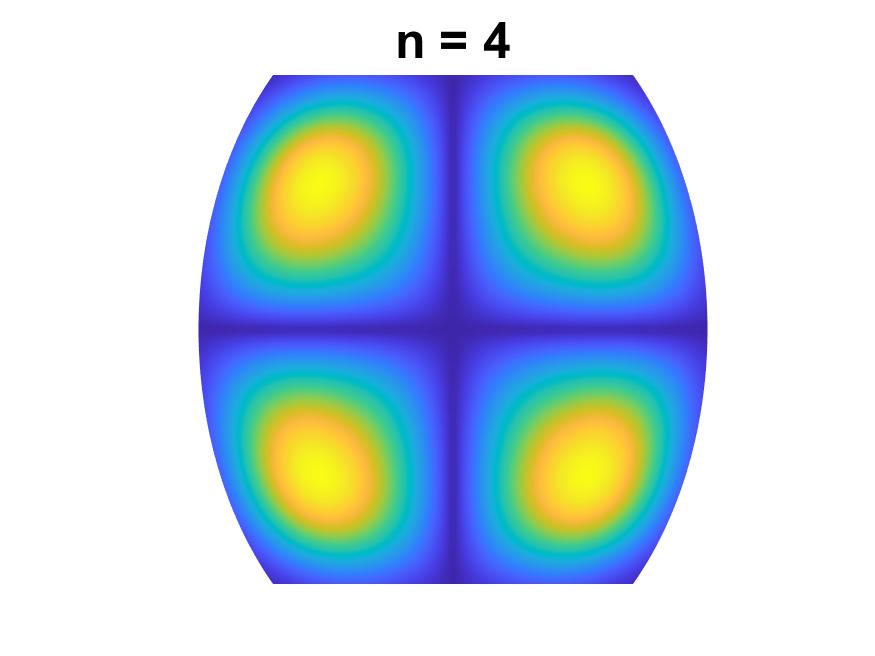}
\includegraphics[ width=.12\linewidth,trim={6cm 2.5cm 5.3cm .5cm},clip]{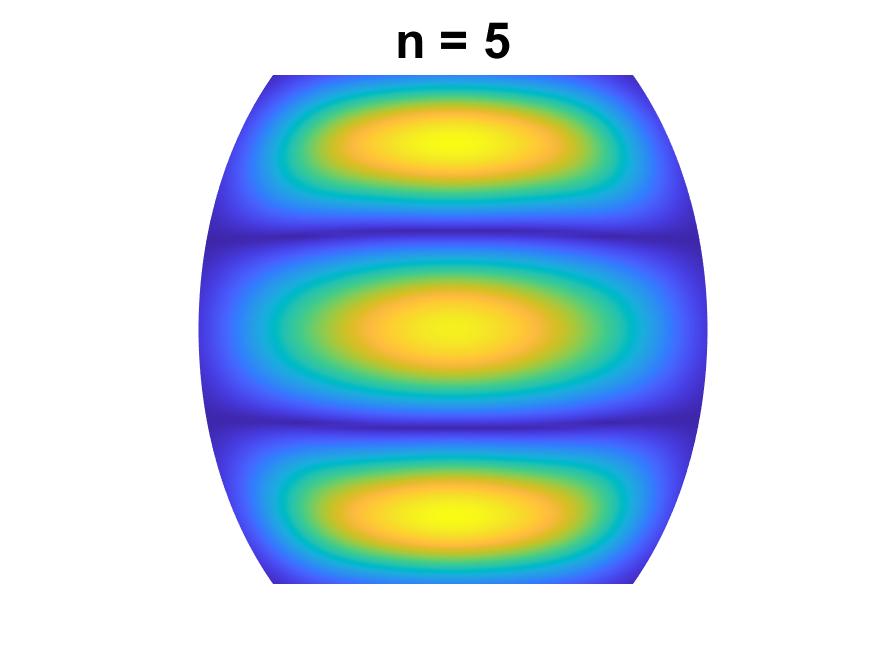}
\includegraphics[ width=.12\linewidth,trim={6cm 2.5cm 5.3cm .5cm},clip]{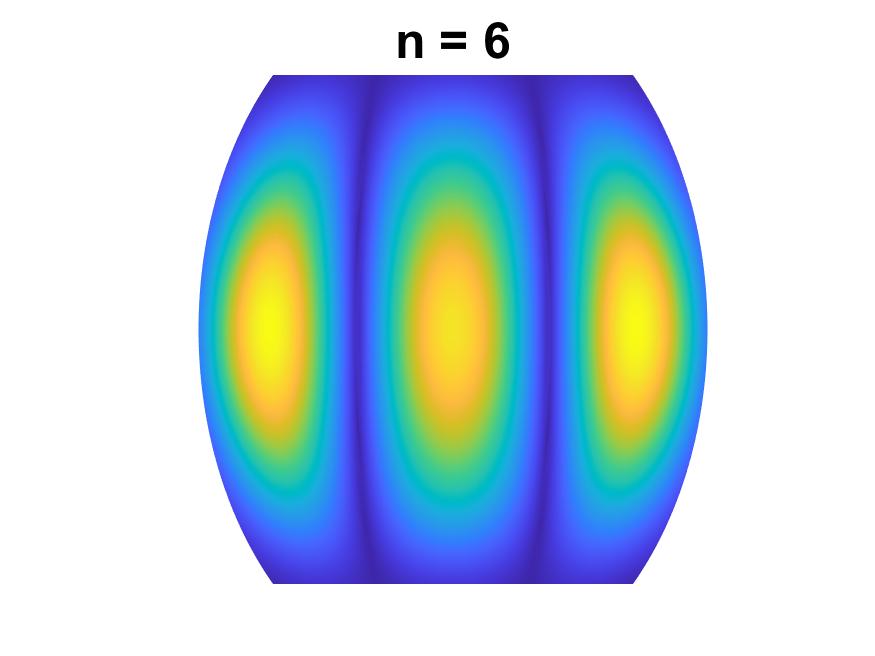} \\
\includegraphics[ width=.12\linewidth,trim={6cm 2.5cm 5.3cm .5cm},clip]{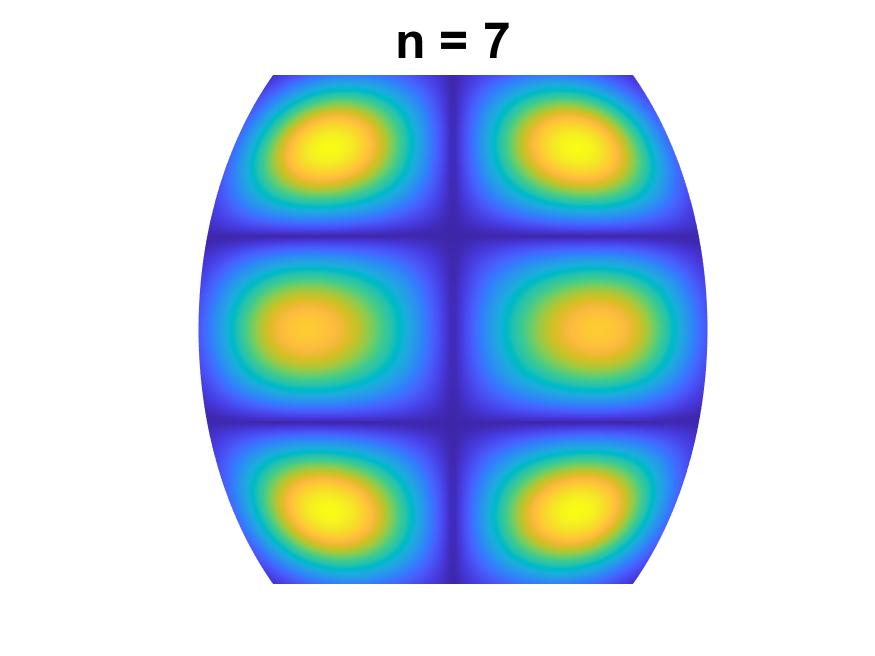}
\includegraphics[ width=.12\linewidth,trim={6cm 2.5cm 5.3cm .5cm},clip]{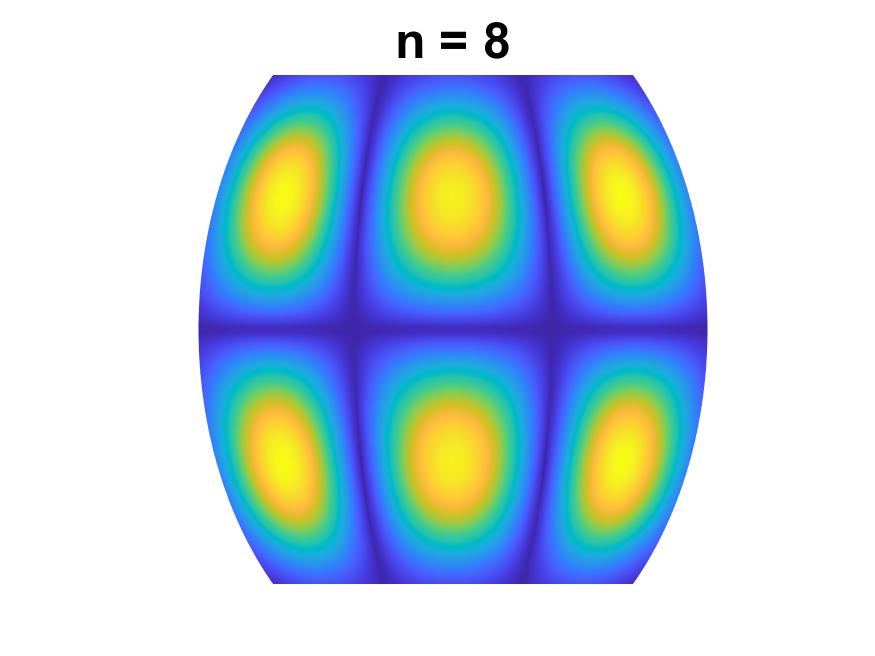}
\includegraphics[ width=.12\linewidth,trim={6cm 2.5cm 5.3cm .5cm},clip]{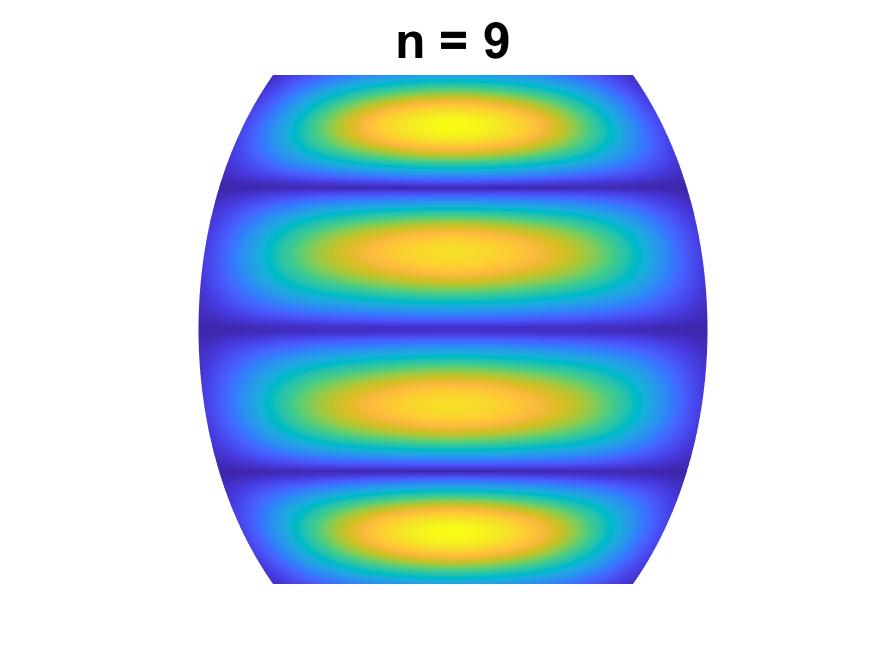}
\includegraphics[ width=.12\linewidth,trim={6cm 2.5cm 5.3cm .5cm},clip]{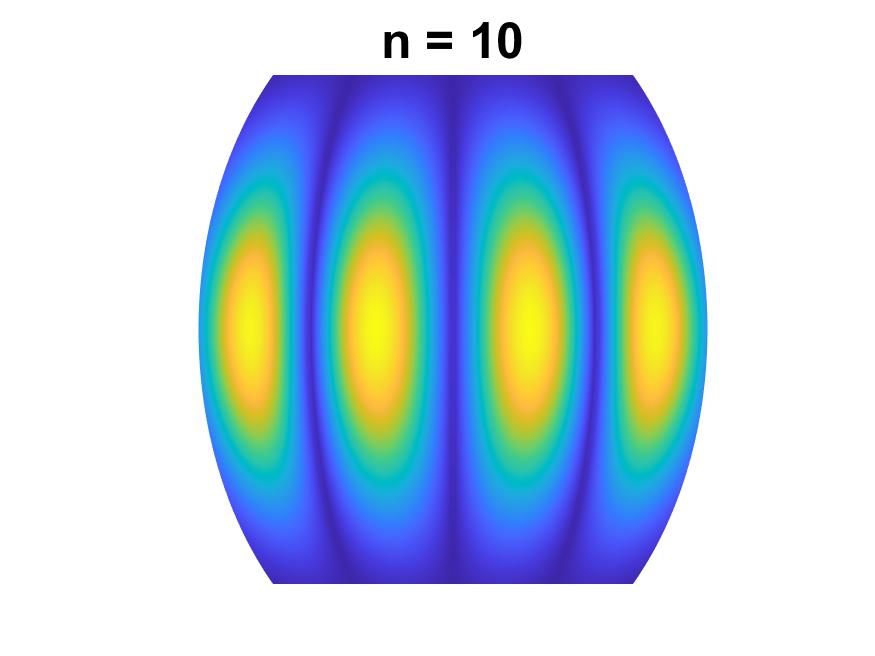}
\includegraphics[ width=.12\linewidth,trim={6cm 2.5cm 5.3cm .5cm},clip]{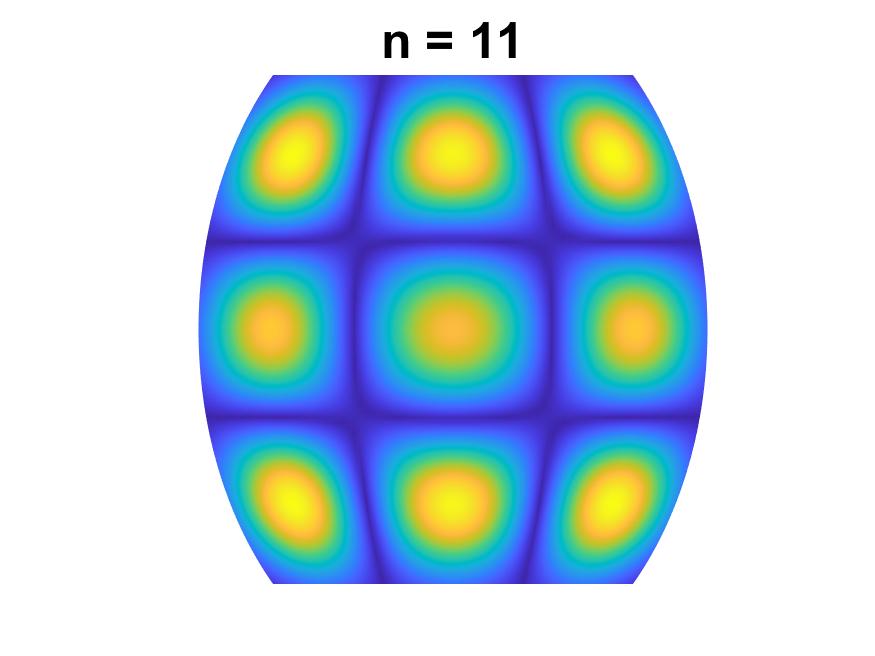}
\includegraphics[ width=.12\linewidth,trim={6cm 2.5cm 5.3cm .5cm},clip]{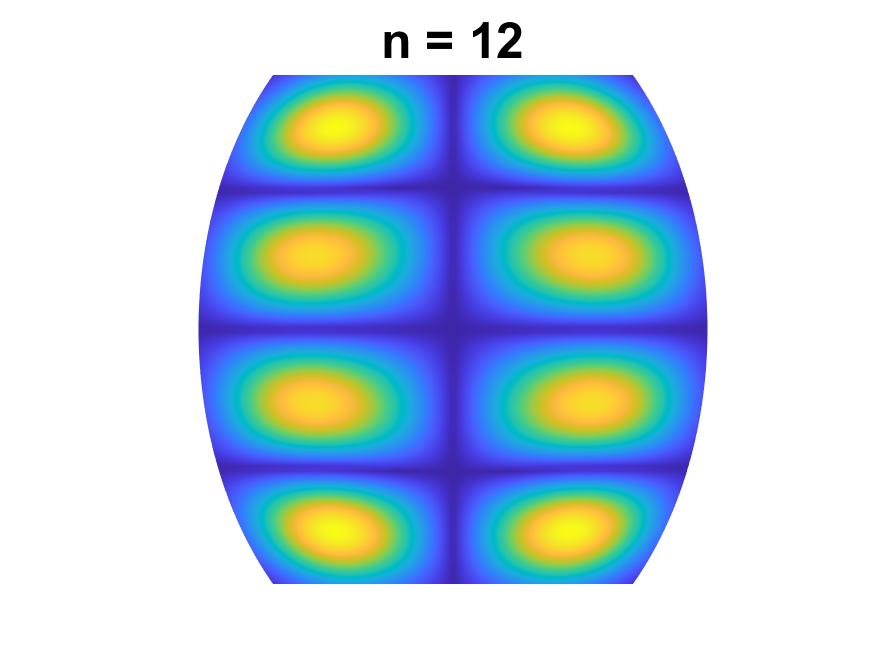} \\
\includegraphics[ width=.12\linewidth,trim={6cm 2.5cm 5.3cm .5cm},clip]{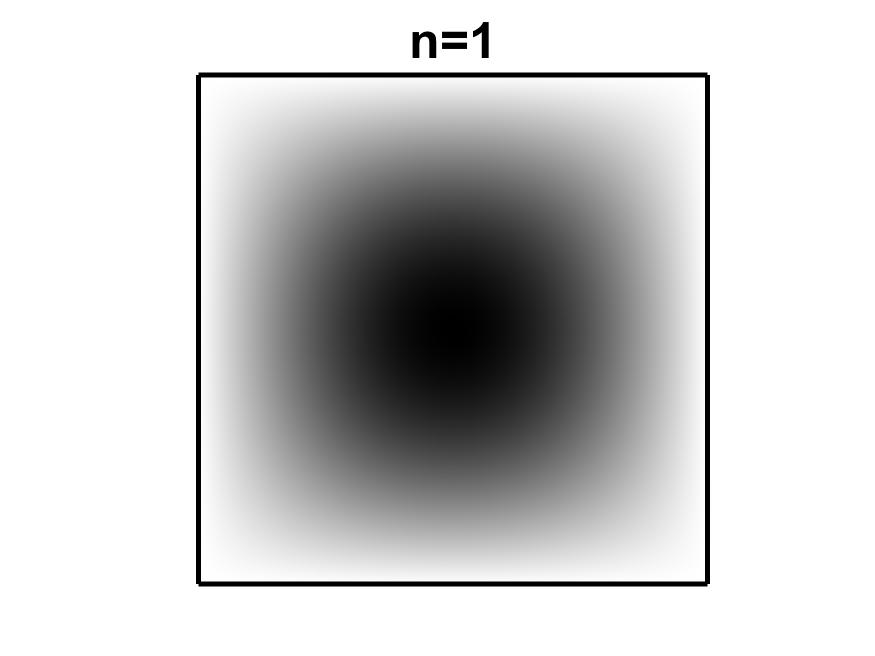}
\includegraphics[ width=.12\linewidth,trim={6cm 2.5cm 5.3cm .5cm},clip]{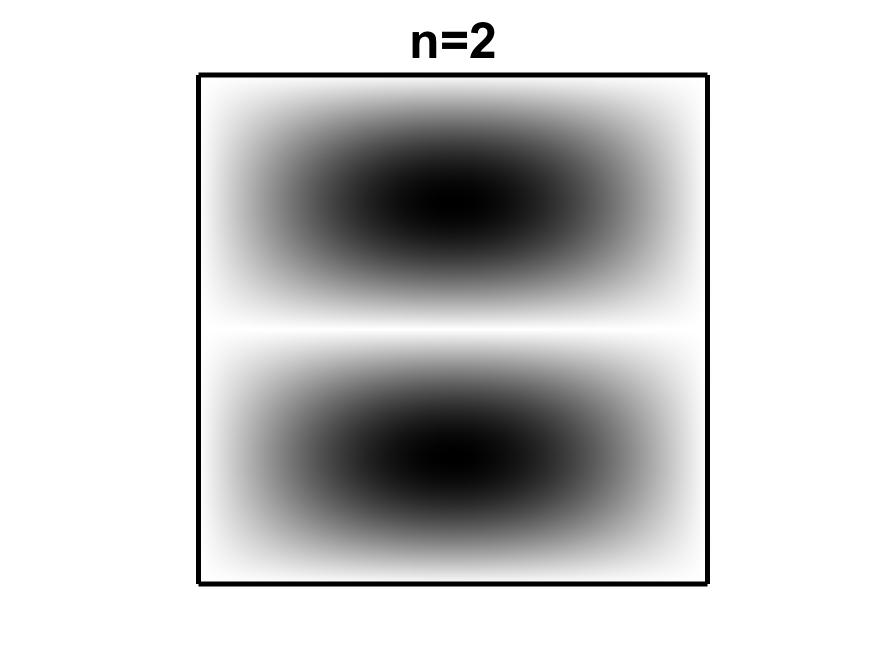}
\includegraphics[ width=.12\linewidth,trim={6cm 2.5cm 5.3cm .5cm},clip]{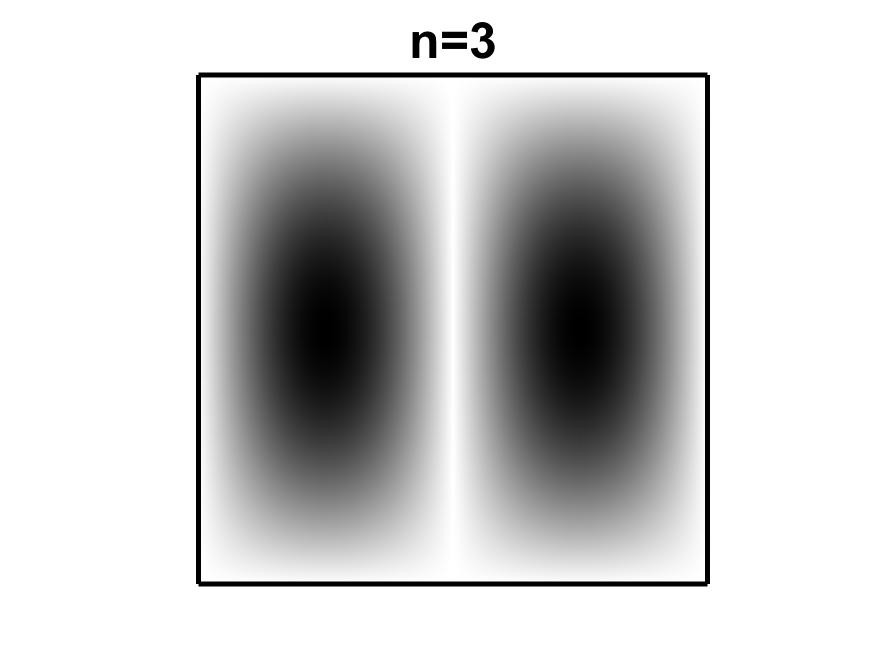}
\includegraphics[ width=.12\linewidth,trim={6cm 2.5cm 5.3cm .5cm},clip]{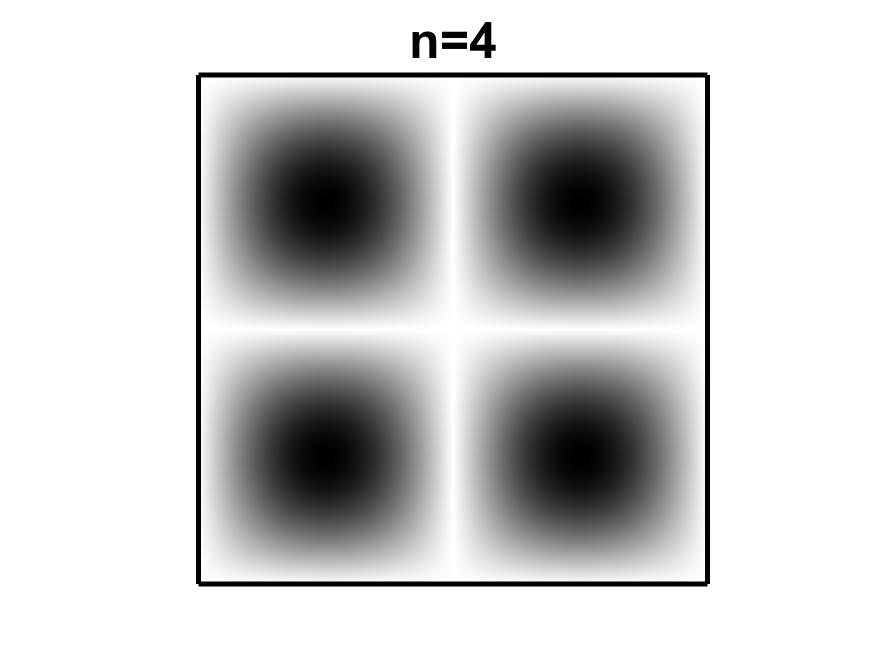}
\includegraphics[ width=.12\linewidth,trim={6cm 2.5cm 5.3cm .5cm},clip]{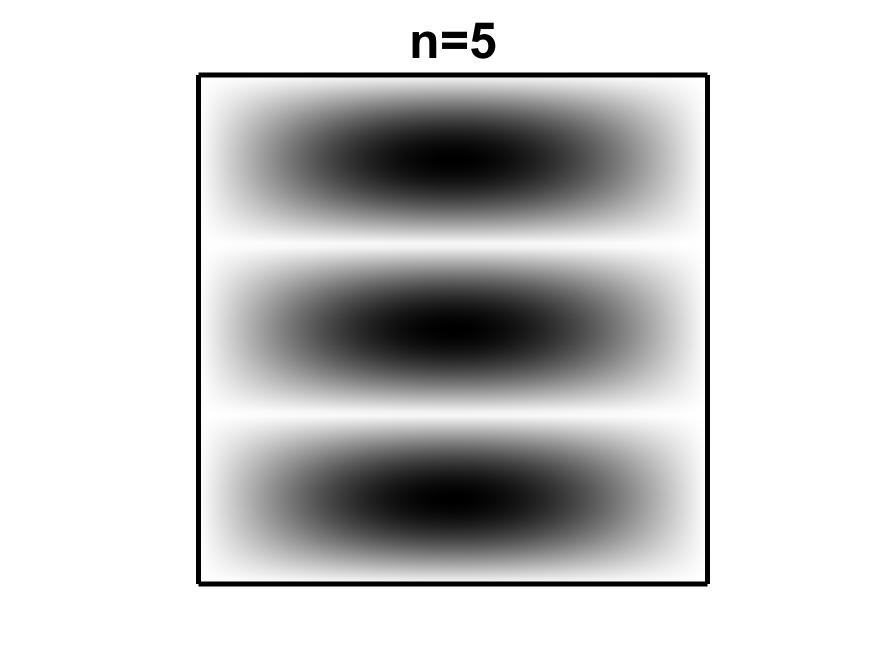}
\includegraphics[ width=.12\linewidth,trim={6cm 2.5cm 5.3cm .5cm},clip]{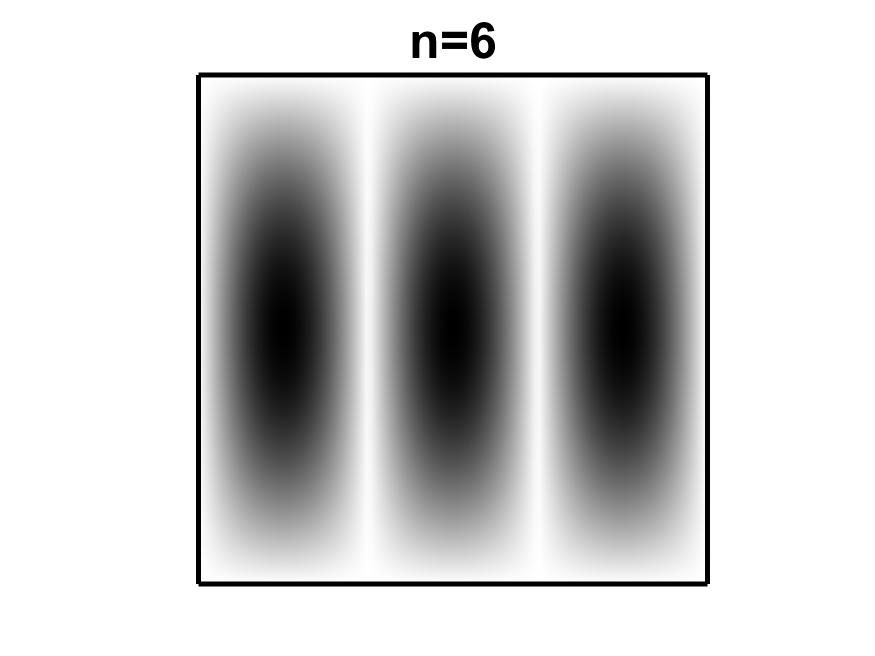} \\
\includegraphics[ width=.12\linewidth,trim={6cm 2.5cm 5.3cm .5cm},clip]{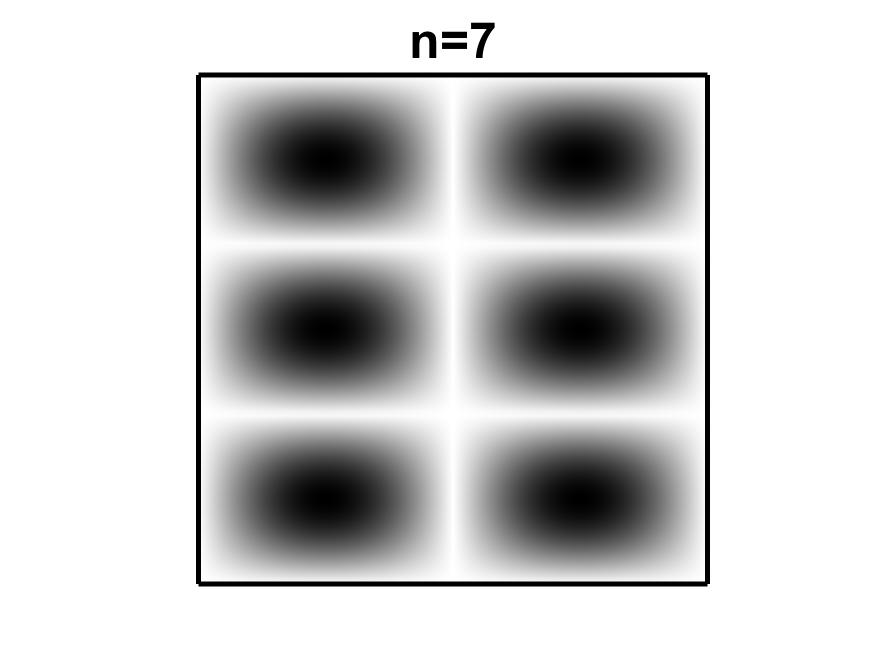}
\includegraphics[ width=.12\linewidth,trim={6cm 2.5cm 5.3cm .5cm},clip]{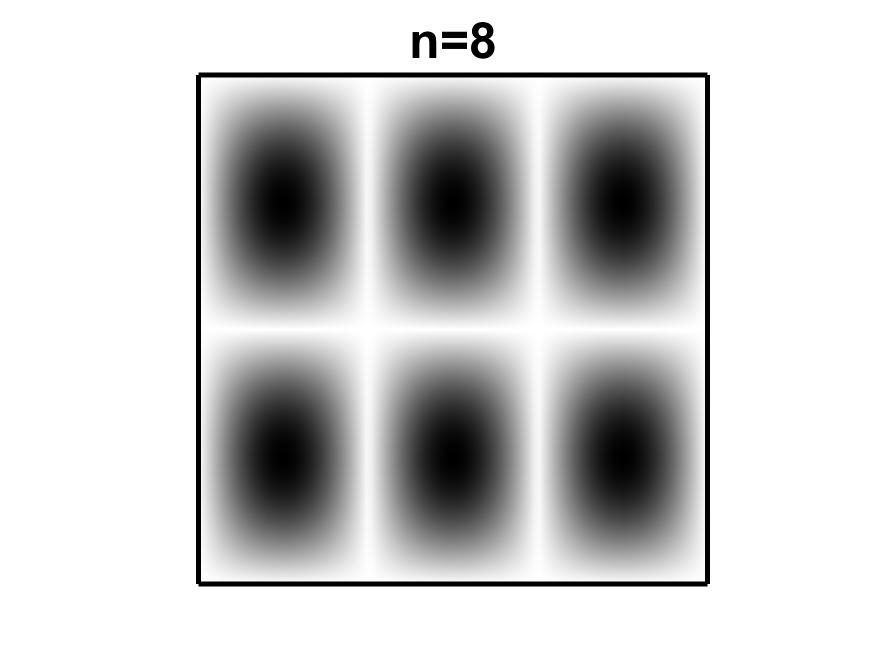}
\includegraphics[ width=.12\linewidth,trim={6cm 2.5cm 5.3cm .5cm},clip]{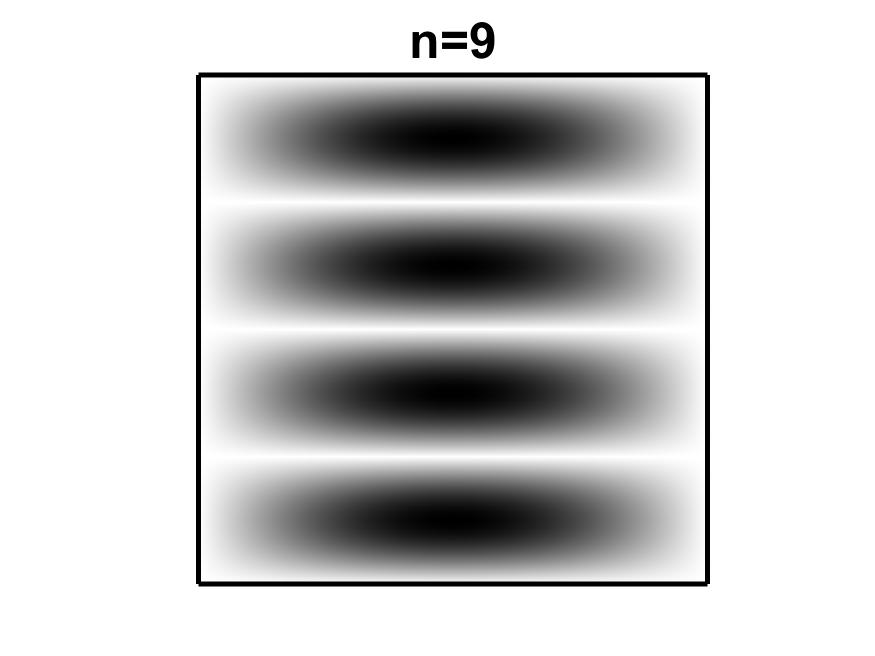}
\includegraphics[ width=.12\linewidth,trim={6cm 2.5cm 5.3cm .5cm},clip]{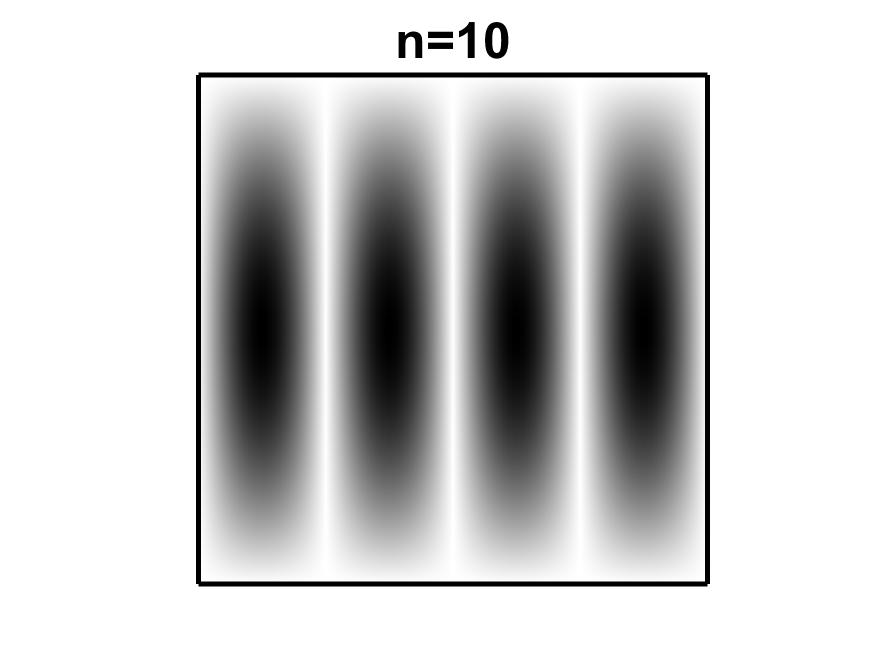}
\includegraphics[ width=.12\linewidth,trim={6cm 2.5cm 5.3cm .5cm},clip]{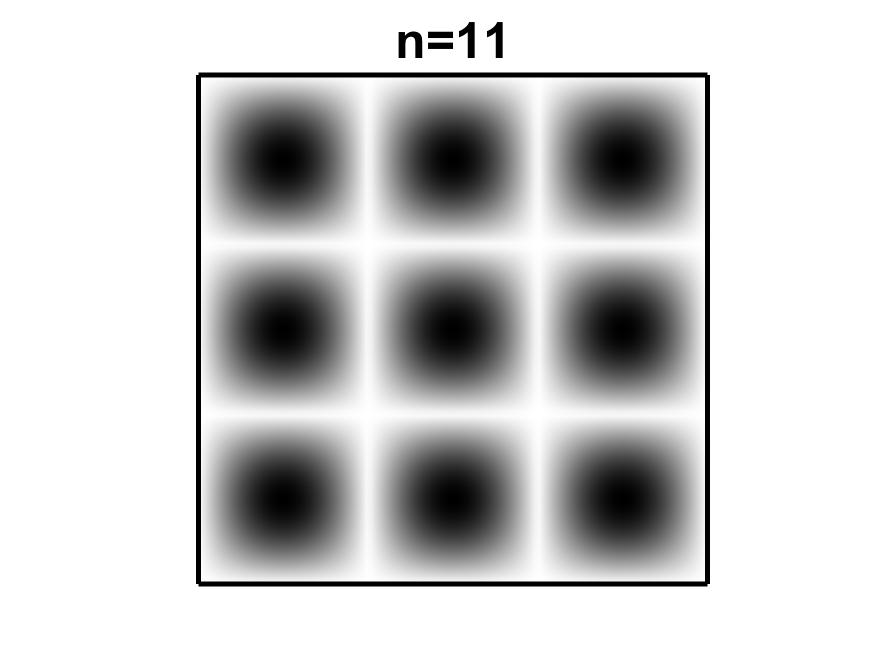}
\includegraphics[ width=.12\linewidth,trim={6cm 2.5cm 5.3cm .5cm},clip]{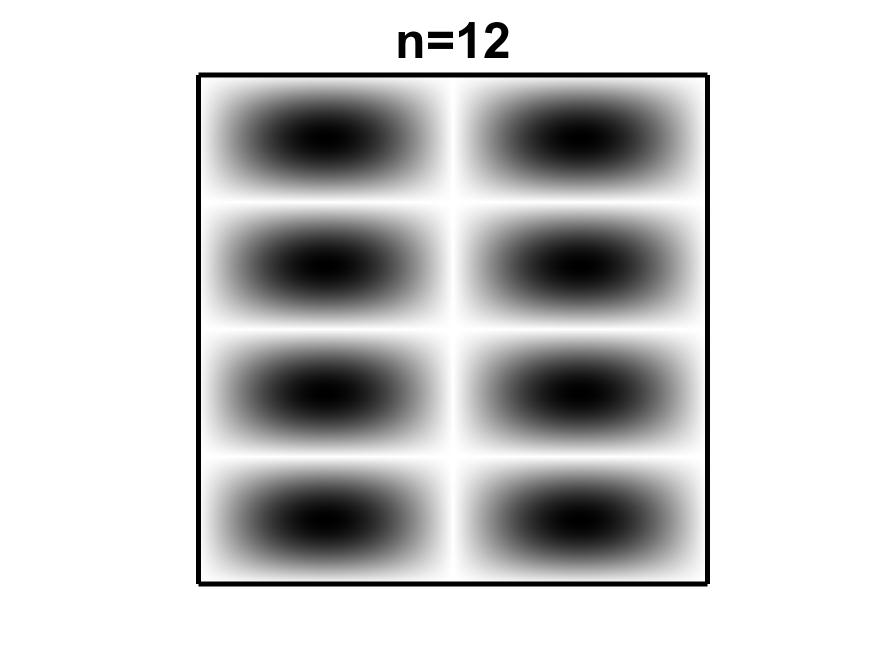}
\caption{The first 12 modes of the hexant or spherical square (above) and those of a planar square (below).}\label{fig:sphsq}
\end{figure}

\subsection{Periodic domains.}\label{sec:per}

The Schr\"odinger equation \eqref{eq:timeind} in periodic domains has important application to solving for Bloch states of electrons in a crystalline solid \cite{bloch}. Here we present an extension of the expansion method to computing eigenvalues and eigenfunctions of the Schrodinger and Helmholtz
equations for periodic domain.
\begin{definition}
A $d$-dimensional lattice $\Gamma_B$ is the set $\{\sum x_i {b}_i : x_i\in\mathbb{Z}\}$ with vectors $\{{b}_i\}\in \mathbb{R}^d$. Consequently, $\Gamma_B=\mathbb{Z}^dB$ for $B \in GL(d,\mathbb{R})$, the group of $d\times d$ invertible real matrices.
\end{definition}
\begin{definition}
The dual of a lattice $\Gamma_B$ is $\Gamma_B^* = \{{x}\in\mathbb{R}^d:\langle{x},{y}\rangle\in\mathbb{Z}, \forall {y}\in\Gamma_B\}$. Consequently, $\Gamma_B^*=\mathbb{Z}^dB^{-T}$ for $B \in GL(d,\mathbb{R})$.
\end{definition}
\begin{definition}
A $d$-dimensional flat torus $T_B$ is defined as the quotient space $T_B = {\mathbb{R}^d}/{\Gamma_B}$ for $B \in GL(d,\mathbb{R})$. 
\end{definition}

\begin{figure}[h!]\centering
\includegraphics[width=.3\linewidth]{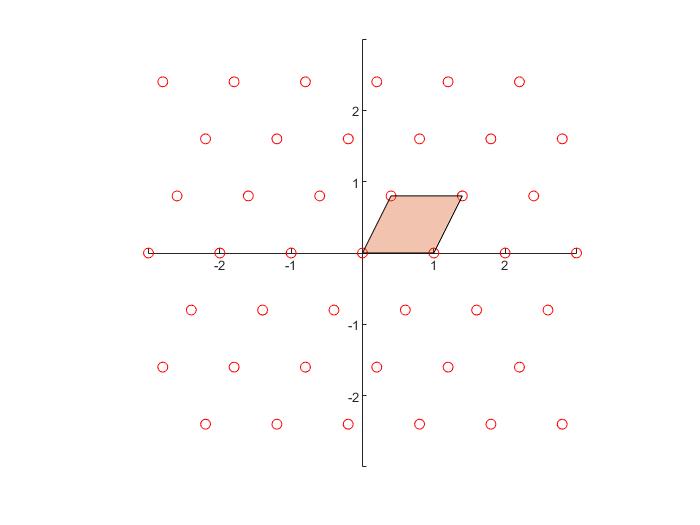}\hfil
\begin{tikzpicture}[baseline= -20pt,
scale=1.50,
line cap=round,
marr/.style={
  decoration={
    markings,
    mark=at position 0.5 with {#1}
    },
  postaction={decorate}
}
] 
\path
  coordinate (n1) at (0,0)
  coordinate (n2) at (.75,1.65)
  coordinate (n3) at (2,0)
  coordinate (n4) at (2.75,1.65);
\foreach \from/\to in {n3/n1,n4/n2}
    \draw[marr=\Singlearrow] (\from) -- (\to);
\foreach \from/\to in {n1/n2,n3/n4}
    \draw[marr=\Doublearrow] (\from) -- (\to);

\end{tikzpicture}
\includegraphics[width=.3\linewidth]{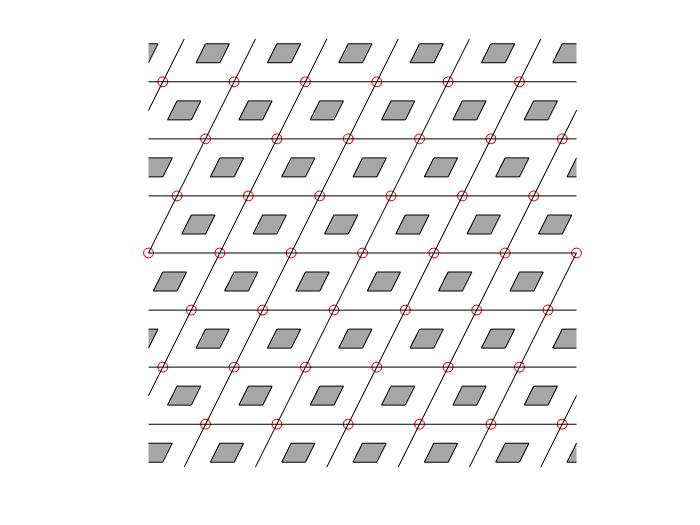}
\caption{Fundamental region of a lattice in $\mathbb{ R} ^2$ (left). A $2$-dimensional flat torus is formed by periodic boundary conditions over opposite edges of the fundamental region (middle). An example of a periodic domain---the dark regions indicate holes with Dirichlet boundary conditions where $V_0 \gg 1$ (right).} \label{fig:lattice1}
\end{figure}
For Euclidean space $\mathbb{R}^d$, eigenfunctions for $-\Delta$ on $\Omega=T_B$ are of the form $\phi_n({x})= \exp(2\pi i \langle {x},{w}\rangle)$ for $ w \in \Gamma^*_B$, and the eigenvalues are $-\Delta\phi_n({x})/\phi_n({x}) = 4\pi^2\|{w}\|^2$. We extend the expansion method to these domains using these eigenpairs and the potential $V( x) = V_0 \chi_{\Omega^c}( x)$. This allows us to compute eigenfunctions of domains with mixed Dirichlet and periodic boundaries. In \autoref{fig:lattice1}, we show an example of a periodic domain with a hole removed in each cell.

Just as in previous sections, the expansion method for a flat torus is given by the discretization of $\hat{H}$ where
\begin{equation}\label{flattexp}
    \begin{array}{ll}
H_{nm} & = \langle \phi_n,\hat{H} \phi_m\rangle\\
&= \lambda_n(T_B)\delta_{nm} +  V_0\mathlarger{\int}_{T_B \setminus \Omega}\phi_n^*\phi_m d x
    \end{array}
\end{equation}

and $\lambda_n(T_B)$ denotes the eigenvalues of the flat torus $T_B$ itself. The eigenvalues and eigenvectors of $H_{nm}$ are approximations of the eigenvalues and eigenfunctions of $\hat{H}$ (in the basis $\{\phi_n\}$). Figure~\ref{fig:per1} displays computed eigenmodes on a periodic domain with holes (domain shown in Figure~\ref{fig:boxtorusspace}). 

\begin{figure}[h!]
    \centering
    \includegraphics[width=.14\linewidth,trim={.5cm .5cm .8cm .8cm}]{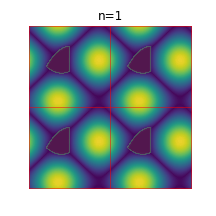}
    \includegraphics[width=.14\linewidth,trim={.5cm .5cm .8cm .8cm}]{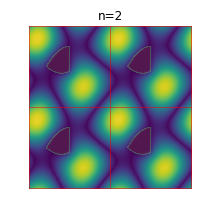}
    \includegraphics[width=.14\linewidth,trim={.5cm .5cm .8cm .8cm}]{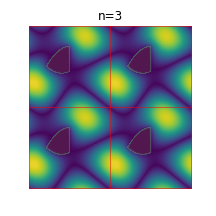}
    \includegraphics[width=.14\linewidth,trim={.5cm .5cm .8cm .8cm}]{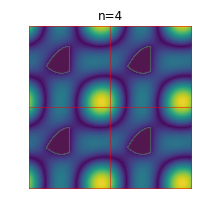}
    \includegraphics[width=.14\linewidth,trim={.5cm .5cm .8cm .8cm}]{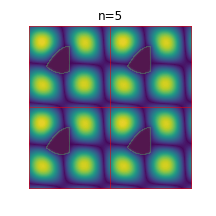}
    \includegraphics[width=.14\linewidth,trim={.5cm .5cm .8cm .8cm}]{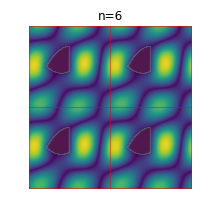}
    
    \includegraphics[width=.14\linewidth,trim={.5cm .5cm .8cm .8cm}]{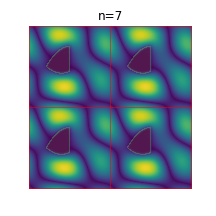}
    \includegraphics[width=.14\linewidth,trim={.5cm .5cm .8cm .8cm}]{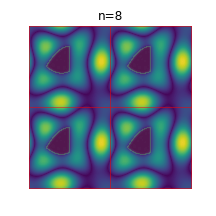}
    \includegraphics[width=.14\linewidth,trim={.5cm .5cm .8cm .8cm}]{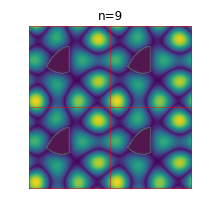}
    \includegraphics[width=.14\linewidth,trim={.5cm .5cm .8cm .8cm}]{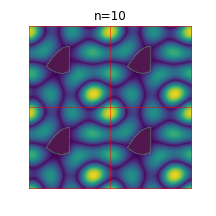}
    \includegraphics[width=.14\linewidth,trim={.5cm .5cm .8cm .8cm}]{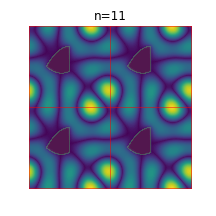}
    \includegraphics[width=.14\linewidth,trim={.5cm .5cm .8cm .8cm}]{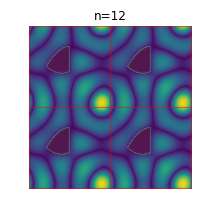}
    
    \caption{The absolute value of the first 12 computed eigenfunctions for a periodic domain (same domain shown in Figure~\ref{fig:boxtorusspace}). We impose periodic conditions on the boundary of the square cells and Dirichlet conditions along the boundary of the omitted shape. }
    \label{fig:per1}
\end{figure}
\subsection{Spectral Clustering and Billiard Trajectories}
Here, we give several examples of implementing the generalized expansion method to heuristically explore quantum signatures of chaos in classical billiards from computed eigenvalue statistics. This is similar to what is done in \cite{expmethod, amorebilliard}, but on manifolds. We use the following conjectures as foundations for the heuristic.
\begin{conj}\label{conj:BT}
(Berry-Tabor) The spectral value spacings of generic integrable systems coincide with those of uncorrelated random numbers from a Poisson process.
\end{conj}
\begin{conj}\label{conj:BGS}
(Bohigas-Gianonni-Schmit) The spectral value spacings of generic classically chaotic systems coincide with those of random matrices from the Gaussian Ensembles.
\end{conj}
\begin{figure}[h!]\centering
\includegraphics[width=.3\textwidth]{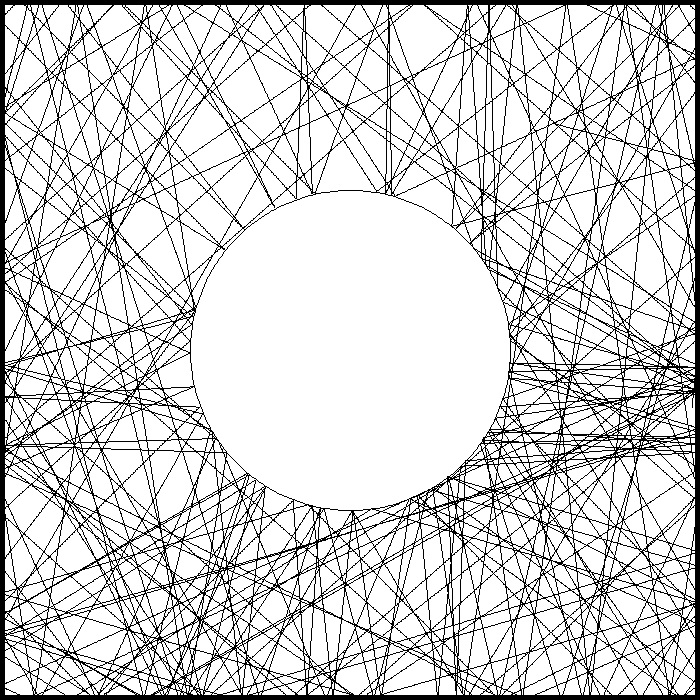} \includegraphics[width=.3\textwidth]{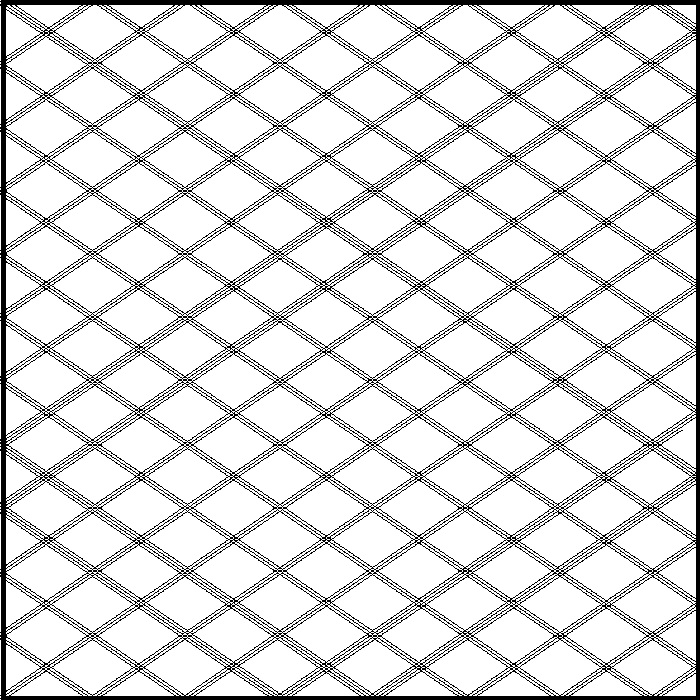}
  \caption{These billiards give simple examples of chaotic trajectories (left) and regular trajectories (right). The chaotic billiard here is the well-known Sinai Billiard.}
  \label{fig:billiards}
\end{figure}

It is understood that \textit{generic systems} have zero probability to have symmetries, and Hamiltonians with symmetries have degenerate states that are not relevant to the discussion of billiard dynamics \cite{quantbill}. Hence, in this discussion and in common practice, we desymmetrize the domains as much as possible before computing eigenvalues in order to ignore the symmetric modes, and we are left with billiards which we assume are sufficiently \textit{generic}. Further insight into these conjectures and their relation to random matrix theory can be found in \cite{rmt1,rmt2}.

Now, by considering the normalized distribution of the computed first $n$ spacings between consecutive Laplace--Beltrami eigenvalues of some desymmetrized region $\Omega$ with Dirichlet boundary conditions, we apply the Conjectures~\ref{conj:BT} and \ref{conj:BGS} as a heuristic to verify the trajectory type (see \autoref{fig:billiards}) of the following regions by comparing these distributions to the Poisson distribution $P_0(s) =e^{-s}$ and GOE distribution $P_{\operatorname{GOE}}(s)= \frac 1 2 \pi s e^{-\pi s^2/ 4}$.

\subsubsection*{\textbf{Planar billiards.}}
As the Sinai Billiard is a well-known classically chaotic billiard, we have used this domain as an example to perform the generalized expansion method and compare the resulting eigenvalue spacing distribution to the expected GOE distribution. As the equilateral triangle is a known integrable system, we expect its eigenvalue spacing distribution to coincide with a Poisson distribution. Using the method, we indeed arrive at these results and show them in \autoref{fig:sinai2}. We perform, when possible, a terminating sequence of desymmetrizations on the domains, as shown in \ref{fig:sinaibill}.

\begin{figure}[h!]\centering
\includegraphics[width=.2\textwidth]{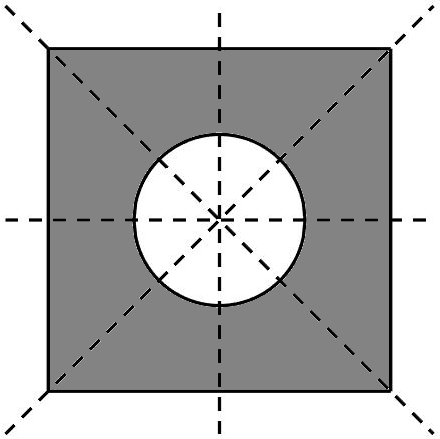}
\includegraphics[width=.2\textwidth]{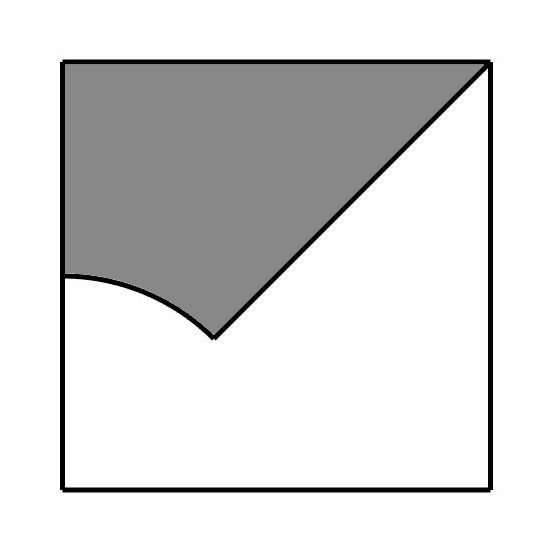}
\includegraphics[width=.2\textwidth]{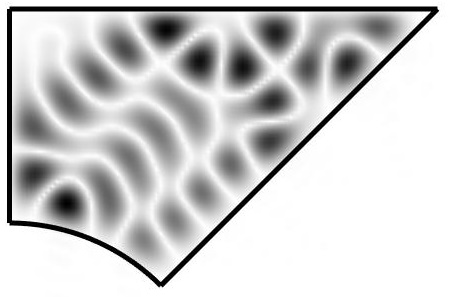}
  \caption{Symmetries of the Sinai billiard (left), desymmetrized domain embedded in a square (middle), and nodal lines of 40th eigenstate (right).}
  \label{fig:sinaibill}
\end{figure}

\begin{figure}[h!]\centering
\includegraphics[width=.4\textwidth]{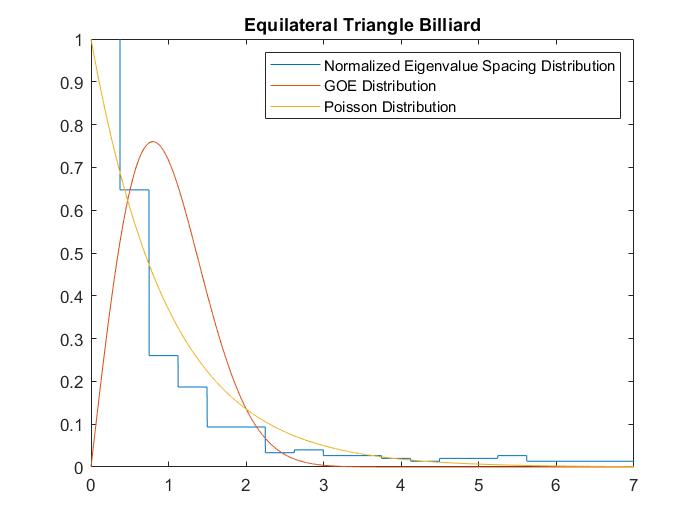}
\includegraphics[width=.4\textwidth]{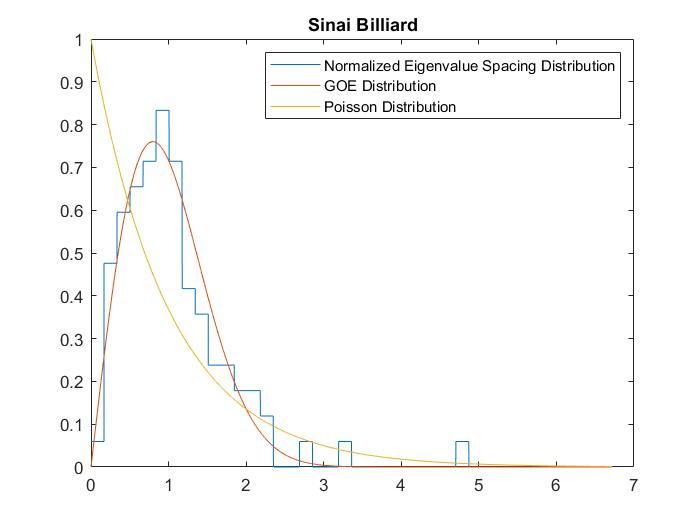}

  \caption{The eigenvalue spacings of an equilateral triangle (left) compared to those of the Sinai Billiard (right).}
  \label{fig:sinai2}
\end{figure}
\subsubsection*{\textbf{Spherical billiards.}}
Here, we consider spherical domains: octant (spherical equilateral triangle) and octant with a hole removed as shown in \autoref{fig:sphhist}. The former does not have a terminating series of desymmetrizations (so we leave it as is). However, we indeed can perform a desymmetrization on the latter, as shown. The eigenvalue distributions are compared in \autoref{fig:sphhist}. These eigenvalue statistics suggest the domains are regular and chaotic, respectively.

\begin{figure}[h!]\centering
\includegraphics[width=.24\textwidth]{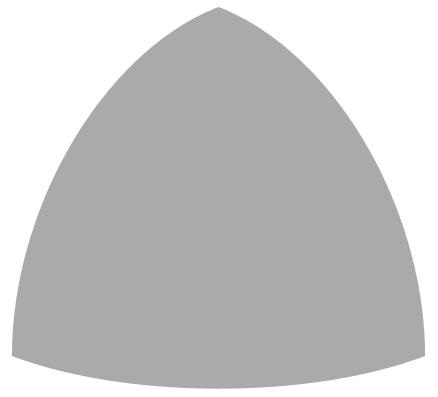}
\includegraphics[width=.24\textwidth]{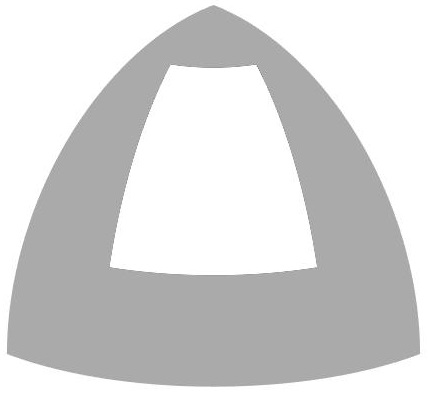}
\includegraphics[width=.24\textwidth, trim ={0cm 0cm -5cm 0cm}]{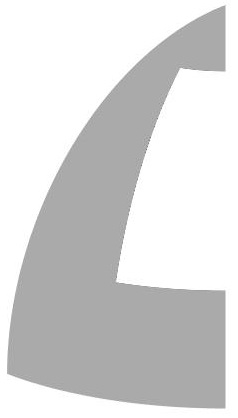} \\

\includegraphics[width=.4\textwidth]{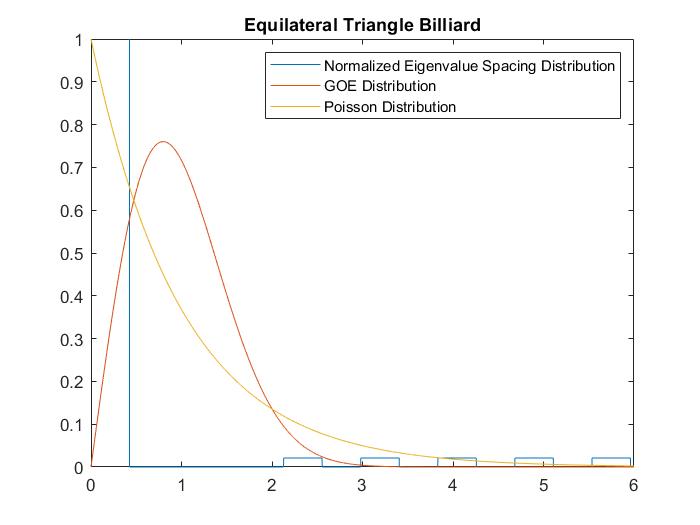}
\includegraphics[width=.4\textwidth]{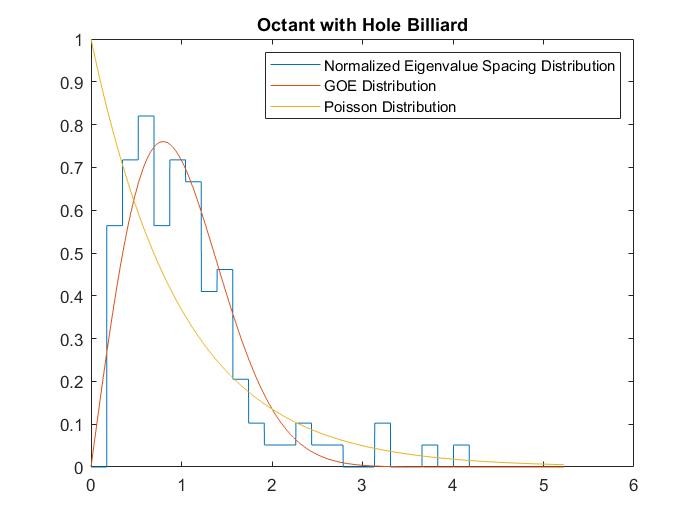}

  \caption{ Spherical octant (top left), octant with hole (top middle), and desymmetrized octant with hole (top right). Their respective eigenvalue spacings (bottom).}
  \label{fig:sphhist}
\end{figure}
\subsubsection*{\textbf{Periodic billiards.}}

In \autoref{fig:boxtorusspace} we illustrate and compute eigenvalue spacings for two periodic domains after necessary desymmetrizing. They both have eigenvalues with low clustering and appear to take on a Poisson distribution, indicating chaotic trajectories.

\begin{figure}[h!]\centering
\includegraphics[width=.3\linewidth, trim ={5.35cm 5.35cm 0cm 0cm},clip]{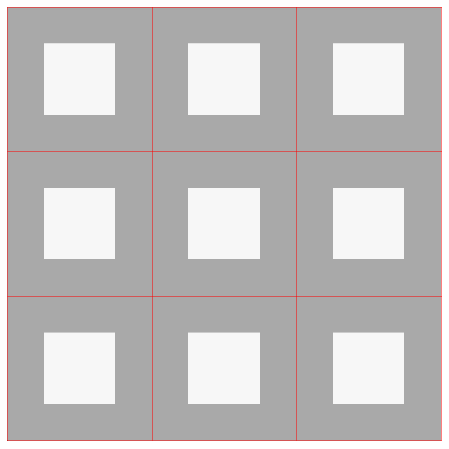} \hspace{1.5cm}
\includegraphics[width=.3\linewidth, trim ={5.35cm 5.35cm 0cm 0cm},clip]{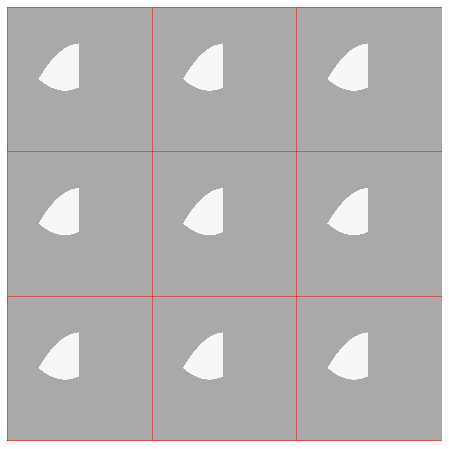} \\
\includegraphics[width=.4\linewidth]{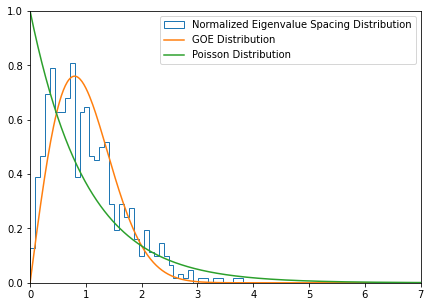}
\includegraphics[width=.4\linewidth]{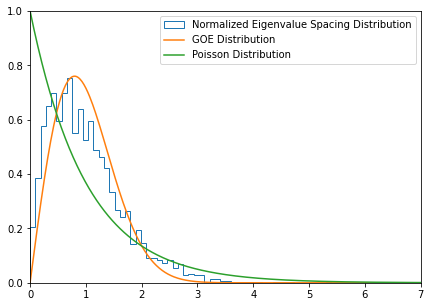}
\caption{Four cells of two periodic domains (top) and their respective histograms of the normalized Eigenvalue Spacing Distributions for the regions
obtained by desymmetrizing (bottom). The grey regions indicate $\Omega$, and we impose Dirichlet boundary conditions where they meet the white regions. The first twelve eigenfunctions for the asymmetric  domain are shown in Figure~\ref{fig:per1}. }\label{fig:boxtorusspace}
\end{figure}



\section{Conclusion}\label{sec:conclusion}
The Laplace--Beltrami operator is crucial to describing many physical phenomena on manifolds, and calculation of its eigenvalues is important to many applications involving non-Euclidean media. The generalized expansion method described in this paper provides a straightforward approach to discretize the Laplace--Beltrami operator to approximate its eigenmodes and eigenvalues. We provided proofs for its spectral convergence (Theorems~\ref{thm:V} and \ref{thm:M}) along with various analytic and numerical examples, including its application to studying billiard problems on surfaces. Notable applications for this method exist in nonlinear systems such as the study of Kerr media and Bose-Einstein condensates where one may approximate solutions by iterating on the ground state solution of the Schr\"odinger equation \cite{bronski2001bose,bao2004ground,bao2004computing}. Additionally, many applications exist in condensed matter physics. For example, as demonstrated in Section~\ref{sec:per}, this method can be used to solve for Bloch states of periodic domains defined on a lattice. Other applications include theories of 2D materials \cite{avouris20172d,paul2017computational}, superconductors \cite{berman2006superconducting}, and types of soft matter such as membranes \cite{meyer2007structure}.



\subsection*{Acknowledgement}
The authors acknowledge the support from Department of Mathematics at University of Utah where this project was initialized. Elena Cherkaev  acknowledges support from the U.S. National Science Foundation through grants DMS-1715680 and DMS-2111117. Dong Wang acknowledges the support from National Natural Science Foundation of China (NSFC) grant 12101524 and the University Development Fund from The Chinese University of Hong Kong, Shenzhen (UDF01001803). 

\printbibliography

@article{wang_2022,
  title={An efficient unconditionally stable method for dirichlet partitions in arbitrary domains},
  author={Wang, Dong},
  journal={SIAM Journal on Scientific Computing},
  volume={44},
  number={4},
  pages={A2061--A2088},
  year={2022},
  publisher={SIAM}
}

@article{conti2003optimal,
	author = {M. Conti and S. Terracini and G. Verzini},
	doi = {10.1016/s0022-1236(02)00105-2},
	journal = {Journal of Functional Analysis},
	number = {1},
	pages = {160--196},
	title = {An optimal partition problem related to nonlinear eigenvalues},
	volume = {198},
	year = {2003}}

@article{cybulski2008,
	author = {Cybulski, O and Holyst, R},
	doi = {10.1103/physreve.77.056101},
	journal = {Physical Review E},
	number = {5},
	pages = {56101},
	title = {Three-dimensional space partition based on the first {L}aplacian eigenvalues in cells},
	volume = {77},
	year = {2008}}

@article{osting2013minimal,
	author = {Osting, B. and White, C. D. and Oudet, \'{E}.},
	doi = {10.1137/130934568},
	journal = {SIAM J. Scientific Computing},
	number = 4,
	pages = {A1635--A1651},
	title = {{Minimal {D}irichlet energy partitions for graphs}},
	volume = 36,
	year = 2014}

@article{Wang_2019,
	author = {Dong Wang and Braxton Osting},
	doi = {10.1016/j.cam.2018.11.015},
	journal = {Journal of Computational and Applied Mathematics},
	month = {may},
	pages = {302--316},
	publisher = {Elsevier {BV}},
	title = {A diffusion generated method for computing Dirichlet partitions},
	volume = {351},
	year = 2019}

@article{bogosel2017efficient,
	author = {Beniamin Bogosel},
	doi = {10.1016/j.amc.2018.03.087},
	journal = {Applied Mathematics and Computation},
	month = {sep},
	pages = {61--75},
	publisher = {Elsevier {BV}},
	title = {Efficient algorithm for optimizing spectral partitions},
	volume = {333},
	year = 2018}

@article{bourdin2010optimal,
	author = {Bourdin, B. and Bucur, D. and Oudet, \'{E}.},
	doi = {10.1137/090747087},
	journal = {SIAM Journal on Scientific Computing},
	number = {6},
	pages = {4100--4114},
	title = {{Optimal partitions for eigenvalues}},
	volume = {31},
	year = {2010}}

@article{bao2004ground,
	author = {W. Bao},
	doi = {10.1137/030600209},
	journal = {Multiscale Modeling \& Simulation},
	number = {2},
	pages = {210--236},
	title = {Ground states and dynamics of multicomponent {B}ose--{E}instein condensates},
	volume = {2},
	year = {2004}}

@article{chang2004segregated,
	author = {S.-M. Chang and C.-S. Lin and T.-C. Lin and W.-W. Lin},
	doi = {10.1016/j.physd.2004.06.002},
	journal = {Physica D: Nonlinear Phenomena},
	number = 3,
	pages = {341--361},
	title = {Segregated nodal domains of two-dimensional multispecies {B}ose--{E}instein condensates},
	volume = 196,
	year = 2004}

@article{bao2004computing,
	author = {W. Bao and Q. Du},
	doi = {10.1137/s1064827503422956},
	journal = {SIAM Journal on Scientific Computing},
	number = {5},
	pages = {1674--1697},
	title = {Computing the ground state solution of {B}ose--{E}instein condensates by a normalized gradient flow},
	volume = {25},
	year = {2004}}

@article{expmethod,
  title={Expansion method for stationary states of quantum billiards},
  author={Kauffman, D and Kosztin, I and Schulten, K},
  journal={Am. J. Phys},
  volume={67},
  pages={133--141},
  year={1999}
}

@article{quantbill,
  title={Semiclassical mechanics of regular and irregular motion},
  author={Berry, Michael V},
  journal={Les Houches lecture series},
  volume={36},
  pages={171--271},
  year={1983},
  publisher={North-Holland Amsterdam}
}

@book{bloch,
  title={Introduction to Solid State Physics},
  author={Kittel, Charles and McEuen, Paul and McEuen, Paul},
  volume={8},
  year={1996},
  publisher={Wiley New York}
}

@book{extremum,
  title={Extremum Problems for Eigenvalues of Elliptic Operators},
  author={Henrot, Antoine},
  year={2006},
  publisher={Springer Science \& Business Media}
}

@article{diab,
  title={Diabolical points in the spectra of triangles},
  author={Berry, Michael Victor and Wilkinson, Mark},
  journal={Proceedings of the Royal Society of London. A. Mathematical and Physical Sciences},
  volume={392},
  number={1802},
  pages={15--43},
  year={1984},
  publisher={The Royal Society London}
}

@article{triangle,
  title={Hearing the shape of a triangle},
  author={Grieser, Daniel and Maronna, Svenja},
  journal={Notices of the AMS},
  volume={60},
  number={11},
  pages={1440--1447},
  year={2013}
}

@article{num1,
  title={On the numerical solution of elliptic partial differential equations on polygonal domains},
  author={Hoskins, Jeremy G and Rokhlin, Vladimir and Serkh, Kirill},
  journal={SIAM Journal on Scientific Computing},
  volume={41},
  number={4},
  pages={A2552--A2578},
  year={2019},
  publisher={SIAM}
}

@article{num2,
  title={Discrete Laplace--Beltrami operators for shape analysis and segmentation},
  author={Reuter, Martin and Biasotti, Silvia and Giorgi, Daniela and Patan{\`e}, Giuseppe and Spagnuolo, Michela},
  journal={Computers \& Graphics},
  volume={33},
  number={3},
  pages={381--390},
  year={2009},
  publisher={Elsevier}
}

@article{hyp1,
  title={Fundamental solution of the Laplacian in the hyperboloid model of hyperbolic geometry},
  author={Cohl, Howard S and Kalnins, Ernie G},
  journal={arXiv preprint arXiv:1201.4406},
  year={2012}
}

@article{num3,
  title={Eigenvalues of the Laplacian in two dimensions},
  author={Kuttler, James R and Sigillito, Vincent G},
  journal={Siam Review},
  volume={26},
  number={2},
  pages={163--193},
  year={1984},
  publisher={SIAM}
}

@article{rmt1,
  title={Distribution of the ratio of consecutive level spacings in random matrix ensembles},
  author={Atas, YY and Bogomolny, E and Giraud, O and Roux, G},
  journal={Physical review letters},
  volume={110},
  number={8},
  pages={084101},
  year={2013},
  publisher={APS}
}

@article{rmt2,
  title={Symmetry deduction from spectral fluctuations in complex quantum systems},
  author={Tekur, S Harshini and Santhanam, MS},
  journal={Physical Review Research},
  volume={2},
  number={3},
  pages={032063},
  year={2020},
  publisher={APS}
}

@article{trefethen2006computed,
  title={Computed eigenmodes of planar regions},
  author={Trefethen, Lloyd N and Betcke, Timo},
  journal={Contemporary Mathematics},
  volume={412},
  pages={297--314},
  year={2006},
  publisher={Providence, RI: American Mathematical Society}
}

@article{yuan2009bounds,
  title={Bounds to eigenvalues of the Laplacian on L-shaped domain by variational methods},
  author={Yuan, Quan and He, Zhiqing},
  journal={Journal of computational and applied mathematics},
  volume={233},
  number={4},
  pages={1083--1090},
  year={2009},
  publisher={Elsevier}
}

@article{fox1967approximations,
  title={Approximations and bounds for eigenvalues of elliptic operators},
  author={Fox, L and Henrici, P and Moler, C},
  journal={SIAM Journal on Numerical Analysis},
  volume={4},
  number={1},
  pages={89--102},
  year={1967},
  publisher={SIAM}
}

@book{teschl2009mathematical,
  title={Mathematical Methods in Quantum Mechanics},
  author={Teschl, Gerald},
  journal={Graduate Studies in Mathematics},
  volume={99},
  pages={106},
  year={2009},
  publisher={American Mathematical Society Providence, RI, USA}
}

@book{gilbarg1977elliptic,
  title={Elliptic Partial Differential Equations of Second Order},
  author={Gilbarg, David and Trudinger, Neil S and Gilbarg, David and Trudinger, NS},
  volume={224},
  number={2},
  year={1977},
  publisher={Springer}
}

@book{avouris20172d,
  title={2D Materials},
  author={Avouris, Phaedon and Heinz, Tony F and Low, Tony},
  year={2017},
  publisher={Cambridge University Press}
}

@article{paul2017computational,
  title={Computational methods for 2D materials: discovery, property characterization, and application design},
  author={Paul, JT and Singh, AK and Dong, Zheng and Zhuang, Houlong and Revard, BC and Rijal, B and Ashton, M and Linscheid, A and Blonsky, M and Gluhovic, D and others},
  journal={Journal of Physics: Condensed Matter},
  volume={29},
  number={47},
  pages={473001},
  year={2017},
  publisher={IOP Publishing}
}

@article{berman2006superconducting,
  title={Superconducting photonic crystals: Numerical calculations of the band structure},
  author={Berman, Oleg L and Lozovik, Yurii E and Eiderman, Sergey L and Coalson, Rob D},
  journal={Physical Review B},
  volume={74},
  number={9},
  pages={092505},
  year={2006},
  publisher={APS}
}

@article{bronski2001bose,
  title={Bose-Einstein condensates in standing waves: The cubic nonlinear Schr{\"o}dinger equation with a periodic potential},
  author={Bronski, Jared C and Carr, Lincoln D and Deconinck, Bernard and Kutz, J Nathan},
  journal={Physical Review Letters},
  volume={86},
  number={8},
  pages={1402},
  year={2001},
  publisher={APS}
}

@article{meyer2007structure,
  title={The structure of suspended graphene sheets},
  author={Meyer, Jannik C and Geim, Andre K and Katsnelson, Mikhail I and Novoselov, Konstantin S and Booth, Tim J and Roth, Siegmar},
  journal={Nature},
  volume={446},
  number={7131},
  pages={60--63},
  year={2007},
  publisher={Nature Publishing Group}
}

@article{num4,
  title={The generalized singular value decomposition and the method of particular solutions},
  author={Betcke, Timo},
  journal={SIAM Journal on Scientific Computing},
  volume={30},
  number={3},
  pages={1278--1295},
  year={2008},
  publisher={SIAM}
}

@book{num5,
  title={Dissipation in Deforming Chaotic Billiards},
  author={Barnett, Alexander Harvey},
  year={2000},
  publisher={Harvard University}
}

@article{num6,
  title={Solving the Helmholtz equation for membranes of arbitrary shape: numerical results},
  author={Amore, Paolo},
  journal={Journal of Physics A: Mathematical and Theoretical},
  volume={41},
  number={26},
  pages={265206},
  year={2008},
  publisher={IOP Publishing}
}

@article{amorebilliard,
  title={Spectroscopy of drums and quantum billiards: Perturbative and nonperturbative results},
  author={Amore, Paolo},
  journal={Journal of mathematical physics},
  volume={51},
  number={5},
  pages={052105},
  year={2010},
  publisher={American Institute of Physics}
}

@book{leveque1992numerical,
  title={Numerical Methods for Conservation Laws},
  author={LeVeque, Randall J},
  volume={214},
  year={1992},
  publisher={Springer}
}

@book{plasma1,
  title={The Zakharov System and Its Soliton Solutions},
  author={Guo, Boling and Gan, Zaihui and Kong, Linghai and Zhang, Jingjun},
  year={2016},
  publisher={Springer}
}

@article{plasma2,
  title={The nonlinear Schr{\"o}dinger equation and applications in Bose-Einstein condensation and plasma physics},
  author={Bao, Weizhu},
  journal={Dynamics in models of coarsening, coagulation, condensation and quantization},
  volume={9},
  pages={141--240},
  year={2007},
  publisher={World Scientific River Edge, NJ, USA}
}

@book{fibich2015nonlinear,
  title={The Nonlinear Schr{\"o}dinger Equation},
  author={Fibich, Gadi},
  volume={192},
  year={2015},
  publisher={Springer}
}

@article{zakharov1968stability,
  title={Stability of periodic waves of finite amplitude on the surface of a deep fluid},
  author={Zakharov, Vladimir E},
  journal={Journal of Applied Mechanics and Technical Physics},
  volume={9},
  number={2},
  pages={190--194},
  year={1968},
  publisher={Springer}
}

\end{document}